\documentclass[openany]{amsbook}
\usepackage{lmodern}
\usepackage{amsmath,amssymb}
\usepackage{amsthm}
\usepackage{amsfonts}
\usepackage{mathrsfs}
\usepackage{dutchcal}
\usepackage{parskip}
\usepackage{enumitem}
\usepackage{url}
\usepackage{graphicx}
\usepackage[export]{adjustbox}
\usepackage{float}
\usepackage{bbm}
\usepackage{etoolbox}
\usepackage{imakeidx}

\makeindex

\newtheorem{theorem}{Theorem}[section]

\newtheorem{lemma}{Lemma}
\newtheorem{definition}{Definition}
\newtheorem{corollary}{Corollary}
\DeclareSymbolFontAlphabet{\mathcalorig}   {symbols}

\begin{document}

\title{On Sharpest Tail Bounds for Functions \\ of Tail Bounded Random Variables}
\author{Stephen Jordan Harrison}

\maketitle

\begin{abstract}
Consider $n$ real/complex, independent/dependent random variables with respective tail bounds and $g$ a measurable function of the r.v.'s. 
Consider $f$ the ``sharpest" tail bound of $g$ (sharpest in the sense, if $f$ were any less, then for some $X_1,...,X_n$ satisfying the 
conditions, $g(X_1,...,X_n)$ would not satisfy the tail $f$). Significant research has been done to approximate $f$ often with high accuracy. 
These results are often of the form, for $g$ in this family, and tail bounds of $X_k$ in this family, $f$ is bounded by some $f'$ with high accuracy. 
However, the question ``what would it take to find $f$ exactly?" has received little attention, apparently even for simple cases. This is the question 
we try to answer. For $X_1,...,X_n$ required to be mutually ind., first $X_k$ are simplified to be monotone on $(0,1)$ WLOG. 
This strengthens convergence in distribution to convergence a.e. (Skorokhod's representation theorem), and allows defining shift operators, which help 
reduce the space of r.v.'s one searches to find $f$ and/or the maximum measure of a subset. We do find $f$ in some special cases, however $f$ rarely has 
a closed form. For $X_1,...,X_n$ dependent/not necessarily independent, another reduction in the space of r.v.'s one searches to find $f$ is done.
\end{abstract}

\tableofcontents

\chapter{Introduction}
\section{Preliminaries}\label{preliminaries}
We refer the reader to \cite{Dudley2002} for the general theory of measure theory.

Throughout this thesis, we will always assume for a given $n \in \mathbb{N}$ that $\mathbb{R}^n$ and any subset of $\mathbb{R}^n$ has the Borel $\sigma$-algebra. For 
$X$ a real r.v., the cumulative distribution function (CDF) of $X$ is $F_X := P(X \leq t), \ \forall t \in \mathbb{R}$. We will often refer to $P(X \geq t), \ \forall t \in \mathbb{R}$ as the reversed CDF 
of $X$. 

All CDFs can be shown to have the following properties \cite{Dudley2002} pg. 283: they are right continuous, monotonically increasing, approach $0$ as $t \rightarrow -\infty$, and approach $1$ as $t \rightarrow \infty$. In fact, these properties 
are also sufficient for there to exist a real r.v.\ with CDF the specified function. Similarly, all reversed CDFs are left continuous, 
monotonically decreasing, approach $1$ as $t \rightarrow - \infty$, and $0$ as $t \rightarrow \infty$. Furthermore, these properties are sufficient 
for there to exist a real r.v.\ with reversed CDF the specified function. 

The CDF of a real r.v.\ uniquely characterizes its distribution by the $\lambda$-$\pi$ theorem (see below), but the same is true of its reversed CDF. 

If $(X^k)_k$ is a sequence of real r.v.'s and $X$ is a real r.v.\ such that $F_{X^k}(t) \xrightarrow[k \rightarrow \infty]{} F_X(t)$ for all $t$ at which $F_X$ is continuous, one says 
$X^k$ converges to $X$ in distribution and writes 
$$X^k \xrightarrow[]{d} X.
$$ 
If $X, X^k: \Omega \rightarrow \mathbb{R}^n$ are r.v.'s for $k=1,2,...$ such that $X^k(w) \xrightarrow[k \rightarrow \infty]{} X(w)$ for 
all $\omega \in \Omega$ except a subset of measure of $0$, one says $X^k$ converges to $X$ almost everywhere and writes 
$$
X^k \xrightarrow[]{a.e.} X.
$$
A $\pi$-system on a set $\Omega$ is a set of subsets of $\Omega$ that is non-empty and closed under intersection. We are not concerned with the precise 
statement of ``Dynkin's $\lambda$-$\pi$ theorem" \cite{Dudley2002}, only the following immediate corollary: if $m_1, m_2$ are two probability measures on a measurable space $\Omega$ that agree on a $\pi$-system that generates the $\sigma$-algebra of $\Omega$, then $m_1 = m_2$. 

For a r.v.\ $X$ on a separable metric space $(S,d)$ with the Borel $\sigma$-algebra, a measurable subset $A$ of $S$ is a continuity set of $X$ iff\footnote{$\partial A$ is the topological boundary of $A$.} $P(X \in \partial A) = 0$ \cite[pg. 386]{Dudley2002}. Portmanteau's theorem \cite[pg. 386]{Dudley2002} is a theorem which, for 
r.v.'s $X^1,X^2,...$ also defined on $(S,d)$, gives several equivalent conditions for $X^m$ converging to $X$ in distribution. We only need the following two equivalences: 
\begin{enumerate}[label=(\roman*)]
   \item $X^m \xrightarrow[]{d} X$
   \item $P(X^m \in A) = P(X \in A)$ for all continuity sets $A$ of $X$
\end{enumerate} 
Specifically, suppose $(S,d)$ is $\mathbb{R}$ with the usual metric and (i) holds. Notice for all $t \in \mathbb{R}$ at which $P(X \geq t)$ is continuous, that $P(X \in \partial [t,\infty)) = 0$. Then by definition, $[0,\infty)$ is a continuity set, thus $P(X^m \geq t) \xrightarrow[m \rightarrow \infty]{} P(X \geq t)$ for all $t$ at which $P(X \geq t)$ is continuous. 

This provides a symmetric counterpart to convergence in distribution, which will be helpful in some proofs, as there may be one part that invokes convergence of the CDFs, and a symmetric part 
which invokes convergence of the reversed CDFs.

The following, Fubini's theorem, is from \cite{Dudley2002} pg. 139.
\begin{theorem}\label{Fubini'stheorem}
Let $\mathcalorig{F}_j$ be a $\sigma$-algebra for $\Omega_j$ with $\sigma$-finite measure $\mu_j$ for $j=1,...,n$. Then there is a unique measure $\mu$ on the product $\sigma$-algebra $\mathcalorig{F}$ of $\Omega=\Omega_1\times \cdots \times \Omega_n$
such that for any $A_j \in \mathcalorig{F}_j$, for $j=1,...,n$, $\mu(A_1 \times \cdots \times A_n) = \mu_1(A_1) \times \cdots \times \mu_n(A_n)$, or $0$ if any $\mu_j(A_j) = 0$, even if another is $+\infty$. If $f$ is non-negative and jointly measurable on $\Omega$, or if $f \in \mathcalorig{L}^1(\Omega,\mathcalorig{F},\mu)$, then 
$$
\int f \ d \mu = \int \cdots \int f(x_1,...,x_n) d\mu_1(x_1)\cdots d \mu_n(x_n)
$$
where for $f \in \mathcalorig{L}^1(\Omega,\mathcalorig{F},\mu)$, the iterated integral is defined recursively, ``from the outside" in the sense that for $\mu_n$-almost all $x_n$, the iterated integral with respect to the other variables is defined and finite, so that except on a set of $\mu_{n-1}$ measure $0$ (possibly depending on $x_n$) the iterated integral for the first $n-2$ variables is defined and finite, and so on. The same holds if integrations are done in any order.

\end{theorem} 
Our primary interest with this is constructing product r.v's. In particular, if $X_j:\Omega_j \rightarrow \mathbb{R}$ is a r.v.\ with measure $\mu_j$ on its domains, $j=1,...,n$, it allows constructing the 
r.v.\ $X=(X_1,...,X_n):\Omega_1 \times \cdots \times \Omega_n \rightarrow \mathbb{R}^n$ with the product measure $\mu$ on its domain. In particular, $X$ has mutually ind.\ components. 

\section{Introduction to tail bounds}\label{tailboundintroduction}

For a real r.v. \(X\) the most common form of a tail bound is \(P(|X| \ge t) \le f(t)\), \(t \ge 0\) where \(f(t) \searrow 0\). Let \(g : \mathbb{R}^n \to \mathbb{R}\) 
be measurable and let \(X_1, \dots, X_n\) be real r.v.'s with restrictions on them to prevent them from being arbitrarily large. This could be in the form of tail 
bounds on each r.v., restrictions of their moments/moment generating function, the r.v.'s being bounded, being explicitly specified, or a restriction on their 
mean/variance.

These restrictions are in some sense equivalent: one can convert from one type of restriction to another, but it is lossy; converting will generally result in a 
somewhat weaker restriction than the restriction one started with. This is seen by converting twice and observing your original restriction has been weakened. 
There is an exception: tail bounds are quite detailed and can often be converted to without any weakening of the restriction. They are the most detailed type of 
restriction, but may not be the easiest to work with.

The objective of most tail bound research is to consider a specific measurable function \(g : \mathbb{R}^n \to \mathbb{R}\), r.v.'s \(X_1, \dots, X_n\) with some 
form of restrictions on them, and place as small a tail bound\footnote{Also known as concentration inequalities.} on \(P(g(X_1, \dots, X_n) \ge t)\) or 
\(P(|g(X_1, \dots, X_n)| \ge t)\) as one can.\footnote{In this thesis we will only use tail bounds as our form of restriction.}

\cite{Vershynin2023} is a good introduction on this, helping to bridge the gap between graduate level and research level.

The following are results in the early chapters of \cite{Vershynin2023}.

\begin{theorem}[Hoeffding's inequality]
Let \(X_1, \dots, X_N\) be independent symmetric Bernoulli r.v.'s, and let \(a = (a_1, \dots, a_n) \in \mathbb{R}^N\). Then, for any \(t \ge 0\), we have
\[
P\left( \sum_{i=1}^N a_i X_i \ge t \right) \le \exp\left( -\frac{t^2}{2 \|a\|_2^2} \right).
\]
\end{theorem}

\begin{theorem}[Hoeffding's inequality, two-sided]
Let \(X_1, \dots, X_N\) be independent symmetric Bernoulli r.v.'s, and let \(a = (a_1, \dots, a_n) \in \mathbb{R}^N\). Then, for any \(t \ge 0\), we have
\[
P\left( \left| \sum_{i=1}^N a_i X_i \right| \ge t \right) \le 2 \exp\left( -\frac{t^2}{2 \|a\|_2^2} \right).
\]
\end{theorem}

The following definitions are given: if \(X_i\) is a real r.v. such that \(P(|X_i| > t) \le 2e^{-c t^2}\), \(\forall t \ge 0\)\footnote{This was the definition given by Roman Vershynin \cite{Vershynin2023}, however this thesis will only work with \(\leq\) and \(\geq\) in inequalities. This is done for practical reasons.} 
one says \(X_i\) is \emph{sub-gaussian}. In this case, the sub-gaussian norm of \(X_i\) is defined as \(\|X_i\|_{\psi_2} = \inf\{ t > 0 : \mathbb{E} \exp(X_i^2/t^2) \le 2 \}\) which can be shown to be a norm. The following is then shown.

\begin{theorem}[General Hoeffding inequality]
Let \(X_1, \dots, X_N\) be independent mean-zero sub-gaussian r.v.'s, and let \(a = (a_1, \dots, a_n) \in \mathbb{R}^N\). Then, for any \(t \ge 0\), we have
\[
P\left( \left| \sum_{i=1}^N a_i X_i \right| \ge t \right) \le 2 \exp\left( -\frac{c t^2}{K^2 \|a\|_2^2} \right)
\]
where \(K = \max_i \|X_i\|_{\psi_2}\).
\end{theorem}

\cite{Vershynin2023} includes much more complicated concentration inequalities, often involving matrices. That is, \(g(X_1, \dots, X_n)\) is often expressed using matrices. To my 
knowledge, current concentration inequality research does not prove that a tail bound is the absolute best (sharpest) tail bound possible. Sometimes proved is that 
certain constants in a tail bound function cannot be improved, although it is not shown the function is better than all other functions. After writing this 
dissertation it has become somewhat apparent to me why this is: producing ``best possible'' tail bounds is very effortful and likely won't have a closed form in 
terms of common functions.

The goal of this thesis was to investigate ``best possible'' (sharpest) tail bounds. To my knowledge, before this thesis a sharpest tail bound had never been proven. 
We write some sharpest tail bounds in non-closed forms, and one sharpest tail bound in a semi-closed form,\footnote{Using the error function. See \eqref{errortailbound}.} although in 
order to obtain that semi-closed form it was necessary to use very cooperative functions.

A significant amount of research has placed a good approximate tail bound on \(g(X_1, \dots, X_n)\) for particular forms of \(g\) and particular tail bounds on 
\(X_i\). Anna Skripka asked me to find tail bounds for Schur multipliers \cite{Skripka2024}, which are complicated functions I would describe as generalized matrix 
multiplication. That is, the \(g\) is complicated.

A tail bound \(f\) on \(g\) is called \emph{sharpest} iff (i) for all \(X_1, \dots, X_n\) satisfying their respective tail bounds (and independence conditions), 
\(g(X_1, \dots, X_n)\) satisfies \(f\) and (ii) were \(f\) any less, then for some \(X_1, \dots, X_n\) satisfying their tails, \(g(X_1, \dots, X_n)\) would not 
satisfy \(f\).

Before addressing Schur multipliers in particular, I questioned how one would begin going about computing the sharpest tail bound of \(g(X_1, \dots, X_n)\) for 
general \(g\) and general tail bounds placed on \(X_i\). This question has not received much attention and I was not successful finding previously developed methods 
which can solve the most simple cases of this question. This is the question we are interested in in this thesis. Our sharpest results are applicable to Schur 
multipliers in certain cases (see theorem~\ref{Schurmultipliertailbound} and \eqref{Schurmultipliereq2}), but the general case remains elusive.

We will develop tools which are theoretically capable of expressing sharpest tail bounds for any function 
$g: \mathbb{R}^n \rightarrow \mathbb{R}^n $ of real mutually ind.\ r.v.’s, but as of now only in a few 
cases are there methods of actually obtaining this expressing. I have no real idea how to proceed further.

\section{Different types of tail bounds}\label{tailtypes}

Let $X_k$ be real r.v.'s, $k=1,...,n$. There are multiple types of tail bounds one can place on $X_k$: 
\begin{enumerate}[label=(\roman*)]
\item $P(|X_k| \geq t) \leq f_k(t),  \forall t\geq 0$
\item $P(X_k \geq t) \leq g_k^+(t), \forall t \in \mathbb{R}$
\item $P(X_k \leq t) \leq g_k^-(t), \forall t \in \mathbb{R}$
\item $P(X_k \geq t) \leq h_k^+(t), \forall t \in \mathbb{R}$ and $P(X_k \leq t) \leq h_k^-(t), \forall t \in \mathbb{R}$ 
\end{enumerate}
where $f_k(t),g_k^+(t),g_k^-(-t),h_k^+(t),h_k^-(-t) \rightarrow 0$ as $t \rightarrow \infty$. WLOG assume these 
functions are monotone and have range $[0,1]$.

These types of tail bounds will be called absolute, right, left, and 2 tail bounds respectively, whereas 
``tail bound" alone means any one of these types. (i), (iv) are convertible to each other, but non-equivalent: given (iv) 
the most we can infer is 
$$
P(|X_k| \geq t) \leq \min \{1, h_i^+(t)+h_i^-(-t)\}, \forall t \geq 0.
$$
Whereas, given (i) the most we can infer is 
\begin{align*} 
   P(X_k \geq t) &\leq f_k(t) , \ \  \forall t \geq 0\\ 
   P(X_k \leq t) &\leq f_k(-t), \forall t \leq 0.
\end{align*}
In particular, transferring type (i) to type (iv) and back is weaker than the original tail bound. Similarly for transferring type (iv) to type (i) and back. 

Throughout we will assume that tails like $f_k, g_k^+,h_k+$ are left continuous and 
tails like $g_k^-,h_k^-$ are right continuous. This is because, for example, given $P(X_k \geq t) \leq f_k(t)$, in this thesis we will often 
want to consider a r.v.\ $\widetilde{X}_k$ characterised by 
$$
P(\widetilde{X}_k \geq t) = f_k(t), \ \forall t \in \mathbb{R}. 
$$
That is, the (reversed) cumulative distribution of $\widetilde{X}_k$ equals $f_k(t)$. In particular, for such an $\widetilde{X}_k$ to exist 
requires $f_k$ is left continuous by left continuity of the reversed CDF. Also, assuming left continuity (right continuity for left tail bounds) will 
significantly reduce the number of limits in proofs.

In (iv) it is prudent to assume $h_k^+(0) = h_k^-(0) = 1$. This is due to the following. Suppose
$$
\exists t',\ h_k^-(t'), h_k^+(t') <1. 
$$
WLOG $h_k^-(t') \leq h_k^+(t')$. Suppose $X_k$ satisfies (iv). It seems natural to let $X_k$ have $h_k^-(t')$ mass at $t'$, 
however this leads to a contradiction. Indeed, since $P(X_k \leq t') \leq h_k^-(t')$ and $X_k$ has $h_k^-(t')$ mass at $t'$, it follows $P(X_k < t') = 0$. Thus
\begin{equation*}
   1 = P(X < t') + P(X \geq t') \leq 0+ h^+(t') < 1.
\end{equation*}
To avoid this easy contradiction, it is sufficient to assume $\exists t_0 \in \mathbb{R}$ 
such that $h_k^+(t_0) = h_k^-(t_0) = 1$. By translating, WLOG $t_0 = 0$. 

It should be noted the logic I am using here is not one of a mathematical proof; there are multiple possible definitions. I am providing a logical explanation for why the definition would be 
very clumsy if it were not done the way I am explaining. In particular if you find both $h_k^-(t')<1$ and $h_k^+(t')<1$ then you can get an easy contradiction. The definition I chose is sufficient to exclude 
the possibility of this meaningless contradiction.

\section{Examples of simple cases}\label{sectionsimplecases}

Learning how to find sharpest tail bounds is equivalent to learning how to maximize the mass\footnote{Throughout this thesis
the measure of a set will often be called the mass of a set. I prefer this view, because I feel it accurately portrays the 
definition of a $\sigma$-algebra with a measure (also a ring of sets with a pre-measure).} (measure) 
of subsets in $\mathbb{R}^n$ over $X_1,...,X_n$ satisfying the conditions. That is, 
\begin{align}\label{eg1sup}
   \sup_{X_k \text{ tail bounded}} P(g(X_1,...,X_n) \geq t) = \sup_{X_k \text{ tail bounded}} &P((X_1,...,X_n)  \\ 
                                                                                              &\in g^{-1}[t,\infty)) \nonumber
\end{align}

One might further require $X_k$ are mutually ind.\ in \eqref{eg1sup}. 

\textbf{Example} $1$: 

The simplest non-trivial case of \eqref{eg1sup}
is $n=2$ and $V := g^{-1}[-t,\infty) = \{p_1,p_2\}$ where $p_1 = (a_1,b_1), \ p_2 = (a_2,b_2)$, with $0 \leq a_1 < a_2$ and $0 \leq b_2 < b_1$. 

\begin{figure}[h!]
   \includegraphics[scale = .18, center]{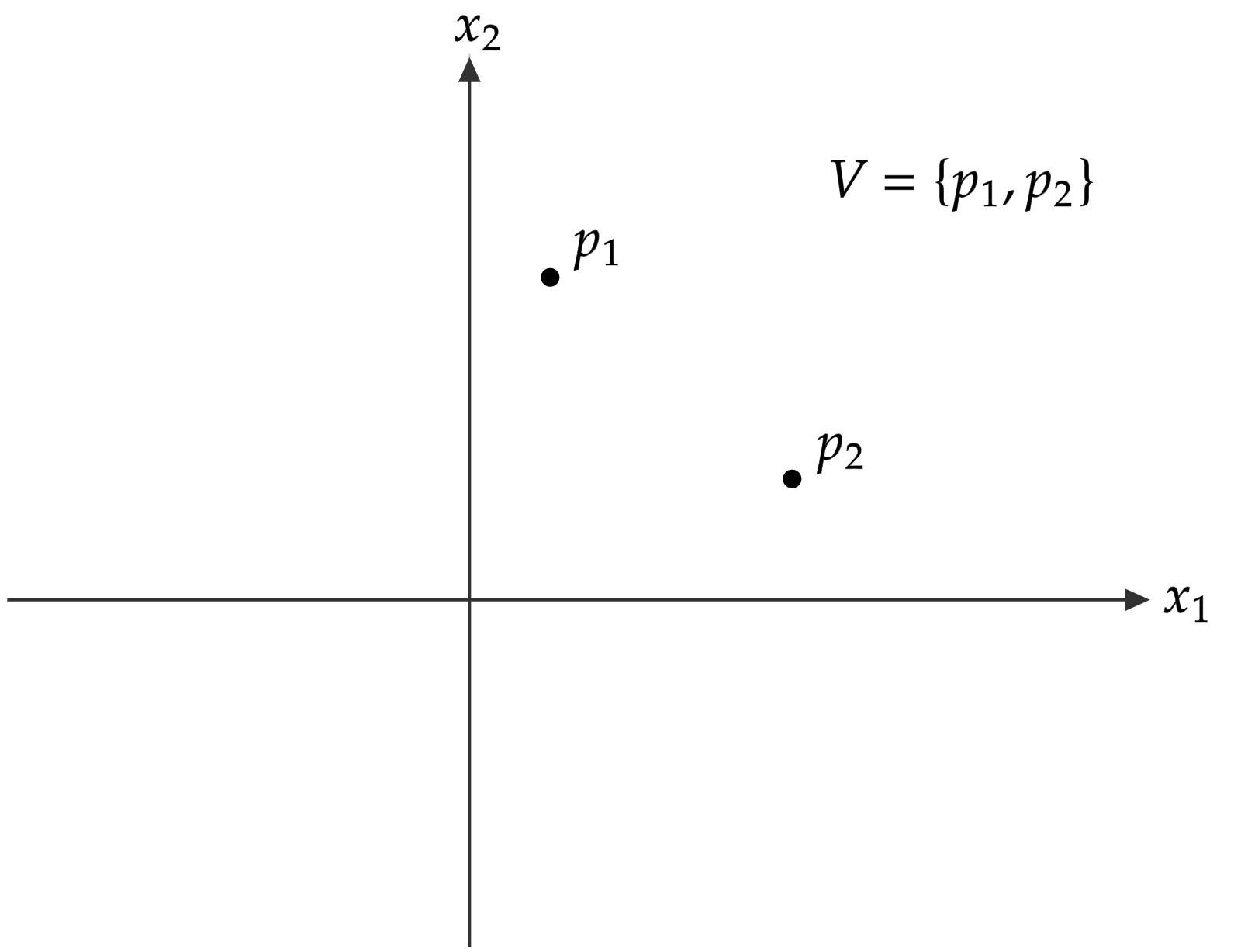}
   \caption{Example $1$}\label{eg.1fig}
\end{figure}

When $X_1,X_2$ are ind., solving \eqref{eg1sup} in terms of $a_1,a_2,b_1,b_2$ and the tail bounds is not easy.\footnote{It seems to be a non-concave/non-convex non-linear program.}
It can be brute forced with elementary calculus, but solving \eqref{eg1sup} this way for $n=2, \ X_1,X_2$ ind., and 
\begin{align*} 
V:= g^{-1}[t,\infty) = &\{(a_1,b_1),(a_2,b_2),(a_3,b_3)\},
\end{align*}
with $0 \leq a_1 < a_2 < a_3, \ 0 \leq b_3 < b_2 < b_1$, is exponentially more difficult. 
\textit{Shift operators} will be introduced later, which will reduce solving \eqref{eg1sup} for $X_1,...,X_n$ mutually 
ind.\ and $g^{-1}[t,\infty) = \{\text{finite points}\} \subset \mathbb{R}^n$ to a straightforward finite computation in section~\ref{solvingeg.1section}.

\textbf{Example} $2$: Another basic case is suppose $X_i$ are mutually ind., right tail bounded (section \ref{tailtypes} (ii)) by $f_i(t),\  i=1,...,n$,
and 
$$
g:(x_1,...,x_n) \mapsto x_1+\cdots + x_n.
$$
Given this, what is the best $f:\mathbb{R} \rightarrow [0,1]$ such that $P(X_1 + \cdots + X_n \geq t) \leq f(t)$? First introduce
mutually ind. $\widetilde{X}_1,...,\widetilde{X}_n$ such that\footnote{By hypothesis, in section $\ref{tailtypes}$,
right tails will satisfy the properties of (reversed) CDFs, which is known to be sufficient for $\widetilde{X}_k$ to exist.}
\begin{equation}\label{eg2Xtilde}
P(\widetilde{X}_i \geq t) = f_i(t), i=1,...,n.
\end{equation}

$f_i$ satisfy the properties of a (reversed) CDF. $\widetilde{X}_i$ are $X_i$ made as large as the tails allow, so apparently
\begin{equation}\label{eg2eq}
P(X_1+\cdots + X_n \geq t) \leq P(\widetilde{X}_1+\cdots + \widetilde{X}_n \geq t), \forall t \in \mathbb{R}.
\end{equation}
On the other hand $\widetilde{X}_i$ satisfy the bounds, so one cannot do better than the RHS of \eqref{eg2eq} 
for a tail bound on $P(X_1+\cdots+X_n \geq t)$. However, there is no immediately clear way to prove \eqref{eg2eq}. 

To prove \eqref{eg2eq}, \textit{neat} r.v.'s will be introduced: 
\begin{definition}\label{neatrvs}
A r.v. $X$ is \textit{neat} iff $X:(0,1)\rightarrow \mathbb{R}$ and $X \nearrow$. $\mathcal{N}$ is the set of neat r.v.'s.
\end{definition}

To understand why this helps, suppose in \eqref{eg2eq} every $X_i, \widetilde{X}_i$ is neat. It will be 
shown in the next section, in this case that
$$
X_i(s) \leq \widetilde{X}_i(s), \ \forall s \in (0,1).
$$
It follows $\forall (s_1,...,s_n) \in (0,1)^n$, 
$$
X_1(s_1) + \cdots + X_n(s_n) \geq t \implies \widetilde{X}_1(s_1) + \cdots + \widetilde{X}_n(s_n) \geq t.
$$
Therefore
\begin{align*}
& \{(X_1,...,X_n) \in g^{-1}[t,\infty)\} \subset \{(\widetilde{X}_1,...,\widetilde{X}_n) \in g^{-1}[t,\infty)\} \\
\implies & m \{(X_1,...,X_n) \in g^{-1}[t,\infty)\} \leq m \{(\widetilde{X}_1,...,\widetilde{X}_n) \in g^{-1}[t,\infty)\}.
\end{align*}
which is \eqref{eg2eq}. 

Chapter $2$ will primarily be results about neat r.v.'s, which are needed for shift operators, relevant to $X_1,...,X_n$ mutually ind.\ and tail bounded (of any type). Because shift operators are much more complicated than 
the case where $X_1,...X_n$ are dependent and tail bounded, dependent r.v.'s will be addressed in chapter $3$, before shift operators in chapter $4$. 

Results about the sharpest tail bound for dependent r.v.'s is applicable to the mutually ind.\ case: in the rigorous set theoretic formulation of probability, every probability variable has a measure associated to it. 
When you consider random variables which are n-tuples of multiple random variables, there is one measure assigned to all of the tuples containing the information of all n of the variables. 
Random variables which are independent have a special characteristic on this one measure. 
Those which are not independent don't necessarily have this characteristic. Thus when making proofs about dependent r.v.'s, 
you are considering all possible cases of the measures, which includes when the measure has the independent property. 
Such an application will no longer provide the sharpest tail bound for the mutually ind.\ case, but 
the tail bound often becomes easier to work with (by considering the dependent case) and I hypothesize it should still 
be asymptotically correct to the sharpest tail bound for the mutually ind.\ case. 

\chapter{Mutually independent real random variables}
In this chapter we will show if $X: \Omega \rightarrow \mathbb{R}$ is a r.v., WLOG $X: (0,1) \rightarrow \mathbb{R}$ and $X \nearrow$. This domain simplification is generalizable to $X: \Omega \rightarrow \mathbb{R}^n$ 
when $X$ has mutually ind.\ components. We prove simplifying the domains of the r.v.'s like this strengthens convergence in distribution to convergence almost 
everywhere.\footnote{(Footnote edited after dissertation submission which originally only mentioned the possibility of the following generalization and did not mention 
Sklar's theorem or have a reference to). Because of this, I tried to find a canonical way to simplify the domain of a general r.v.\ $X: \Omega \rightarrow \mathbb{R}^n$ 
(which does not have mutually ind.\ components) thereby somehow strengthening convergence in distribution. This was made possible by Sklar's theorem \cite{Sklar}, but it 
resulted in a new type of convergence that isn't any of the conventional types of convergence. This result will be in a followup paper.} (The following comment and footnote about Skorokhod's theorem
was added after the university publication, which did not mention Skorokhod's theorem or have a reference of it). This is an example of 
Skorokhod's representation theorem\footnote{Unfortunately, I didn't know Skorokhod's representation theorem while writing this dissertation. The result mentioned in 
the previous footnote was the natural generalization of the view taken in this dissertation, and turned out to be different from Skorokhod's representation theorem.} \cite{Skorokhod} pg. 70. 
Skorokhod's representation theorem states (roughly) if $X^n$ converges to $X$ in distribution, then by changing the domains of the random variables to a common domain, 
its possible to get $X^n$ converges to $X$ a.e.. 

It is also shown why this domain simplification helps with sharpest tail bounds.

\section{Neat r.v.'s}
Through this thesis, we will always assume the Borel $\sigma$-algebra is the $\sigma$-algebra being used. $(0,1)$,
as the domain of neat r.v.'s, will have the Lebesgue measure, but on the Borel $\sigma$-algebra. This is because
r.v.'s tend to work more easily with the Borel $\sigma$-algebra. 

The next three lemmas discuss the uniqueness/existence of neat r.v.'s (definition \ref{neatrvs}).

\begin{lemma}[uniqueness of neat r.v.'s]\label{uniquenessneat}
Let $X,Y \in \mathcal{N}$ (definition~\ref{neatrvs}). 
\begin{enumerate}[label=(\roman*)]
\item If $X, Y \in \mathcal{N}$ are both left continuous, or both right continuous, and $F_X = F_Y$, then $\forall s \in(0,1), \ X(s) = Y(s)$
\item The function $X'(s) = \lim_{s' \nearrow s} X(s'), \ \forall s \in (0,1)$ is neat, left continuous, and has the same distribution as $X$
\item The function $X'(s) = \lim_{s' \searrow s} X(s'), \ \forall s \in (0,1)$ is neat, right continuous, and has the same distribution as $X$
\end{enumerate}
\end{lemma}

\begin{proof}
(i): Suppose $X,Y$ are both left continuous. Fix $s \in (0,1)$. Since $F_X = F_Y$ and ($X \nearrow$ by hypothesis) $\{X \leq X(s)\} \supset (0,s)$, 
\begin{align*}
P(Y \leq X(s)) &= P(X \leq X(s)) \\ 
&\geq m(0,s) = s.
\end{align*}
It is necessary
$$
\forall s' \in (0,s), \ Y(s') < X(s).
$$ 
To show this, suppose some $s'$ made this false. One would find (since $Y \nearrow$), $s \leq P(Y \leq X(s)) = m(Y^{-1}(X(s))) \leq m(0,s') = s'$. 
A contradiction. It follows, by left continuity of $Y$, 
$$
X(s) \geq \lim_{s' \nearrow s} Y(s') = Y(s).
$$
By symmetry, $Y(s) \geq X(s)$. Thus $X(s) = Y(s)$ as desired. 

The case where $X,Y$ are both right continuous is treated symmetrically.

(ii): We have
\begin{equation}\label{leftcontwlog}
X'(s) = \lim_{s' \nearrow s} X(s'), \ \forall s \in(0,1).
\end{equation}
$X'$ is neat: $X'$ is well defined since $X \nearrow$ (definition~\ref{neatrvs}). Notice $X' \nearrow$ on $(0,1)$. 
Because monotone functions are measurable, $X'$ is measurable. This shows $X' \in \mathcal{N}$. 

$X'$ is left continuous: if $s' < s'' < s$, then
$$
X(s') \leq X'(s'') \leq X'(s).
$$
Take $s', s'' \rightarrow s$ and use $X(s') \rightarrow X'(s)$ as $s' \nearrow s$ to conclude $X'$ is left continuous.
By \eqref{leftcontwlog}, if $X$ is continuous at $s$, then $X(s) = X'(s)$. But $X\nearrow$, thus has countably many discontinuities, 
hence $X'=X$ a.e.. It is now routine to show $X'$ and $X$ have the same distribution. 

(iii): This proof is symmetric to (ii).
\end{proof}
We will construct $\widetilde{X}$ satisfying \eqref{eg2Xtilde}.
\begin{lemma}[$\widetilde{X}$ lemma]\label{Xtilde}
      Let $f \searrow 0$ on $\mathbb{R}, \ f(t) \rightarrow 1$ as $t \rightarrow -\infty$, and $f$ be left continuous. 
      Define
      \begin{equation}\label{Xtildedefn}
      \widetilde{X}(s) := \sup f^{-1}[1-s,1], \ \forall s \in (0,1).
      \end{equation}
      Then $\widetilde{X}$ has the following properties.
      \begin{enumerate}[label = (\roman*)]
      \item $\widetilde{X}$ is a neat r.v.
      \item $\widetilde{X}$ is right continuous and $P(\widetilde{X} \geq t) = f(t), \ \forall t \in \mathbb{R}$
      \item If $X \in \mathcal{N}$ and $P(X \geq t) \leq f(t), \ \forall t \in \mathbb{R}$, then $$X(s) \leq \widetilde{X}(s), \ \forall s \in (0,1)$$
      \end{enumerate}
\end{lemma}
\index{$\widetilde{X}$} Comment: in the sense of (ii) and (iii), $\widetilde{X}$ is the largest possible neat r.v.\ which satisfies $f$.
\begin{proof}
(i): First we must show $\widetilde{X}$ is finite. $\forall s \in (0,1), \ f^{-1}[1-s,1] \neq \emptyset$ because 
$f \searrow$ and $\lim_{t \rightarrow -\infty} f(t) = 1$ by hypothesis. On the other hand, because $f \searrow 0, \ \forall s \in (0,1),  
\ \exists t \in \mathbb{R}$ such that $t' \geq t \implies f(t') < 1-s$. In particular, $(t,\infty) \cap f^{-1}[1-s,1] = \emptyset$. So 
$\emptyset \neq f^{-1}[1-s,1]$ is bounded above. So $\sup f^{-1}[1-s,1]$ is finite as desired.

We must show $\widetilde{X} \nearrow$. 
\begin{equation}\label{Xtilde(i)eq}
f^{-1}[1-s,1] \subset f^{-1}[1-s',1], \ \forall s,s', 0 < s \leq s' < 1. 
\end{equation}
Taking the supremum, $\widetilde{X}(s) \leq \widetilde{X}(s'), \ \forall s,s', 0 < s \leq s' < 1$. Since monotone functions are measurable, $\widetilde{X}$ is measurable.

(ii): Fix $s \in (0,1)$. We claim 
\begin{align}\label{Xtilde(ii)eq}
f^{-1}[1-s,1] &= \bigcap_{s<s'<1} f^{-1}[1-s',1] \\ 
&\stackrel{\eqref{Xtilde(i)eq}}{=} \lim_{s' \searrow s} f^{-1}[1-s',1].
\end{align}
LHS, RHS will refer to this equation. LHS $\subset$ RHS immediately by \eqref{Xtilde(i)eq}. We must show LHS $\supset$ RHS. If 
$t \in $ RHS then 
$$
f(t) \in [1-s',1], \ \forall s' \in (s,1).
$$
This implies $f(t) \in [1-s,1]$ which implies $ t \in f^{-1}[1-s,1]$. Hence LHS $\supset$ RHS. \eqref{Xtilde(ii)eq} holds as 
desired. 

Right continuity of $\widetilde{X}(s) = \sup f^{-1}[1-s,1]$ follows from \eqref{Xtilde(ii)eq}: If $t'<t \in f^{-1}[1-s,1]$ then $f(t') \geq f(t)$ by hypothesis, i.e.\ $t' \in f^{-1}[1-s,1]$. This says 
$f^{-1}[1-s,1]$ is an interval with endpoints $-\infty$ and $\sup f^{-1}[1-s,1]$. Furthermore, it is a closed interval: suppose $(t_n)_n \subset f^{-1}[1-s,1]$ such that 
$$
t_n \nearrow \sup f^{-1}[1-s,1].
$$ 
So $(f(t_n))_n \subset [1-s,1]$. By hypothesis $f$ is left continuous, hence 
\begin{equation}\label{Xtildeclosedinterval}
f(\sup f^{-1}[1-s,1]) = \lim _{n \rightarrow \infty} f(t_n) \in [1-s,1].
\end{equation}
That is, $\widetilde{X}(s) \stackrel{\eqref{Xtildedefn}}{=} \sup f^{-1}[1-s,1] \in f^{-1}[1-s,1]$. Substituting 
$f^{-1}[1-s,1] = (-\infty, \widetilde{X}(s)], \ \forall s \in(0,1)$ into \eqref{Xtilde(ii)eq},
$$
(-\infty, \widetilde{X}(s)] = \lim_{s' \searrow s} \ (-\infty, \widetilde{X}(s')]
$$
i.e.\ $\widetilde{X}$ is right continuous.

$P(\widetilde{X} \geq t) = f(t)$: It is sufficient to show, 
\begin{equation}\label{Xtilde(iii)eq}
\forall s \in (0,1), \ \widetilde{X}(s) \geq t \iff s \in [1-f(t),1]
\end{equation}
since $m([1-f(t),1])=f(t)$. 

$(\Leftarrow)$: If $s \in [1-f(t),1]$, then 
$$
\widetilde{X}(s) \stackrel{\widetilde{X} \nearrow}{\geq} \widetilde{X}(1-f(t)) \stackrel{\eqref{Xtildedefn}}{=} \sup f^{-1}[f(t),1] \geq t
$$
since $t \in f^{-1}[f(t),1]$. 

$(\Rightarrow)$: One has 
$$
f(t) \stackrel{f \searrow}{\geq} f(\widetilde{X}(s)) \stackrel{\eqref{Xtildeclosedinterval}}{\in} [1-s,1].
$$ 
It follows $0 < 1-s \leq f(t) \implies s \in [1-f(t),1]$. This shows \eqref{Xtilde(iii)eq} as desired.

(iii): Define
\begin{equation}\label{Xtildef'def} 
f'(t) := P(X \geq t), \ \forall t \in \mathbb{R}.
\end{equation} 
By lemma~\ref{uniquenessneat} (iii), 
\begin{equation}\label{XtildeX'def} 
X'(s) := \lim_{s' \searrow s} X(s'), \ \forall s \in (0,1)
\end{equation} 
is a neat, right continuous r.v.\ with $P(X' \geq t) = P(X \geq t) = f'(t), \ \forall t \in \mathbb{R}$. 
By parts (i) and (ii), 
\begin{equation*}
\widetilde{X}'(s) := \sup f'^{-1}[1-s,1], \ \forall s \in (0,1)
\end{equation*}
is also a right continuous neat r.v.\ with $P(\widetilde{X}' \geq t) = f'(t), \ \forall t \in \mathbb{R}$. Only one such neat r.v.\ exists by lemma~\ref{uniquenessneat} (i), so they must be equal. That is, 
\begin{equation}\label{XtildeX=X'}
X'(s) = \widetilde{X}'(s) = \sup {f'}^{-1}[1-s,1], \ \forall s\in(0,1).
\end{equation}
Now, by hypothesis $P(X \geq t) \leq f(t), \ \forall t \in \mathbb{R}$, i.e.\ $f'(t) \stackrel{\eqref{Xtildef'def}}{\leq} f(t), \ \forall t \in \mathbb{R}$. Fix $ s \in (0,1)$. It follows
${f'}^{-1}[1-s,1] \subset f^{-1} [1-s,1]$. Indeed,
$$
t \in {f'}^{-1}[1-s,1] \implies f'(t) \in [1-s,1] \implies f(t) \in [1-s,1].
$$
Applying the supremum to this,
$$
X'(s) \stackrel{\eqref{XtildeX=X'}}{=} \sup {f'}^{-1} [1-s,1] \leq \sup f^{-1} [1-s,1] \stackrel{\eqref{Xtildedefn}}{=} \widetilde{X}(s)
$$
It is now sufficient to show $X(s) \leq X'(s)$: because $X \nearrow $ ($X \in \mathcal{N}$ by hypothesis (definition~\ref{neatrvs})) one has 
$X(s) \leq \lim_{s' \searrow s} X(s') \stackrel{\eqref{XtildeX'def}}{=} X'(s)$. 
\end{proof}
We will frequently use the following notation throughout this thesis.
\begin{align}\label{Xtildeabbreviations}
\widetilde{X}_k &= \sup f_k^{-1}[1-s,1], \ \forall s \in (0,1) \\ 
X &= (X_1,...,X_n)
\end{align}

The purpose of the next lemma, and neat r.v.'s in general, is to simplify the domain of real r.v.'s as much as 
possible. As described, this will lead to proving \eqref{eg2eq}. Later, this will help proving results about
shift operators, because WLOG convergence in distribution becomes a.s.\ convergence, and the organized domains
will be used in proofs.

\begin{lemma}[existence of neat r.v.'s]\label{existenceneat}
   If $Y: \Omega \rightarrow \mathbb{R}$ is a random variable, then $\exists Y' \in \mathcal{N}$ with the same distribution as $Y$.
\end{lemma}
\begin{proof}
$f(t) := P(Y \geq t)$ satisfies the requirements to apply lemma \ref{Xtilde}. Applying it, 
\begin{equation}\label{existenceeq1}
Y'(s) := \sup f^{-1}[1-s,1], \ \forall s \in (0,1), \text{ is neat\ }
\end{equation}
and 
\begin{equation}\label{existenceeq2}
P(Y' \geq t) = P(Y \geq t), \ \forall t \in \mathbb{R}.
\end{equation}
Thus $Y$ and $Y'$ have the same reversed CDF. Therefore they have the same distribution (section~\ref{preliminaries}). 
\end{proof}
Define $\mathcal{N}^n$ to be the set consisting of all product r.v.'s of $n$ neat r.v.'s $X_1,...,X_n \in \mathcal{N}$. Thus the elements of $\mathcal{N}^n$ are of the form $(X_1,...,X_n): (0,1)^n \rightarrow \mathbb{R}^n$ which has mutually ind.\ components by construction of product r.v.'s (section~\ref{preliminaries}) and each component is monotonically increasing. 
This allows generalizing the previous result to higher dimensions. 

\begin{corollary}[WLOG $(X_1,...,X_n)$ is neat for $X_k$ mutually ind.] \label{c2.1} 
   
   Let \\ $X = (X_1,...,X_n):\Omega \rightarrow \mathbb{R}^n$ have mutually independent components. Then $\exists X'= (X_1',...,X_n') \in \mathcal{N}^n$ 
   such that $X'$ has the same distribution as $X$.
   \end{corollary}
   \begin{proof}
   Apply lemma~\ref{existenceneat} to $X_k$ to get $X_k' \in \mathcal{N}$ with the same distribution as $X_k$ 
   for $k=1,...,n$. Consider the product r.v.\ $X' = (X_1',...,X_n') \in \mathcal{N}^n$ of $X_1',...,X_n'$. 

   Fix a closed product set $J = J_1 \times \cdots \times J_n \subset \mathbb{R}^n$. By independence of the components of $X$, and since $X_k$ has the same distribution as $X_k'$ for all $k$,
   \begin{align*}
   P\big((X_1,...,X_n) \in J\big) &= P(X_1 \in J_1) \cdots P(X_n \in J_n)  \\ 
   &= P(X_1' \in J_1) \cdots P(X_n' \in J_n).
   \end{align*}
   The product r.v.\ $(X_1',...,X_n')$ has $P(X_1' \in J_1) \cdots P(X_n' \in J_n) = P\big((X_1',...,X_n') \in J_1 \times \cdots \times J_n\big)$ (section~\ref{preliminaries}). 
   Since the closed product set $J$ was arbitrary, and such subsets of $\mathbb{R}^n$ form a $\pi$-system which 
   generates the Borel $\sigma$-algebra on $\mathbb{R}^n$, we may apply Dynkin's $\lambda$-$\pi$ theorem (section~\ref{preliminaries}) to get 
   $X$ and $X'$ have the same distribution.
\end{proof}
For $U \subset \mathbb{R}^n$ in the Borel $\sigma$-algebra, we are interested in the supremum of \\
$P((X_1,...,X_n) \in U)$ over $X_k$ mutually ind.\ and respectively tail bounded (section~\ref{tailboundintroduction} and 
section~\ref{sectionsimplecases}). \index{$\mathcal{B}$} By corollary~\ref{c2.1}, WLOG we may assume $X_k$ are neat. This motivates the following notation.
\begin{equation}\label{absolutetails}
\mathcal{B} := \big\{f:[0,\infty) \rightarrow [0,1]\ \big|\ f(0) = 1, f \searrow 0, f \text{ is left continuous} \big\}.
\end{equation}
This is our set of absolute tail bounds (recall section~\ref{tailtypes} (i)). \index{$\mathcal{T}(f)$} Also define 
\begin{equation}\label{T(f)}
\mathcal{T}(f) := \big\{X \in \mathcal{N}^n \ \big| \ P(|X_k| \geq t) \leq f_k(t), \forall t \geq 0, k = 1,...,n\big\}, \ \forall f \in \mathcal{B}^n.
\end{equation}
Here $f=(f_1,...,f_n)$ and $X=(X_1,...,X_n)$ (these abbreviations for $n$-tuples of tail bounds, and $n$-tuples of neat r.v.'s, will 
be used throughout this thesis).

In other words, $\mathcal{T}(f)$ is the set of $n$-tuples of neat r.v.'s whose components satisfy the absolute tail bounds 
$f_1,...,f_n$ respectively. Left continuity is assumed because
\begin{enumerate}[label=(\roman*)]
   \item It reduces the need for limits in tail bound proofs.
   \item One often wants a r.v. $\widetilde{X}$ such that $P(\widetilde{X} \geq t) = f(t)$. The existence of such a 
   $\widetilde{X}$ requires $f$ to be left continuous (section~\ref{preliminaries}).
\end{enumerate}

In summary, we have the convenient description
\begin{equation}\label{thingthatneedsareference}
\sup_{\substack{X_k \text{ mutually ind., absolute} \\ \text{tail bounded by } f_k, \ \forall k}} P(X \in U) = \sup_{X \in \mathcal{T}(f)} P(X \in U).
\end{equation}

\section{Solving example $2$ (eq. $(3)$)}
As said in section~\ref{tailtypes}, there are characteristically different types of tail bounds. 
The approach for maximizing the measure of $U$ for different types of tail bounds on $X_k$ turns out to be conceptually almost identical, but the proofs must be treated separately. 
In particular, right, left, absolute, and $2$ tail bounds.

The previous discussion/notation for absolute tail bounds easily applies to the other types of tail bounds. \index{$\mathcal{B}_R$} \index{$\mathcal{T}_R(f)$} Define
\begin{align}\label{BRTR}
   \mathcal{B}_R &:= \{f:\mathbb{R} \rightarrow [0,1] \ \big| \lim_{t \rightarrow -\infty} f(t) = 1, f \searrow 0, f \ \text{is left continuous}\} \\
   \mathcal{T}_R(f) &:= \{X \in \mathcal{N}^n \ \big| \ \forall t \in \mathbb{R}, \ P(X_k\geq t) \leq f_k(t), k=1,...,n\}, \ \forall f \in \mathcal{B}_R^n \label{TR}
\end{align}
$\mathcal{B}_R$ is the set of right tail bounds and $\mathcal{T}_R(f),\  f=(f_1,...,f_n)$, is the set of $n$-tuples of neat r.v.'s whose components satisfy the right tails $f_1,...,f_n$ respectively. Notice 
$\mathcal{T}(f) \subset \mathcal{N}^n$, and as such only contains r.v.'s with mutually ind.\ components (because all elements of $\mathcal{N}^n$ have mutually ind.\ components by construction of product r.v.'s).

For left tail bounds, \index{$\mathcal{B}_L$} \index{$\mathcal{T}_L(f)$} the corresponding notation is 
\begin{align*}
   \mathcal{B}_L &:= \{f:\mathbb{R} \rightarrow [0,1] \ \big|  \lim_{t \rightarrow -\infty} f(t) = 0, f \nearrow 1, f \ \text{is right continuous} \} \\
   \mathcal{T}_L(f) &:= \{X \in \mathcal{N}^n \big| \forall t \in \mathbb{R}, P(X_i \leq t) \leq f_i(t), i = 1,...,n \}, \ \forall f \in \mathcal{B}_L^n
\end{align*}
For $2$ tail bounds, \index{$\mathcal{B}_2$} \index{$\mathcal{T}_2(f)$} the corresponding definitions are 
\begin{align}\label{B2T2}
   \mathcal{B}_2 &:= \{ (f^-,f^+) \in \mathcal{B}_L \times \mathcal{B}_R \big| f^-(0) = f^+(0) = 1\}\\
   \mathcal{T}_2(f) &:= \big\{X \in \mathcal{N}^n \ \big| \ \forall t \in \mathbb{R}, \ P(X_k \geq t) \leq f_k^+(t), \\ & \ \ \ \ \ \ \ \ \ \ \ \ \ \ \ \ \ \ \ \ \ \ \ \ \ \ \ \ \ \ P(X_k \leq t) \leq f_k^-(t), \ k=1,...,n \big\}, \ \forall f \in \mathcal{B}_2^n \nonumber
\end{align}
(here $f= ((f_1^-,f_1^+),...,(f_n^-,f_n^+))$. The presence of the condition $f^-(0) = f^+(0) = 1$ for $2$ tails was described in 
section~\ref{tailtypes}.

The next few theorems will show how neat r.v.'s and the $\widetilde{X}$ in lemma~\ref{Xtilde} can help prove the sharpest tail 
bound. When example $2$ was introduced in section~\ref{tailtypes}, an outline was given for how to prove the sharpest tail of 
$P(X_1+\cdots + X_n \geq t)$ is $P(\widetilde{X}_1+\cdots+\widetilde{X}_n \geq t)$. It required lemma~\ref{Xtilde} and lemma~\ref{existenceneat}
to work, but now these have been proved, so we are ready; the argument for \eqref{eg2eq} is given below.

The steps consist of
\begin{align}
\sup_{\substack{X_k \text{mutually ind.}, \\ P(X_k \geq t) \leq f_k(t), \forall t}} P(X_1 + \cdots + X_n \geq t) &\stackrel{(1)}{=} \sup_{X \in \mathcal{T}_R(f)} P(X_1 + \cdots + X_n \geq t) \\ 
&\stackrel{(2)}{=} P(\widetilde{X}_1+\cdots + \widetilde{X}_n \geq t) \label{1.3again}
\end{align}
where $\widetilde{X}_k := \sup f_k^{-1}[1-s,1], \forall s \in (0,1)$, is from lemma~\ref{Xtilde}. 

To show $(1)$: WLOG $X \in \mathcal{N}^n$ by corollary~\ref{c2.1}, and $P(X_k \geq t) \leq f_k(t), k=1,...,n \implies (X_1,...,X_n) \in \mathcal{T}_R(f)$ by definition
of $\mathcal{T}_R$, \eqref{TR}. 

$(2)$ is an application of the theorem below.
\begin{theorem}\label{thm.1}
   Let $g:\mathbb{R}^n \rightarrow \mathbb{R}$ be Borel measurable such that
   \begin{equation}\label{ghypothesis}
   \forall x,y \in \mathbb{R}^n, \ x_i \leq y_i, \ i=1,...,n, \implies g(x) \leq g(y).
   \end{equation} 
   Let $f \in \mathcal{B}_R^n$ \eqref{BRTR} be an $n$-tuple of right tails. Then, 
   \begin{equation}\label{whoknows}
   \sup_{X \in \mathcal{T}_R(f)} P(g(X) \geq t) = P(g(\widetilde{X}) \geq t), \ \forall t \in \mathbb{R}.
   \end{equation}
\end{theorem}
\begin{proof} 
Fix $X = (X_1,...,X_n) \in \mathcal{T}_R(f)$ \eqref{TR}. In particular, $X_k$ are neat. By \eqref{ghypothesis} and lemma~\ref{Xtilde} (iii), 
$$
\forall s = (s_1,...,s_n) \in (0,1)^n, \ g(X(s)) \leq g(\widetilde{X}(s))
$$
where $\widetilde{X} = (\widetilde{X}_1,...,\widetilde{X}_n)$. Fix $t \in \mathbb{R}$. It follows 
$$
\{s \in(0,1)^n \ \big| \ g(X(s)) \geq t\} \subset \{ s \in(0,1)^n \ \big| \ g(\widetilde{X}(s)) \geq t\}
$$
Denote these LHS, RHS. Thus,
$$
P(g(X) \geq t) = m(\text{LHS}) \leq m(\text{RHS}) = P(g(\widetilde{X})) \geq t).
$$ 
Taking $\sup$ over $X \in \mathcal{T}_R(f)$, and then using $\widetilde{X} \in \mathcal{T}_R(f)$ (by lemma~\ref{Xtilde} (i), (ii), and \eqref{BRTR}),
\begin{equation*}
\sup_{X \in \mathcal{T}_R(f)} P(g(X) \geq t) = P(g(\widetilde{X}) \geq t).
\end{equation*}
Since $t$ was arbitrary we are done.
\end{proof}  

The form of \eqref{whoknows} is characteristic of sharpest tail bounds, or equivalently, the maximum mass/measure of subsets. That is, for $X_1,...,X_n$
tail bounded and mutually ind., the sharpest tail of $g(X_1,...,X_n)$ is the probability that $g(X_1^t,....,X_n^t)$ exceeds $t$ for some carefully chosen
$X_k^t$ satisfying the tails. 

One gets the rephrasement of this in terms of subsets by setting $V_t = g^{-t}[t,\infty)$. In particular, the following form of equation is common: 
\begin{equation}\label{Xtmasseq}
\sup_{X \in \mathcal{T}_R(f)} P(X \in V_t) = P\big((X_1^t,...,X_n^t) \in V_t \big).
\end{equation}
It seems difficult to find $g$ and tails $f_1,...,f_n$, such that \eqref{Xtmasseq} has a closed form, or even a simple expression with integrals. 
This explains why previous tail bound research has looked to find good approximations of $\eqref{Xtmasseq}$ for certain $g$ and tails.

$\textbf{Example}$: One semi-closed form of a sharpest tail bound can be contrived with Gaussian r.v.'s: consider ind.\ Gaussian r.v.'s $Y_k, \ k = 1,...,n,$ with probability distribution functions
$$
f_{Y_k}(x) = \frac{1}{\sigma_k \sqrt{2\pi}} \exp \left(-\frac{1}{2}\left(\frac{x - \mu_k}{\sigma_k} \right)^2\right)
$$
where $\mu_k$ is the mean and $\sigma_k^2$ the variance. It is well known
$$
P(Y_k \geq t) = \frac{1}{2} \left[1 + \text{erf} \left( \frac{-t+\mu_k}{\sqrt{2}\sigma_k}\right) \right], \ \forall t \in \mathbb{R}
$$
where $\text{erf}(z) = \frac{2}{\sqrt{\pi}} \int_0^z e^{-u^2} du$. Consider the right tails $f_k(t) := P(Y_k \geq t), \ \forall t \in \mathbb{R}, \ 
k=1,...,n$. 

Suppose $X_k$ are mutually ind.\ r.v.'s with $P(X_k \geq t) \leq f_k(t), \ \forall t \in \mathbb{R}, \ k = 1,...,n$. Notice since $P(\widetilde{X}_k \geq t) = f_k(t), \ \forall t \in \mathbb{R}$ (lemma~\ref{Xtilde}, \eqref{Xtildedefn}) that $\widetilde{X}_k$ is a Gaussian r.v.\ with mean $\mu_k$ and variance $\sigma_k^2$ too.

Applying theorem~\ref{thm.1} to $f=(f_1,...,f_n)$ and $g:(x_1,...,x_n) \mapsto x_1 + \cdots + x_n$ gives
\begin{align} 
   \sup_{X \in \mathcal{T}_R(f)} P(X_1 + \cdots + X_n \geq t) &= P(\widetilde{X}_1 + \cdots + \widetilde{X}_n \geq t) \nonumber \\ 
   &=\frac{1}{2} \left[1 + \text{erf} \left(\frac{-t+\mu}{\sqrt{2}\sigma} \right) \right],\ \forall t \in \mathbb{R} \label{errortailbound}
\end{align} 
since $\widetilde{X}_1+\cdots+\widetilde{X}_n$ is Gaussian with mean $\mu = \mu_1+\cdots+\mu_n$ and variance $\sigma^2 = \sigma_1^2+\cdots + \sigma_n^2$. This 
gives the sharpest right tail for $X_1+\cdots + X_n$, $X_k$ being subject to the conditions.

Moving on, the method used to prove \eqref{whoknows} works for other cases.

\section{Example $2$'s solution applied to other cases}\label{afewothercases}
Suppose $g$ is a multinomial with positive coefficients, $X_i$ are (mutually) ind., and 
$$
P(|X_i| \geq t) \leq f_i(t), \forall t \geq 0, i=1,...,n.
$$
Suppose $\exists \widetilde{X}_i$ such that (note $f_i(0) =1$) 
$$
P(\widetilde{X}_i \geq t) = 
\begin{cases}
f_i(t) & t \geq 0 \\
1 & t \leq 0
\end{cases}
$$
Intuitively, $\widetilde{X}_i$ are $X_i$ made as large as possible. Furthermore, $g(\widetilde{X}_1,...,\widetilde{X}_n)$ should be 
$g(X_1,...,X_n)$ made as large (positive) as possible, that is
$$
P(g(\widetilde{X}_1,...,\widetilde{X}_n) \geq t)
$$
should be the sharpest right tail of $g(X_1,...,X_n)$. The next theorem proves this.
\begin{theorem}\label{thm.2}
Let $g:\mathbb{R}^n \rightarrow \mathbb{R}$ be measurable, satisfy 
\begin{enumerate}[label=(\roman*)]
\item $0 \leq x_i \leq y_i, i=1,...,n  \implies g(x_1,...,x_n) \leq g(y_1,...,y_n)$
\item $\forall (x_1,...,x_n) \in \mathbb{R}^n, \ g(x_1,...,x_n) \leq g(|x_1|,....,|x_n|)$
\end{enumerate}
Let $f \in \mathcal{B}^n$ \eqref{absolutetails} be an $n$-tuple of absolute tail bounds. Then, 
\begin{equation}\label{positivemultinomialeq}
   \sup_{X \in \mathcal{T}(f)} P(g(X) \geq t) = P(g(\widetilde{X}) \geq t), \ \forall t \in \mathbb{R}.
\end{equation}
\end{theorem} 
\begin{proof}
Minor technicality: by definition ($\mathcal{B}$ definition \eqref{absolutetails}) the domain of $f_i$ is $[0,\infty)$ and $f_i(0)=1$. So, 
to apply lemma~\ref{Xtilde}, which had $f_i$ with domain $\mathbb{R}$, one should extend $f_i(t) = 1, \forall t < 0$. 

Fix $X \in \mathcal{T}(f)$. Writing $s=(s_1,...,s_n)$, 
$$
I_t := \{ s \in (0,1)^n \ \big| \ g(X(s)) \geq t \} \stackrel{\text{(ii)}}{\subset} I_t' := \{s \in (0,1)^n \ \big| \ g(|X_1(s_1)|),...,|X_n(s_n)|) \geq t\}
$$
So
\begin{equation}\label{multinomialeq4}
P(g(X) \geq t) = m(I_t) \leq m(I_t') = P((g(|X_1|),...,|X_n|) \geq t).
\end{equation}
Because $(X_1,...,X_n)$ has ind.\ components, $(|X_1|,...,|X_n|)$ does too. By corollary~\ref{c2.1} there exists 
$Y=(Y_1,...,Y_n) \in \mathcal{N}^n$ with the same distribution as $(|X_1|,...,|X_n|)$. In particular, 
\begin{equation}\label{multinomialeq44}
P\big((|X_1|,...,|X_n|) \in g^{-1}[t,\infty)\big) = P\big(Y \in g^{-1}[t,\infty)\big).
\end{equation}
Notice $P(|Y_i| \geq t) \leq f_i(t), \ \forall t \geq 0$ (since $P(|X_i| \geq t) \leq f_i(t)$ by hypothesis). Then $Y = (Y_1,...,Y_n) \in \mathcal{T}(f)$
by definition of $\mathcal{T}$ \eqref{absolutetails}. Furthermore $Y \geq 0$. It follows, by \eqref{multinomialeq4} and \eqref{multinomialeq44},
\begin{equation}\label{multinomialeq5}
\sup_{X \in \mathcal{T}(f)} P(g(X) \geq t) \leq \sup_{\substack{Y \in \mathcal{T}(f) \\ Y_i \geq 0}} P(g(Y) \geq t)
\end{equation} 
The converse inequality holds trivially, thus equality holds. At this point one can use an almost identical argument to that used for \eqref{whoknows} to show 
\begin{equation}\label{multinomialeqlast}
\sup_{\substack{Y \in \mathcal{T}(f) \\ Y_i \geq 0}} P(g(Y) \geq t) = P(g(\widetilde{X}) \geq t).
\end{equation}
$t \geq 0$ was arbitrary, so we are done.  
\end{proof}

Comment about application: if one tries to apply theorem~\ref{thm.1} and theorem~\ref{thm.2} to specific cases to get concrete tail bounds/concentration inequalities, 
one gets integrals. The goal would be to accurately bound these integrals above or below with understandable functions. This is good because it readily allows 
access to tail bounds above and below which one proveably knows is close to the best possible tail bound. However when I tried to approximate the integrals for 
a few simple cases, the approximations were extremely unwieldly so I did not look far into the idea. Theorem~\ref{thm.1} and theorem~\ref{thm.2} are closer to 
application than the shift operators in chapter 4 are. (Although I defined and proved the results about shift operators I wanted to prove, I do not know how to 
make shift operators useful).

Theorem~\ref{thm.2} is also applicable to Schur multipliers. \cite{Skripka2024} includes a definition of Schur multipliers.

\begin{definition}[Schur multiplier]
Let \(\mathcal{M}_n\) denote the space of \(n \times n\) real matrices and \(E_{i,j}\) the elementary matrix whose only non-zero entry is in position \((i,j)\). 
Let \(d,n \in \mathbb{N}\) and let
\[
\mathfrak{m}(d) = \{m_{j_1,\dots,j_{d+1}}\}_{j_1,\dots,j_{d+1}=1}^n
\]
be a multi-dimensional matrix with real entries. The \(d\)-linear Schur multiplier \(\mathfrak{M}_{\mathfrak{m}(d)} : \mathcal{M}_n \times \cdots \times \mathcal{M}_n 
\to \mathcal{M}_n\) is
\begin{equation}\label{Schurmultipliereq0}
\mathfrak{M}_{m(d)}(X_1,\dots,X_d) = \sum_{j_1,\dots,j_{d+1}=1}^n m_{j_1,\dots,j_{d+1}} x^{(1)}_{j_1,j_2} \cdots x^{(d)}_{j_d,j_{d+1}} E_{j_1,j_{d+1}}.
\end{equation}
\end{definition}

Although the definition of Schur multipliers is for real \(n \times n\) matrices, we are interested in inputting \(d\) real \(n \times n\) random matrices into 
\(\mathfrak{M}_{m(d)}\). So although in definition 2 \(X_k\) are real matrices, in theorem 2.3.2 \(X_k\) will be real random matrices. Hence there is a change from 
lower case \(x\)'s (real variables) for the entries of each \(X_k\) to upper case \(x\)'s (real random variables) for the entries of each \(X_k\) in theorem~\ref{Schurmultipliertailbound}.

\begin{theorem}\label{Schurmultipliertailbound}
Let \(\mathfrak{m}(d) = \{m_{j_1,\dots,j_{d+1}}\}_{j_1,\dots,j_{d+1}=1}^n\) have all non-negative entries. Let \(X_1,\dots,X_d\) be real random \(n \times n\) matrices 
with mutually independent components. Denote the \(i,j\) entry of \(X_k\) by \(X^{(k)}_{ij}\) for all \(i,j,k\). Let \(f_{ijk} \in B\) and 
\(X^{(k)}_{ij} \in \mathcal{T}(f_{ijk})\) for all \(i,j,k\). Denote \(\widetilde{X}_k = \{\widetilde{X}^{(k)}_{ij}\}_{i,j=1}^n\) for \(k=1,\dots,d\). Here 
\(\widetilde{X}^{(k)}_{ij}\) is the ``right-shifted'' version of \(X^{(k)}_{ij}\) (see lemma~\ref{Xtilde} \eqref{Xtildedefn}). Then $\forall t \ge 0$,
\begin{equation}\label{Schurmultipliereq1}
P\bigg(\operatorname{Tr}\big(\mathfrak{M}_{\mathfrak{m}(d)}(X_1,\dots,X_d)\big) \ge t\bigg) \le P\bigg(\operatorname{Tr}\big(\mathfrak{M}_{\mathfrak{m}(d)}(\widetilde{X}_1,\dots,\widetilde{X}_d)\big) \ge t\bigg)
\end{equation}
and the RHS cannot be improved.
\end{theorem}

\begin{proof}
Let \(g(x_1,\dots,x_{n^2 d}) : \mathbb{R}^{n^2 d} \to \mathbb{R}\) where \(g(x_1,\dots,x_{n^2 d}) = \\ \operatorname{Tr}(\mathfrak{M}_{\mathfrak{m}(d)}(X_1,\dots,X_d))\) where 
\(\operatorname{Tr}(\mathfrak{M}_{\mathfrak{m}(d)}(X_1,\dots,X_d))\) has had each real random variable entry replaced by a real variable \(x_i\), of which there are \(n^2 d\) many. 
Because \(m_{j_1,\dots,j_{d+1}}\) are non-negative for all \(j_l\), one sees by \eqref{Schurmultipliereq0} that (i) and (ii) of theorem~\ref{thm.2} hold 
for \(g\). Therefore \eqref{positivemultinomialeq} holds, which is the same as \eqref{Schurmultipliereq1}.
\end{proof}

Assuming the same conditions as theorem~\ref{Schurmultipliertailbound}, theorem~\ref{thm.2} also derives
\begin{align}
&P\bigg(\Big|\operatorname{Tr}(\mathfrak{M}_{\mathfrak{m}(d)}\big(X_1,\dots,X_d)\big)\Big| \ge t\bigg) \nonumber \\ 
&\leq P\bigg(\Big|\operatorname{Tr}(\mathfrak{M}_{\mathfrak{m}(d)}\big(\widetilde{X}_1,\dots,\widetilde{X}_d)\big)\Big| \ge t\bigg), \forall t \ge 0 \label{Schurmultipliereq2}
\end{align}
and the RHS cannot be improved. The proof is essentially the same, simply include absolute value bars as part of \(g\).

I'm quite sure \eqref{Schurmultipliereq2} can be generalized to \(X_k\) being random complex matrices with mutually independent matrix entries without too much difficulty, but I have not worked out the details. Although in the shift operator chapter we will develop tools which in theory are capable of expressing the sharpest tail bound for
\[
P\bigg(\Big|\operatorname{Tr}(\mathfrak{M}_{\mathfrak{m}(d)}(X_1,\dots,X_d))\Big| \geq t\bigg)
\]
when \(\mathfrak{m}(d) = \{m_{j_1,\dots,j_{d+1}}\}_{j_1,\dots,j_{d+1}=1}^n\) has non-negative and negative entries, I do not know how to produce this theoretically 
possible expression. This is representative of the current state of shift operators: they are theoretically capable of expressing sharpest tail bounds for every function of mutually independent r.v.'s, but only in a few cases are methods found which lets one actually do this expression.

Indeed, in chapter 4 we generalize the method used in theorem~\ref{thm.1} and theorem~\ref{thm.2}; typically one does not want to consider the r.v. 
\(g(\widetilde{X})\), but the r.v. \(g(\mathscr{S}(f,G))\) where \(G = G^1 \times \cdots \times G^n \subset \mathbb{R}^n\) is a properly chosen closed subset. 
\(\mathscr{S}\) is called a shift operator, and returns (very roughly speaking) the r.v. \((X_1,\dots,X_n)\) that is ``shifted outwards'' on \(G\) as far as possible 
while satisfying the tails. For \(G = \mathbb{R}^n\), in fact
\[
\mathscr{S}(f,\mathbb{R}^n) = \widetilde{X},
\]
so \(\mathscr{S}(f,G)\) is a generalization of \(\widetilde{X}\).

There are different types of shift operators, corresponding if \(X_i\) are left, right, absolute, or 2 tail bounded (section~\ref{tailtypes}). 
These will respectively be written $\mathscr{S}_L(f,G)$ for $f \in \mathcal{B}_L^n$, $\mathscr{S}_R(f,G)$ for $f \in \mathcal{B}_R^n$, 
$\mathscr{S}(f,G)$ for $f \in \mathcal{B}^n$, and $\mathscr{S}_2(f,G)$ for $\mathcal{B}_2^n$ \eqref{absolutetails}-\eqref{B2T2}.

The proofs for the different shift operators are very similar and all used the same proof blueprint with modest adjustments. Therefore, I tried to define a generalized
shift operator which combined the other shift operators, but for some unknown reason I encountered newly necessary sets whose measurability I could not prove.

\section{Convergence of neat r.v.'s}
\begin{lemma}[Tail bounds are preserved under convergence in distribution]\label{preservedconvind}
Let \\ $(X^m)_{m \in \mathbb{N}}$ be a sequence of real r.v.'s (not necessarily neat) that converges to the real r.v.\ $X$ in distribution.
\begin{enumerate}[label=(\roman*)]
\item If $f \in \mathcal{B}_R$ \eqref{BRTR} and $\forall m \in \mathbb{N}, \forall t \in \mathbb{R}, \ P(X^m \geq t) \leq f(t)$, then $$\forall t \in \mathbb{R}, \ P(X \geq t) \leq f(t)$$
\item If $f \in \mathcal{B}_L$ and $\forall m \in \mathbb{N}, \forall t \in \mathbb{R}, \ P(X^m \leq t) \leq f(t)$, then $$\forall t \in \mathbb{R}, \ P(X \leq t) \leq f(t)$$
\item If $f \in \mathcal{B}$ \eqref{absolutetails} and $\forall m \in \mathbb{N}, \forall t \geq 0, \ P(|X^m| \geq t) \leq f(t)$, then $$\forall t \geq 0, \ P(|X| \geq t) \leq f(t)$$
\item If $f=(f^-,f^+) \in \mathcal{B}_2$ \eqref{B2T2} and $\forall m \in \mathbb{N}, \forall t \geq 0, \ P(X^m \geq t) \leq f^+(t)$, also $\forall t \leq 0, P(X^m \leq t) \leq f^-(t)$, then $$\forall t \geq 0, P(X \geq t) \leq f^+(t) \text{ and } \forall t \leq 0, P(X \leq t) \leq f^-(t)$$
\end{enumerate}
\end{lemma}
\begin{proof}
(i): Since $X^m$ converges to $X$ in distribution, $P(X^m \geq t) \xrightarrow[m \rightarrow \infty]{} P(X \geq t), \forall t$ with $P(X \geq t)$ continuous at $t$ (section~\ref{preliminaries}). Since monotone functions have 
countable discontinuities, $P(X^m \geq t) \rightarrow P(X \geq t)$ on a dense subset. 

Fix $t' \in \mathbb{R}$. Since $P(X \geq t)$ is continuous on a dense subset, let $(t_k)_k$ be such that $\forall k \in \mathbb{N}, P(X \geq t)$ 
is continuous at $t_k$ and $t_k \nearrow t'$. It follows,
$$
\forall k \in \mathbb{N}, P(X \geq t_k) = \lim_{m \rightarrow \infty} P(X^m \geq t_k) \leq f(t_k).
$$
Notice $P(X \geq t') \leq P(X \geq t_k) \leq f(t_k)$. Since $f$ is left continuous, taking $k\rightarrow \infty$ gives $P(X \geq t') \leq f(t')$. 

(ii): Done similarly (symmetrically) to (i).

(iii): Define $Y^m := |X^m|, \ Y := |X|$. First we will show $Y^m \xrightarrow[]{d} Y$. Suppose $P(Y \geq t)$ is continuous at $t$. Because $Y^m, Y \geq 0$, if $t \leq 0$ then 
$P(Y \geq t) = P(Y^m \geq t) = 1$ so $P(Y^m \geq t) \rightarrow P(Y \geq t)$ trivially. Therefore assume $t > 0$. 

Because $P(Y \geq t)$ is continuous at $t$, necessarily $P(Y = t) = 0$, i.e.\ $P(X \in \{-t,t\}) = 0$. Because the topological boundary $\partial\big((-\infty,t] \cup [t,\infty) \big) = \{-t,t\}$, by definition $(-\infty, t] \cup [t,\infty)$ is a continuity set of $X$. Recall Portmanteau's theorem (section~\ref{preliminaries}) states convergence in distribution is equivalent to convergence on continuity sets. Thus because $X^m \xrightarrow[]{d} X$, 
$$
P\big(X^m \in (-\infty, -t] \cup [t,\infty)\big) \rightarrow P\big(X \in (-\infty,-t] \cup [t,\infty)\big).
$$
On the other hand clearly
\begin{align*} 
P(Y^m \in [t,\infty)) &= P(X^m \in (-\infty, -t] \cup [t,\infty)), \ \forall m \in \mathbb{N} \\ 
P(Y \in [t,\infty)) &= P(X \in (-\infty, t] \cup [t,\infty)).
\end{align*}
Thus $P(Y^m \in [t,\infty)) \rightarrow P(Y \in [t,\infty))$ too. This shows $Y^m \xrightarrow[]{d} Y$. Thus one can use (i) on $(Y^m)_m, Y$.

(iv): Apply (i) to $f^+$ and (ii) to $f^-$.
\end{proof}
The following lemma easily generalizes to higher dimensions. In particular, if \\ $(X^m)_{m \in \mathbb{N}} \subset \mathcal{N}^n, X \in \mathcal{N}$, such that the joint CDFs of $X^m$ converge to the joint CDF of $X$ at every point of $\mathbb{R}^n$ the joint CDF of $X$ is continuous at, then one has $X^m \xrightarrow[]{a.e.} X$ (section~\ref{preliminaries}). Recall any r.v.\ with mutually ind.\ components has distribution 
equal to a r.v.\ in $\mathcal{N}^n$ (corollary~\ref{c2.1}), and that convergence in distribution is the weakest form of convergence. It follows, for real r.v.'s with mutually ind.\ components, 
without loss of generality all convergence is convergence a.e.. (The following comment about Skorokhod's representation 
theorem (and reference to) was added after the dissertation submission, and is not present in the university publication due to my ignorance of Skorokhod's theorem). As mentioned at the 
start of chapter $2$, this is a special case of Skorokhod's representation theorem \cite{Skorokhod} pg. 70. Note that Skorokhod's representation theorem only guarantees almost sure convergence. The countability of the non-convergent subset in the 
monotone case (lemma~\ref{neatnessconvergence}) is an additional observation made specific to this case.
\begin{lemma}\label{neatnessconvergence}
   Let $(X^m)_m \subset \mathcal{N}, X \in \mathcal{N}$ be such that $X^m \xrightarrow[]{d} X$. Then $X^m \xrightarrow[]{a.e.} X$. 
\end{lemma}
\begin{proof}
Recall if $(a_n)_{n \in \mathbb{N}} \subset V$ and $U \subset V$, one says $a_n$ is eventually in $U$ iff $\exists N \in \mathbb{N}$ such that $\forall n \geq N, a_n \in U$. 

Suppose $0<s'<s_0<s<1$. We will prove the lemma from the following two claims: $\forall \epsilon > 0, \ X^m(s_0)$ is eventually in $(-\infty, X(s) + \epsilon)$ and $\forall \epsilon >0, X^m(s_0)$ 
is eventually in $(X(s')-\epsilon, \infty)$. To understand why these claims are sufficient to prove the lemma, suppose $0 < \epsilon' < \min\{s_0, 1-s_0\}$. Set $s = s_0 + \epsilon'$ and $s' = s_0 - \epsilon'$. 

Recall $X \nearrow$ ($X$ is neat), thus $X(s_0-\epsilon') \leq X(s_0+\epsilon')$. Combining the claims,
$$
\forall \epsilon > 0, \ X^m(s_0) \text{ is eventually in } (-\epsilon + X(s_0 - \epsilon'), X(s_0 + \epsilon') + \epsilon).
$$
If $X$ is continuous at $s_0$, taking $\epsilon, \epsilon' \rightarrow 0$ in this equation gives $X^m(s_0) \rightarrow X(s_0)$. Since 
$X \nearrow$, $X$ has countable discontinuous, i.e.\ is continuous almost everywhere. Therefore $X^m \xrightarrow[]{a.e.} X$ as desired. 

Claim $1$: Fix $0<s_0<s<1$. $\forall \epsilon >0, \ X^m(s_0)$ is eventually in $(-\infty, X(s)+\epsilon)$. 

By contradiction, suppose this was false. Then $\exists (m_k)_k$ such that $\forall k, X^{m_k}(s_0) \geq X(s) + \epsilon$. Because $F_X$ is continuous on a dense subset, 
$\exists \epsilon' \in (0,\epsilon)$ such that $F_X$ is continuous at $X(s) + \epsilon'$. 

Now $X \nearrow$ implies $\{s' \in (0,1) \ \big| \ X(s') \leq X(s)\} \supset (0,s)$.
Applying the Lebesgue measure gives
\begin{equation}\label{neatnessconvergenceeq1}
F_X(X(s)) \geq s.
\end{equation}
Because $X^{m_k} \nearrow$ ($X^{m_k}$ are neat) and by construction of $(m_k)_k$ and choice of $\epsilon'$, 
$$
X^{m_k}(s_0) \geq X(s)+\epsilon > X(s) + \epsilon'.
$$
It follows
$$
\Big\{s' \in (0,1) \ \big| \ X^{m_k}(s) \leq X(s) + \epsilon' \Big\} \subset (0,s_0).
$$
Applying $m$, 
$$ 
F_{X^{m_k}}(X(s) + \epsilon') \leq s_0.
$$
These equations are a problem, because now 
\begin{align*}
s > s_0 \geq \lim_{k \rightarrow \infty} F_{X^{m_k}}(X(s) + \epsilon') \stackrel{(1)}{=} F_X(X(s) + \epsilon') \geq F_X(X(s)) \stackrel{\eqref{neatnessconvergenceeq1}}{\geq} s.
\end{align*}
where $(1)$ is since $\epsilon'$ was chosen such that $F_X$ is continuous at $X(s) + \epsilon'$ and $X^{m_k} \xrightarrow[]{d} X$. A contradiction, so the claim holds.

Claim $2$: Fix $0<s'< s_0 < 1$. $\forall \epsilon >0, \ X^m(s_0)$ is eventually in $(X(s')-\epsilon, \infty)$. 

Because $X^m \xrightarrow[]{d} X$, one has $P(X^m \geq t) \xrightarrow[m \rightarrow \infty]{} P(X \geq t)$ for all $t$ at which $P(X \geq t)$ is continuous (section~\ref{preliminaries}). The rest is similar (symmetric) to claim $1$.
\end{proof}

The purpose of the next lemma is to extract from some $(X^m)_{m \in \mathbb{N}} \subset \mathcal{T}(f), f \in \mathcal{B}^n$, a 
subsequence which converges almost everywhere. This construction will be used when shift operators are introduced to show 
for $U \subset \mathbb{R}^n$ closed,
$$
\exists X_0 \in \mathcal{T}(f), P(X_0 \in U) = \sup_{X \in \mathcal{T}(f)}P(X \in U).
$$

\begin{lemma}[convergent a.e.\ subsequence]\label{conva.e.subseq}
Let $f \in \mathcal{B}, (X^m)_{m \in \mathbb{N}} \subset \mathcal{T}(f)$. Then, $\exists (m_k)_k \subset \mathbb{N}$ and $X \in \mathcal{T}(f)$ 
such that $X^{m_k} \xrightarrow[]{a.e.} X$.
\end{lemma}
\begin{proof}
This will be done in three stages. 

Step $1$: Find $(m_k)_k \subset \mathbb{N}$ such that $F_{X^{m_k}}(t)$ converges $\forall t \in \mathbb{Q}$.

Step $2$: Construct $X$ by naturally extending $F_X(q):= \lim_{k \rightarrow \infty} F_{X^{m_k}}(q), \forall q \in \mathbb{Q}$ to $\mathbb{R}$. 

Step $3$: Show $X^{m_k} \xrightarrow[]{d} X$. Then, by lemma~\ref{existenceneat}, without loss of generality $X$ is 
neat. Then $X \in \mathcal{T}(f)$ by lemma~\ref{preservedconvind}. $X^{m_k} \xrightarrow[]{a.e.} X$ by lemma~\ref{neatnessconvergence}.

We begin now. Step $1$: Let $(q_j)_j$ be an enumeration of $\mathbb{Q}$. Since $\forall m \in \mathbb{N}, \forall q \in \mathbb{Q}$, $0 \leq F_{X^m}(q) \leq 1$, 
one can construct subsequences
$$
(X^m)_m \supset (X^{m_k^1})_k \supset (X^{m_k^2})_k \supset \cdots
$$
such that $\lim_{k \rightarrow \infty} F_{X^j_k}(q_j)$ exists $\forall j \in \mathbb{N}$. Consider $(X^{m_k})_k, m_k = m_k^k, \forall k \in \mathbb{N}$. Then $\forall q \in \mathbb{Q}$, 
$$
F_{X^{m_k}}(q) \xrightarrow[k \rightarrow \infty]{} \text{ (some) } r_q \in [0,1].
$$

Step $2$: Let $G: \mathbb{Q} \rightarrow [0,1], \ G(q) := r_q$. For the rest of the proof $q,q'$ will mean rational numbers. We claim
\begin{enumerate}[label=(\roman*)]
\item $G \nearrow$
\item $G \rightarrow 1$ as $q \rightarrow \infty$
\item $G \rightarrow 0$ as $q \rightarrow -\infty$
\end{enumerate}
(i): $\forall k \in \mathbb{N}, q < q',$ one has $F_{X^{m_k}}(q) \leq F_{X^{m_k}}(q')$. Letting $k \rightarrow \infty$ gives $G(q) \leq G(q')$.\\
(ii), \ (iii): Using $P(|X^m| \geq t) \leq f(t), \forall t \geq 0$, it is trivial to show
\begin{align*}
P(X^m \leq q) &\in [1-f(q),1], \ \forall q \geq 0 \\
P(X^m \leq q) &\in [0,f(|q|)], \ \forall q < 0
\end{align*}
In particular, it is true for $q \geq 0$ that $ 1 \geq G(q) \geq 1-f(q)$. The RHS goes to $1$ as $q \rightarrow \infty$. Similarly for $q < 0, \ 0 \leq G(q) \leq f(|q|)$ which goes to $0$ as $q \rightarrow -\infty$. This concludes the claim.

Define 
$$
F(t) := \lim_{q \searrow t} G(q), \ \forall t \in \mathbb{R}.
$$
This will be the CDF of $X$, but first we must show it satisfies the CDF properties $F \nearrow, F(t) \rightarrow 0$ as 
$t \rightarrow -\infty, F(t) \rightarrow 1$ as $t \rightarrow \infty$, and $F$ is right continuous. 
First, notice since $G \nearrow$ (by (i)), that 
$$
\forall q>t, \ F(t) \leq G(q). 
$$
This will be used multiple times. Now, given $t' > t, \exists q \in \mathbb{Q} \cap [t,t']$. Also, 
$$
F(t) \leq G(q) \leq G(q'), \ \forall q'>t'.
$$
Since $F(t') = \lim_{q' \searrow t'} G(q')$, letting $q' \searrow t'$ here shows $F(t) \leq F(t')$. So $F \nearrow$. Next, since $G(q) \rightarrow 1$
as $q \rightarrow \infty$, by (ii), it follows 
$$
\lim_{t \rightarrow \infty} F(t) = \lim_{t \rightarrow \infty} \lim_{q \searrow t} G(q) = 1.
$$
Similarly, one can argue $F(t) \rightarrow 0$ as $t \rightarrow -\infty$. Lastly, since $F \nearrow$, it follows $\forall q'>t'>t$ that
\begin{equation*}
F(t) \leq F(t') \leq G(q').
\end{equation*}
As $q',t' \searrow t$ the RHS converges to $F(t)$, so $F$ is right continuous at $t$. This shows $F$ is a CDF. So there exists a random variable $X$ with CDF $F$ (section~\ref{preliminaries}).

Step $3$: Suppose $F$ is continuous at $t \in \mathbb{R}$. We must show $F_{X^m}(t) \rightarrow F_X(t)$ as $m \rightarrow \infty$. By definition 
of $F$, and since $F,G$ are monotonically increasing, $\forall t' < q < t < q' < t''$,
\begin{equation}\label{countabledist1}
F(t') \leq G(q) \leq F(t) \leq G(q') \leq F(t'')
\end{equation}

Fix $\epsilon > 0$. By continuity at $t$ suppose $t',t''$ are sufficiently close to $t$ that 
\begin{equation}\label{whateverwhocares}
|F(t) - F(t')|,\ |F(t)-F(t'')|<\epsilon.
\end{equation}
Since $G(q) = \lim_{k \rightarrow \infty} F_{X^{m_k}}(q)$ by definition, $\exists K \in \mathbb{N}$ such that $k\geq K \implies |G(q) - F_{X^{m_k}}(q)| < \epsilon$. There exists a similar $K'$ for $q'$. For $k \geq \max\{K,K'\}$ 
\begin{equation*}
G(q) - \epsilon < F_{X^{m_k}}(q) \leq F_{X^{m_k}}(t) \leq F_{X^{m_k}}(q') < G(q')+\epsilon
\end{equation*}
By \eqref{countabledist1}, $F(t')-\epsilon < F_{X^{m_k}}(t) < F(t'')+\epsilon$. By \eqref{whateverwhocares}, 
$F(t)-2\epsilon < F_{X^{m_k}}(t) < F(t)+2\epsilon$. Since $\epsilon>0$ was arbitrary, $F_{X^{m_k}}(t) \rightarrow F(t)$ as desired.
\end{proof}

\begin{corollary}\label{S2subseqconva.e.}
   Let $f = (f^-,f^+) \in \mathcal{B}_2$ and $(X^m)_{m \in \mathbb{N}} \subset \mathcal{T}_2(f)$. Then $\exists (m_k)_k \subset \mathbb{N}$ and $X \in \mathcal{T}_2(f)$ such that $X^{m_k} \xrightarrow[]{a.e.} X$.
   \end{corollary}
   \begin{proof}
   Define $h(t) := \min \{1,f^+(t)+f^-(-t)\}, \forall t \in [0,\infty)$. $P(X^m \geq t) \leq f^+(t)$ and $P(X^m \leq t) \leq f^-(t)$ imply
   $$
   P(|X^m| \geq t) \leq h(t), \ \forall t \geq 0.
   $$
   $h \in \mathcal{B}$ because $h(0) = 1$, $h \searrow 0$, and the sum of two left continuous functions is left continuous ($f^+(t), f^-(-t)$ are left continuous by definition of $\mathscr{T}_2(f)$ \eqref{B2T2}). So 
   $(X^m)_m \subset \mathcal{T}(h)$. By lemma~\ref{conva.e.subseq}, $\exists$ a subsequence $X^{m_k}_k \xrightarrow[]{a.e.} X$ for some $X \in \mathcal{T}(h)$). By lemma~\ref{preservedconvind}, $X \in \mathcal{T}_2(f)$. So we are done.

\end{proof}

\chapter{Dependent/not necessarily ind.\ real random variables}

For $V \subset \mathbb{R}^n$ closed, the main interest in this chapter is finding $\sup P\big((X_1,...,X_n) \in V\big)$, where the supremum is taken over real r.v.'s 
$X_1,...,X_n$, allowed to be dependent, and absolute tail bounded (section~\ref{tailtypes}) by 
$f_1,...,f_n \in \mathcal{B}$ given by \eqref{absolutetails} respectively. We emphasize the supremum is being taken over r.v.'s allowed to be dependent, but not 
necessarily dependent. Thus it includes the cases where the $X_1,...,X_n$ are actually independent. Therefore this dependent case supremum is an upper bound for the 
ind.\ case supremum. So this chapter is applicable to the mutually ind.\ case, although this application will not give the sharpest tail bound for the ind.\ case. 
On the other hand, because the dependent case is easier to analyze, we are able to provide a formula which is a significant simplification of the dependent supremum.

\section{Notation and rephrasement using measures}\label{notationdep}
In this chapter we will often abbreviate 
\begin{equation}\label{absoluteabbreviationdep}
\{x \in \mathbb{R}^n \ \big| \ |\pi_i(x)| \geq t\} = \{x \in \mathbb{R}^n\}_{|x_i| \geq t}
\end{equation}
where $\pi_i$ denotes the projection on the $i$th component. Similarly, we will abbreviate 
\begin{equation}\label{rightabbreviationdep}
\{x \in \mathbb{R}^n \  | \ \pi_i(x) \geq t \} = \{x \in \mathbb{R}^n\}_{x_i \geq t}
\end{equation}

Recall for a measure $m$ on $\Omega_1$, and $h:\Omega_1 \rightarrow \Omega_2$ measurable, the pushforward 
measure $h_{\#}(m)$ on $\Omega_2$ is given by 
\begin{equation}\label{pushforwarddef}
h_{\#}(U) = m(h^{-1}(U)), \  \forall U \text{ in the } \sigma \text{-algebra of } \Omega_2.
\end{equation} 
Notice if $h = X$, then $m(X^{-1}(U)) = P(X \in U)$. 

Let $Q^1,...,Q^{2^n}$ be a listing of the closed quadrants of $\mathbb{R}^n$ with $Q^1 = [0,\infty)^n$.  Let 
\begin{equation}\label{Qdef}
   Q(x_1,...,x_n) := (|x_1|,...,|x_n|), \ \forall x \in \mathbb{R}^n.
\end{equation}
These notations will be used often in this chapter.

Consider $\sup P\big((X_1,...,X_n) \in V\big)$, where the supremum is taken over real r.v.'s $X_1,...,X_n$, allowed to be dependent, and absolute tail bounded (section~\ref{tailtypes}) by 
$f_1,...,f_n \in \mathcal{B}$ \eqref{absolutetails} respectively. We will simplify the computation of this supremum in a few ways. First, for the sake of making the proofs easier, we will rephrase the supremum in terms of measures (this is easy to reverse\footnote{One can let $X$ be the identity function on $\mathbb{R}^n$.}); 
in particular, the collection of measures the r.v.'s induce on $\mathbb{R}^n$ via pushforward.  This is done now. Define 
\begin{equation}\label{Td}
   \begin{split}
   \mathcal{T}^d(f) := &\{m \ \text{a measure on the Borel} \ \sigma\text{-algebra of} \ \mathbb{R}^n \ \text{such that} \\  &m \big(\{x \in \mathbb{R}^n\}_{|x_i| \geq t}\big) \leq f_i(t), \ \forall t \geq 0, \ i = 1,...,n \}, 
   \ \forall f \in \mathcal{B}^n.
   \end{split}
\end{equation}
where $f=(f_1,...,f_n)$ is notation we will use. 

\begin{lemma}
For all $U \subset \mathbb{R}^n$ measurable and $f \in \mathcal{B}^n$, 
\begin{equation}\label{rephrasementwithmeasureseq}
\sup_{\substack{ {\rm dep. \ } X_i, {\rm \ absolute \ tail} \\ {\rm bounded \ by }f_i, \ i = 1,...,n}} P(X \in U) = \sup_{m \in \mathcal{T}^d(f)} m(U).
\end{equation}
\end{lemma}
\begin{proof} 
LHS, RHS will refer to \eqref{rephrasementwithmeasureseq}. 

To show LHS $\leq$ RHS: suppose $X = (X_1,...,X_n)$ satisfies the conditions of the LHS. $X$ induces a pushforward measure $m = X_\#(\mu)$ on $\mathbb{R}^n$, where $\mu$ is the measure on the domain of $X$, such that 
\begin{equation}\label{deptails}
m\big(\{x \in \mathbb{R}^n\}_{|x_i| \geq t} \big) \stackrel{\eqref{pushforwarddef}}{=} P\big((X_1,...,X_n) \in \{x \in \mathbb{R}^n\}_{|x_i| \geq t}\big), \ \forall t \geq 0, \ i = 1,...,n. 
\end{equation}
Now, $(X_1,...,X_n) \in \{x \in \mathbb{R}^n\}_{|x_i| \geq t}$ implies $|X_i| \geq t$. However, by hypothesis, $P(|X_i| \geq t) \leq f_i(t), \ \forall t \geq 0$. In particular 
$P\big((X_1,...,X_n) \in \{x \in \mathbb{R}^n\}_{|x_i| \geq t}\big) \leq f_i(t)$. Thus
\begin{equation}\label{deptails} 
   m\big(\{x \in \mathbb{R}^n\}_{|x_i| \geq t} \big) \leq f_i(t), \ \forall t \geq 0, \ i=1,...,n.
\end{equation} 
That is, $m \in \mathcal{T}^d(f)$. $P(X \in U) = m(U)$ trivially. It follows LHS $\leq$ RHS. 

We will now show LHS $\geq$ RHS. Suppose $m \in \mathcal{T}^d(f)$ (that is, $m$ satisfies \eqref{deptails}). Because $f_i(0) = 1$ \eqref{absolutetails} and $\{ x \in \mathbb{R}^n\}_{|x_i| \geq 0} = \mathbb{R}^n$, one finds 
$m(\mathbb{R}^n) \leq 1$. Introduce the measure $m_0$ on $\mathbb{R}^n$ such that for measurable $A \subset \mathbb{R}^n$, 
$$
m_0(A) = 
\begin{cases} 
m(A) & \text{if } 0 \notin A \\
m(A) + 1-m(\mathbb{R}^n) & \text{if } 0 \in A
\end{cases} 
$$
This way $m_0(\mathbb{R}^n) = 1$. Furthermore, $m_0$ still satisfies \eqref{deptails}, i.e.\ $m_0 \in \mathcal{T}^d(f)$.

Consider the r.v.\ $X = (X_1,...,X_n): x \mapsto x, \ \forall x \in \mathbb{R}^n$, such that the domain has measure $m_0$. Note $X_i = \pi_i \circ X$ is the identity function on $\mathbb{R}$. Furthermore, $\forall t \geq 0$,
\begin{align*} 
P(|\pi_i \circ X| \geq t) &= P\big(X \in \{x \in \mathbb{R}^n\}_{|x_i| \geq t}\big) \\ 
                          &= m_0\big(\{x \in \mathbb{R}^n\}_{|x_i| \geq t}\big) \\ 
                          &\stackrel{m_0 \in \mathcal{T}^d(f)}{\leq} f_i(t).
\end{align*} 
That is, $X$ satisfies the conditions of the LHS. Now, if $0 \in U$ notice the r.v.\ $Y: \Omega \rightarrow \mathbb{R}^n$ such that $P(Y \in A) = 1$ if $0 \in A$ and $P(Y \in A) = 0$ if $0 \notin A$ for every measurable $A \subset \mathbb{R}^n$ 
trivially satisfies the conditions of the LHS and has $P(Y \in U) = 1$. Thus LHS $\geq 1$. Clearly LHS $\leq 1$ so LHS $= 1$. Similarly, the measure $\mu$ on $\mathbb{R}^n$ such that $\mu(A) = 1$ if $0 \in A$ and $\mu(A) = 0$ if $0 \notin A$ 
satisfies the conditions of the RHS and has $\mu(U) = 1$. So RHS $=1$ too. So LHS $=$ RHS.  

If instead $0 \notin U$, then 
$$
P(X \in U) = m_0(X^{-1}(U)) \stackrel{(1)}{=} m_0(U) \stackrel{(2)}{=} m(U)
$$
where $(1)$ is since $X$ is the identity function on $\mathbb{R}^n$ and $(2)$ since $m_0(V) = m(V)$ for all sets not containing $0$. This shows LHS $\geq$ RHS. Therefore \eqref{rephrasementwithmeasureseq} holds.
\end{proof}

\eqref{rephrasementwithmeasureseq} will be simplified in two more ways in this chapter. First we will show for all $V \subset \mathbb{R}^n$ closed, that  
$$
\sup_{m \in \mathcal{T}^d(f)} m(V) = \sup_{m \in \mathcal{T}^d(f)} m(Q(V)), \ \forall  V \subset \mathbb{R}^n \text{ closed.}
$$
where $Q:(x_1,...,x_n) \mapsto (|x_1|,...,|x_n|)$ is from \eqref{Qdef}. Then, for any closed subset $V \subset [0,\infty)^n$, we will define the \textit{southwest boundary} $\partial_{SW}V$ of $V$ which, roughly speaking, 
consists of the points of $\partial V$ closest to the origin (see \eqref{southwestboundaryeq}). We will then show 
$$
\sup_{m \in \mathcal{T}^d(f)} m(V) = \sup_{m \in \mathcal{T}^d(f)} m\big(\partial_{SW} V\big)
$$
Successively applying these simplifications will give
\begin{equation}\label{SWQ(V)WLOG}
\sup_{m \in \mathcal{T}^d(f)} m(V) = \sup_{m \in \mathcal{T}^d(f)} m\big(\partial_{SW} Q(V)\big), \ \forall V \subset \mathbb{R}^n \text{ closed.} 
\end{equation}
\section{Without loss of generality $V \subset [0,\infty)^n$}

Recall $Q^1,...,Q^{2^n}$ is a listing of the (closed) quadrants of $\mathbb{R}^n$ with $Q^1 = [0,\infty)^n$.
\begin{lemma}\label{Qlemma}
Assume notations \eqref{pushforwarddef} and \eqref{Td}. $Q$ given by \eqref{Qdef} is measurable, and if $m \in \mathcal{T}^d(f)$ then $Q_\# (m) \in \mathcal{T}^d(f)$. Furthermore, if $U \subset \mathbb{R}^n$ is measurable, then so is $Q(U)$. 
Lastly, if $V \subset \mathbb{R}^n$ is closed then $Q(V)$ is closed. 
\end{lemma}
\begin{proof}
$Q$ is continuous in each component, hence continuous, hence Borel measurable. Now, $\forall t \geq 0$, 
\begin{align*} 
Q_\# (m) \big(\{x \in \mathbb{R}^n\}_{|x_i| \geq t}\big) &\stackrel{\eqref{pushforwarddef}, \eqref{absoluteabbreviationdep}}{=} m\left(Q^{-1}\big(\{x \in \mathbb{R}^n\}_{|x_i| \geq t}\big) \right) \\ 
&\ \ \ \stackrel{(1)}{=} m\big(\{x \in \mathbb{R}^n\}_{|x_i| \geq t}\big) \\
&\ \ \ \stackrel{(2)}{\leq} f_i(t), \ i=1,...,n
\end{align*} 
where $(1)$ follows since $Q^{-1}\big( \{x \in \mathbb{R}^n\}_{|x_i| \geq t}\big) = \{x \in \mathbb{R}^n\}_{|x_i| \geq t}$ and $(2)$ holds by definition of $\mathcal{T}^d(f)$ \eqref{Td}. So $Q_\# \in \mathcal{T}^d(f)$. 

Now we will prove that $Q(U)$ is measurable. Notice 
\begin{equation}\label{Qismeasurableeq}
Q(U) = Q(U \cap Q^1) \cup \cdots \cup Q(U \cap Q^{2^n}).
\end{equation}
Furthermore, for fixed $k \in \{1,...,2^n\}$, all elements of $U \cap Q^k$ have the same sign of their $i$th component, for $i = 1,...,n$. In particular, by \eqref{Qdef}, $Q(U \cap Q^k)$ is a rigid motion of $U\cap Q^k$. Because rigid motions 
of Borel measurable sets are measurable, $Q(U \cap Q^k)$ is measurable. $k$ was arbitrary, so $Q(U)$ is measurable by \eqref{Qismeasurableeq}.

Now, let $V \subset \mathbb{R}^n$ be closed and $(x^k)_k \subset Q(V), \ x^k \rightarrow y$. Notice there exists a convergent subsequence in $Q^{-1}(\{x^k\}_k) \cap V$ which converges to some $y'$. By closure, $y' \in V$. Clearly $Q(y') = y$, so $y \in Q(V)$. So $Q(V)$ is closed.   
\end{proof}

The next lemma is known facts of measure theory and will help us manipulate measures. 
\begin{lemma}\label{addmeasures}
Let $m_1,m_2$ be measures on $\Omega$ with $\sigma$-algebra $\mathcalorig{F}$. The following are also measures on $\Omega$. 
\begin{enumerate}[label=(\roman*)]
\item $m_1+m_2$ given by pointwise addition (of sets).
\item For $\alpha \geq 0$, $\alpha m_1$ given by $(\alpha m_1)(A) = \alpha m_1(A)$.
\item For $m_1,m_2$ finite and $\forall U \in \mathcalorig{F}$, $m_1(U) \geq m_2(U)$, $m_1 - m_2$ given by pointwise subtraction.
\item For $U \in \mathcalorig{F}$, the restriction of $m_k$ to $U$. That is, $m_k\big|_U(A) := m_k(A \cap U), \forall A \in \mathcalorig{F}$.
\end{enumerate}
\end{lemma}
\begin{proof}
Checking $\sigma$-additivity is elementary and very routine, albeit tedious. The other conditions are clear. We will 
verify $\sigma$-additivity for $4$. Let $U_1,U_2,... \in \mathcalorig{F}$ be pairwise disjoint. Then,
\begin{align*}
m_k\Big|_U \Big(\bigcup_k U_k \Big) &= m_k \bigg( \Big( \bigcup_k U_k\Big) \cap U \bigg) \\ 
&= m_k \bigg(\bigcup_k (U_k \cap U) \bigg) \\
&= \sum_k m_k(U_k\cap U) \ \ \ \ \ \ \ \ \ \ \ \ \ \ \ (\sigma\text{-additivity of } m_k)\\
&= \sum_k m_k\big|_U (U_k).
\end{align*}
\end{proof}
\begin{lemma}\label{l4.a}
   Let $f \in \mathcal{B}^n$. For $U \subset \mathbb{R}^n$ measurable, 
   \begin{equation}\label{Q(U)WLOG}
   \sup_{m \in \mathcal{T}^d(f)} m(U)  = \sup_{m \in \mathcal{T}^d(f)} m(Q(U)),
   \end{equation}
where $Q$ is defined by \eqref{Qdef}.
\end{lemma}
\begin{proof}
LHS, RHS will refer to \eqref{Q(U)WLOG}. For arbitrary $m \in \mathcal{T}^d(f)$, we will construct measures $m', m''$ such that

\begin{align}
   m'(Q(U)) &= m(U) \nonumber \\ 
   m''(U) &= m(Q(U)) \label{m''(U)=m(Q(U))}\\
   m', m'' & \in \mathcal{T}^d(f).     \label{m''tailbounded}
\end{align}
This is sufficient to show LHS $=$ RHS, since if $(m_k)_k \subset \mathcal{T}^d(f)$ such that $m_k(U) \rightarrow $ LHS, then 
$$
\text{LHS }= \lim_{k\rightarrow \infty} m_k(U) = \lim_{k \rightarrow \infty} m_k'(Q(U)) \leq \text{ RHS}.
$$
Similarly, RHS $ \leq$ LHS via $m''$. 

To construct $m'$: Consider the measure $m' := Q_\#(m \big|_U)$ on $\mathbb{R}^n$ (lemma~\ref{addmeasures} and \eqref{pushforwarddef}) characterised by $Q_\#(m \big|_U)(A) = m(Q^{-1}(A) \cap U)$ for $A$ measurable. Now,
\begin{align*}
   \forall t \geq 0, \ Q_\#\left(m_k \big|_U\right)\big(\{x \in \mathbb{R}^n\}_{x_i \geq t}\big) &= m_k \left(Q^{-1}\big(\{x \in \mathbb{R}^n\}_{x_i \geq t} \cap U\big) \right) \\
   & \leq m_k\left(Q^{-1}\big(\{x \in \mathbb{R}^n\}_{x_i \geq t}\big) \right) \\ 
   &\stackrel{\eqref{Qdef}}{=} m_k \big(\{x \in \mathbb{R}^n\}_{|x_i| \geq t}\big) \stackrel{\eqref{deptails}}{\leq} f_i(t).
   \end{align*}
So $m' \in \mathcal{T}^d(f)$. Furthermore, 
\begin{align*} 
   m'(Q(U)) &= Q_\#(m\big|_U)(Q(U)) \\ 
   &= m\big|_U(Q^{-1}(Q(U)) \\
   &= m (Q^{-1}(Q(U)) \cap U) \\ 
   &= m(U)
\end{align*}
as desired.

To construct $m''$: Still let $Q^1,...,Q^{2^n}$ be an enumeration of the (closed) quadrants of $\mathbb{R}^n$ with $Q^1 = [0,\infty)^n$. Since 
$Q(U) \subset Q^1$, and we want to construct $m'' \in \mathcal{T}^d(f)$ such that $m(Q(U)) = m''(U)$, we may assume WLOG
\begin{equation}\label{WLOGm=0}
m(\mathbb{R}^n \backslash Q^1) = 0.
\end{equation}
For $j=1,...,2^n$ define $Q_j: [0,\infty)^n \rightarrow \mathbb{R}^n$ of the form $(x_1,...,x_n) \mapsto (\pm x_1,..., \pm x_n)$ for some fixed choice of signs as follows.
\begin{align}\label{eq.Q_j}
   &Q_j: [0,\infty)^n \rightarrow \mathbb{R}^n \\ 
      &Q_j: x \mapsto \text{unique } y \in Q^j \text{ with } Q(y) = x.
\end{align}
Clearly each $Q_j$ is continuous, hence measurable. Notice that
\begin{equation}\label{eq.Q_j2}
Q_j^{-1} = Q \big|_{Q^j}.
\end{equation}

Introduce
\begin{align*}
   U_1 &:= U \cap Q^1 \\
   U_2 &:= (U \cap Q^2) \backslash Q^{-1}(Q(U_1)) \\
   U_3 &:= (U \cap Q^3) \backslash ( Q^{-1}(Q(U_1)) \cup Q^{-1}(Q(U_2)))\\ 
   \cdots& \\\
   U_{2^n} &:= (U \cap Q^{2^n}) \backslash (\cup_{i=1}^{2^n-1} Q^{-1}(Q(U_i))).
\end{align*}
We will show $U_j$ have the properties
\begin{enumerate}[label=(\roman*)]
   \item $U_j \subset U \cap Q^j$
   \item $Q(U_j) = Q_j^{-1}(U_j)$ are pairwise disjoint
   \item $U_j$ are measurable 
   \item $Q\big(\bigcup_{j=1}^{2^n} U_j \big) = Q(U)$
\end{enumerate}
Recall \eqref{WLOGm=0}. By (i), which holds trivially, the ``reversal'' of $m$ into only the $j$th quadrant can be defined 
\begin{equation}\label{m^{(j)}def}
m^{(j)} := \big({Q_j}_\#(m) \big) \big|_{U_j}.
\end{equation}
That is, $m^{(j)}(A) = m(Q_j^{-1}(A \cap U_j))$. Next define, using lemma~\ref{addmeasures} (i) repeatedly ($2^n-1$ times to be exact),
\begin{equation}\label{m''def}
m'' := m^{(1)} + \cdots + m^{(2^n)}.
\end{equation}
We have assumed associativity in this operation of adding measures (associativity should not be difficult to prove). That is, $(m_1+m_2)+m_3 = m_1+(m_2+m_3)$ 
for three measures $m_1,m_2,m_3$. We will now prove (ii)-(iv). Afterwards we will prove this $m''$ does indeed satisfy \eqref{m''(U)=m(Q(U))} and \eqref{m''tailbounded}.

(ii): By \eqref{eq.Q_j2} and $U_j \subset Q^j$, $Q_j^{-1}(U_j) = Q\big|_{Q^j} (U_j) = Q(U_j)$. Now, if $1 \leq i < j \leq 2^n$, then $Q^{-1}(Q(U_i)) \cap U_j = \emptyset$ by construction. This says $\nexists x \in U_j$ such that 
$Q(x) \in Q(U_i)$, i.e. $Q(U_i)$ and $Q(U_j)$ are disjoint.

(iii): $U\cap Q^j$ are clearly measurable. By induction assume $U_1,...,U_m$ are measurable where $1 < m < 2^n$. By lemma~\ref{Qlemma}, $Q^{-1}(Q(U_1)) \cup \cdots \cup Q^{-1}(Q(U_m))$ is too. 
It follows $U_{m+1}$ is measurable. 

(iv): Since $\bigcup_{j=1}^{2^n} U_j \subset U$ trivially, this reduces to showing 
\begin{equation}\label{xsi4}
\forall x \in U, \ \exists x' \in \bigcup_{j=1}^{2^n} U_j, \text{ such that } Q(x') = Q(x). 
\end{equation}
Fix $x \in U$. Necessarily $Q^{-1}(Q(x)) \cap (U\cap Q^j) \neq \emptyset$ for some $j$, otherwise $Q^{-1}(Q(x)) \cap U = \emptyset$ because $Q^1 \cup \cdots \cup Q^n = \mathbb{R}^n$. A contradiction. 
Denote
$$ 
j_x \text{ the smallest $j$ such that } Q^{-1}(Q(x)) \cap (U\cap Q^j) \neq \emptyset.
$$
Let
\begin{equation}\label{xsi5}
x' \in Q^{-1}(Q(x))\cap(U \cap Q^{j_x}).
\end{equation}
Notice $Q(x') = Q(x)$. Therefore, to prove \eqref{xsi4}, we need show $x' \in U_j$ for some $j$. 

Case $j_x = 1$: $U_1 \stackrel{\text{(def)}}{=} U \cap Q^1$, so by \eqref{xsi5} $x' \in U_1$.

Case $j_x \neq 1$: It is sufficient to show 
\begin{equation}\label{xsi6}
x' \notin \bigcup_{j=1}^{j_x-1} Q^{-1}(Q(U_j))
\end{equation}
since then, by \eqref{xsi5}, $x' \in (U \cap Q^{j_x}) \backslash \bigcup_{j=1}^{j_x-1} Q^{-1}(Q(U_j)) \stackrel{\text{(def)}}= U_{j_x}$. 

Suppose by contradiction \eqref{xsi6} is false. Then $\exists i<j_x$ such that $x' \in Q^{-1}(Q(U_i))$. But then, $Q(x) = Q(x') \in Q(U_i) \implies \exists x'' \in U_i \stackrel{\text{(i)}}{\subset} U \cap Q^i$ 
such that $Q(x'') = Q(x)$. That is, 
$$
x'' \in Q^{-1}(Q(x)) \cap (U \cap Q^i), \ i < j_x
$$
which contradicts the minimality of $j_x$. This shows \eqref{xsi6}, hence \eqref{xsi4}.

This concludes proving (i)-(iv). We will now prove \eqref{m''(U)=m(Q(U))} and \eqref{m''tailbounded}, which will prove the lemma.

To show $m''(U) = m(Q(U))$ \eqref{m''(U)=m(Q(U))}: Here
\begin{align*}
m''(U) &\stackrel{\eqref{m''def}}{=} \sum_{j=1}^{2^n} m^{(j)}(U) \\ 
&\stackrel{\eqref{m^{(j)}def}}{=} \sum_{j=1}^{2^n} m(Q_j^{-1}(U\cap U_j)) \\ 
&\  \stackrel{\text{(i)}}{=} \sum_{j=1}^{2^n} m(Q_j^{-1}(U_j)).
\end{align*}
Furthermore, 
\begin{align*} 
   \sum_{j=1}^{2^n} m(Q_j^{-1}(U_j)) &\stackrel{\sigma\text{-add.\ and (ii)}} {=} m\Big(\bigcup_{j=1}^{2^n} Q(U_j)\Big) \\ 
   &\ \ \ \ \ \ = m\Big( Q \Big( \bigcup_{j=1}^{2^n} U_j\Big) \Big) \\ 
   &\ \ \ \ \ \ \stackrel{\text{(iv)}}{=} m(Q(U))
\end{align*}
as desired.

To show $m'' \in \mathcal{T}^d(f)$ \eqref{m''tailbounded}: We must show $m''$ satisfies \eqref{Td}. $\forall t \geq 0, 1 \leq k \leq n$, one has the sequence
\begin{align*}
m''\big(\{x \in \mathbb{R}^n\}_{|x_k| \geq t}\big) &\stackrel{\eqref{m''def}}{=} \sum_{j=1}^{2^n} m^{(j)} \big(\{x \in \mathbb{R}^n\}_{|x_k| \geq t}\big) \\
&\stackrel{\eqref{m^{(j)}def}}{=} \sum_{j=1}^{2^n} m\Big(Q_j^{-1}\Big(\{x \in \mathbb{R}^n\}_{|x_k| \geq t} \cap U_j\Big) \Big)  \\
&\stackrel{\text{(i),\eqref{eq.Q_j2}}}{=} \sum_{j=1}^{2^n} m\Big(Q\Big(\{x \in \mathbb{R}^n\}_{|x_k| \geq t} \cap U_j\Big) \Big) \\
&\stackrel{\sigma \text{ add, (ii)}}{=} m \Big(\bigcup_{j=1}^{2^n} Q\Big(\{x \in \mathbb{R}^n\}_{|x_k| \geq t} \cap U_j\Big) \Big) \\
& = m \Big(Q\Big(\{x \in \mathbb{R}^n\}_{|x_k| \geq t} \cap \bigcup_{j=1}^{2^n} U_j\Big) \Big) \\
& \ \leq m\Big(Q\Big(\{x \in \mathbb{R}^n\}_{|x_k| \geq t}\Big)\Big) \leq m \big(\{x \in \mathbb{R}^n\}_{|x_k| \geq t}\big) \leq f_k(t)
\end{align*}
where the last line used monotonicity of $m$ twice (notice $Q\big(\{x \in \mathbb{R}^n\}_{|x_k| \geq t}\big) \subset \{x \in \mathbb{R}^n\}_{|x_k| \geq t}$), and $m \in \mathcal{T}^d(f)$ \eqref{Td} by hypothesis. 
$k$ was arbitrary, so $m'' \in \mathcal{T}^d(f)$ as desired.
\end{proof}
   
\section{Southwest boundary of $V \subset [0,\infty)^n$}
In this section and the next $V$ will be a closed subset of $[0,\infty)^n$. Introduce $\leq'$ on $[0,\infty)^n$ such that 
\begin{equation}\label{leq'def}
\forall x =(x_1,...,x_n), y=(y_1,...,y_n) \in [0,\infty)^n, \ x\leq ' y \iff x_i \leq y_i, \ i =1,...,n.
\end{equation}
\index{Southwest boundary $\partial_{SW}V$}
\begin{definition}[Southwest boundary]
For $V \subset [0,\infty)^n$ closed, the southwest boundary of $V$, $\partial_{SW} V$, is the set of minimal elements of $V$ using $\leq'$. i.e. 
\begin{equation}\label{southwestboundaryeq}
\partial_{SW} V := \big\{x \in V \ \big| \ \nexists x' \in V, x' <' x \big\}.
\end{equation}
\end{definition}

\begin{lemma}\label{southwestboundaryproperties}
Let $V \subset [0,\infty)^n$ be closed. $\partial_{SW} V$ satisfies the following. 
\begin{enumerate}[label=(\roman*)]
\item $\forall x \in V, \ \exists x' \in \partial_{SW}V$ such that $x' \leq' x$. 
\item $\partial_{SW}V \subset \partial V$ 
\item $\partial_{SW}V$ is Borel measurable
\end{enumerate}
\end{lemma}
\begin{proof}
(i): Let $A = \{y \in V \ | \ y \leq' x\}$. Because $x \in A, \ A$ is non-empty. Notice $A = V \cap \{y \in \mathbb{R}^n \ | \ y \leq' x\}$, which is the intersection of two closed sets. Hence $A$ is closed. 

First we will apply Zorn's lemma to show $A$ has a minimal element. Let $\{x^i\}_{i \in I}, \\ x^i \leq' x^j \iff i \leq j$ be a chain in $A$. Write $x^i = (x_1^i,...,x_n^i), \ \forall i \in I$. Then 
$$
y = (\inf_{i \in I} x_1^i,...,\inf_{i\in I} x_n^i) \in A
$$
by closedness of $A$ and since each component is bounded below by $0$. Because $y$ bounds $\{x^i\}_{i\in I}$, by Zorn's lemma $A$ has a minimal element $x' \leq' x$. Suppose, by contradiction, $x' \notin \partial_{SW} V$. Then 
$\exists x'' \in V, \ x'' \leq' x'$. But then $x'' \in A$, which contradicts the minimality of $x'$. So $x' \in \partial_{SW} V$ as desired.

(ii): $\forall x = (x_1,...,x_n) \in \partial_{SW} V, \ \forall \epsilon > 0$, one has $(x_1 - \epsilon,....,x_n - \epsilon) \notin V$ by \eqref{southwestboundaryeq}.  
$x \in V$ by definition. In particular $x \in V$. However $x$ is not in the interior of $V$ because as we pointed out the sequence $(x_1-1/m,...,x_n-1/m)$ from $m=1$ to $\infty$ is not in $V$, but converges to an element of $V$. Therefore $x \in \partial V$.  

(iii): This was proven on math stack exchange by Jonathan Schilhan and uses descriptive set theory. Their proof is repeated below.

$\partial_{SW} V$ is easily $G_\delta$, and thus Borel. Indeed, the set $C = \{ (x,y)\in V\times V : x <' y \}$ is $\sigma$-compact, since $$C = A \cap \bigcup_{i \leq n} \{ (x,y) \in V\times V : \pi_i(x) < \pi_i(y) \}.$$ 
where $A = \bigcap_{k=1}^n \{(x,y) : x,y \in [0,\infty)^n, \pi_k(x) \leq \pi_k(y)\}$. $A$ is the intersection of $n$ closed subsets, hence closed, and the set on the right is open (in $V\times V$) and thus $\sigma$-compact (in a Polish space open sets are $F_\sigma$, and if the space is $\sigma$-compact, then $F_\sigma$ sets are obviously 
also $\sigma$-compact). Thus its projection to the second coordinate $V' = \{ y \in V : \exists x \in V,  \ x <' y \}$ is also $\sigma$-compact (remember that the projection of a compact set is compact). Also 
the complement $\partial_{SW} V = V \setminus V'$ is $G_\delta$.

\end{proof}
\section{Retracting $V \subset [0,\infty)^n$ into the SW boundary}
\begin{definition}[Mass retracts]\label{massretractdef}
   A mass retract \index{Mass retract} is a measurable function $\gamma$ on $\mathbb{R}^n$ such that 
   \begin{align}
      0 \leq' \gamma(x) \leq' x, \ &\forall x \in [0,\infty)^n \label{massretracteq1} \\
      \gamma(x) = x, \  \ &\forall x \notin [0,\infty)^n. \label{massretracteq2}
   \end{align} 
   A mass retract $\gamma$ is said to be of $U$, for $U \subset \mathbb{R}^n$ measurable, iff $\gamma(U) \subset U$ and $\forall x \notin U, \ \gamma(x) = x$. 
\end{definition}
\begin{lemma}\label{massretractproperties}
   Mass retracts have the following properties. 
   \begin{enumerate}[label=(\roman*)]
   \item  If $\gamma$ is a mass retract of $U$, then $\forall$ measures $m, \gamma_\#(m)(U) = m(U)$.
   \item Preserve tail bounds: if $m \in \mathcal{T}^d(f)$, then $\gamma_\#(m) \in \mathcal{T}^d(f)$ too. 
   \item Closed under composition. Also, if $\gamma_1,\gamma_2$ are mass retracts of $U$, then so is $\gamma_2 \circ \gamma_1$.
   \end{enumerate}
\end{lemma}
   \begin{proof}
   (i): $(\gamma_\#m)(U) \stackrel{\text(def.)}{=} m(\gamma^{-1}(U)) = m(U)$ where we know by definition of $\gamma$ being a mass retract of $U$ that 
   $\gamma^{-1}(U) = U$ (since $\gamma(U)\subset U$ and $\forall x \notin U, \gamma(x)=x)$.\\
   (ii): Write $Q^1 = [0,\infty)^n$. If $m\{|x_k| \geq t\} \leq f_k(t), \forall t \geq 0$, where $f_k \in \mathcal{B}$, then 
   \begin{align}
   {Q^1}^c \cap \gamma^{-1} \{|x_k|\geq t\} &\stackrel{\eqref{massretracteq2}}{=} \gamma^{-1}({Q^1}^c) \cap \gamma^{-1}\{|x_k|\geq t\} \nonumber \\ 
   & \ \ = \gamma^{-1}({Q^1}^c \cap \{|x_k| \geq t\}) \stackrel{\eqref{massretracteq2}}{=} {Q^1}^c \cap \{|x_k| \geq t\}. \label{l4.3e1}  
   \end{align}
   Also since $\gamma^{-1}(Q^1) = Q^1$ by \eqref{massretracteq1},
   \begin{align}
   Q^1 \cap \gamma^{-1}\{|x_k| \geq t \} &= \gamma^{-1}(Q^1) \cap \gamma^{-1}\{ |x_k| \geq t\} \nonumber \\
   &= \gamma^{-1}(Q^1 \cap \{|x_k| \geq t\}) \stackrel{(1)}{\subset} Q^1 \cap \{|x_k| \geq t\}. \label{l4.3e2}
   \end{align} 
   To show (1): It is enough to show if $x \notin Q^1 \cap \{|x_k| \geq t\}$ then $\gamma(x) \notin Q^1 \cap \{|x_k| \geq t\}$. 
   Either $x \in {Q^1}^c$, in which case $\gamma(x) = x \notin Q^1 \cap \{|x_k| \geq t\}$, 
   or $x \in \{x' \in Q^1 \ \big| \ |\pi_k(x')| < t\}$, in which case,
   $$
   0 \leq \pi_k \big (\gamma(x)\big) \stackrel{\eqref{massretracteq1}, \eqref{leq'def}}{\leq} \pi_k(x) < t.
   $$
   That is, $\gamma(x) \notin \{|x_k| \geq t\}_{x \in \mathbb{R}^n}$. This shows (1). By $\sigma$-additivity, \eqref{l4.3e1}, and \eqref{l4.3e2},
   \begin{align*}
   \gamma_\# (m)\{|x_k| \geq t\} &= m ( \gamma^{-1} \{|x_k| \geq t\}) \\
   &=m(Q^1 \cap \gamma^{-1}\{|x_k| \geq t\}) + m({Q^1}^c \cap \gamma^{-1}\{|x_k| \geq t\}) \\ 
   &\leq m(Q^1 \cap \{|x_k| \geq t\}) + m({Q^1}^c \cap \{|x_k| \geq t\}) \\
   &= m\{|x_k| \geq t\} \leq f_k(t)
   \end{align*}
   as desired.
   
   (iii): For $\gamma_1,\gamma_2$, we have $\gamma_2 \circ \gamma_1$ is measurable since a composition of measurable functions is measurable. Also 
   $(\gamma_2 \circ \gamma_1)(x) \leq' \gamma_1(x) \leq' x$.  If $\gamma_1,\gamma_2$ are retracts of $U$ then 
   $$
   (\gamma_2 \circ \gamma_1)(U) \subset \gamma_2(U) \subset U
   $$
   and $\forall x \notin U, \gamma_2(\gamma_1(x)) = \gamma_2(x) = x$. So $\gamma_2 \circ \gamma_1$ is a retract of $U$.
\end{proof}

The circus of mass retracts has been to prove the following.
\begin{theorem} 
\begin{equation}\label{depsupcircus}
\sup_{m \in \mathcal{T}^d(f)} m(V)= \sup_{m \in \mathcal{T}^d(f)} m\big(\partial_{SW} V\big), \ \forall V \subset [0,\infty)^n \text{ closed}.
\end{equation}
\end{theorem} 
To understand this, \index{Southwest mass retract $\gamma_{SW}$} suppose there existed a mass retract $\gamma_{SW}$ of $V$ such that $\gamma_{SW}(V) \subset \partial_{SW} V$. Firstly $V \supset \partial_{SW} V$ 
(lemma~\ref{southwestboundaryproperties}). It follows $\forall m \in \mathcal{T}^d(f), \ m(V) \geq m(\partial_{SW} V)$. Thus LHS $\geq$ RHS in \eqref{depsupcircus}.

To prove the converse, given $m \in \mathcal{T}^d(f)$ without loss of generality $m(V^c) = 0$. By lemma~\ref{massretractproperties} (ii), 
${\gamma_{SW}}_\#(m) \in \mathcal{T}^d(f)$. Furthermore,
\begin{equation*}
   {\gamma_{SW}}_\#(m)\big(\partial_{SW}V\big) \stackrel{\text{(a)}}{=} {\gamma_{SW}}_\#(m)(V) = m\big( \gamma_{SW}^{-1}(V)\big) = m(V)
\end{equation*}
where (a) follows since $\gamma_{SW}^{-1}(\partial_{SW}V)=V$ by $\gamma_{SW}(V)\subset\partial_{SW}V$ and $\gamma_{SW}^{-1}(V)\cap V^c=\emptyset$. 
That is, $\gamma_{SW}^{-1}(V) = V$. $m$ was arbitrary so LHS $\leq$ RHS. So \eqref{depsupcircus} follows from the existence of $\gamma_{SW}$. 

We will begin constructing this $\gamma_{SW}$. 

\begin{definition} 
For $V \subset [0,\infty)^n$ closed, define 
\begin{equation}\label{Vtildedef} 
   \widetilde{V} = \bigcup_{x \in V} \{x' \in \mathbb{R}^n \ | \ x' \geq' x \}.  
\end{equation} 
\end{definition} 
Notice $V \subset \widetilde{V} \subset [0,\infty)^n$ trivially. 
\begin{lemma}\label{Vtildeproperties} 
$\widetilde{V}$ is closed and $\partial_{SW} \widetilde{V} = \partial_{SW} V$. 
\end{lemma} 
\begin{proof} 
First we will prove closure. Let $(x^l)_l \subset \widetilde{V}, \ x^l \rightarrow x^\infty$. By \eqref{Vtildedef} $\exists (y^l)_l \subset V$ with $0 \leq' y^l \leq' x^l, \ \forall l \in \mathbb{N}$. Because 
the RHS converges, $(y^l)_l$ is clearly contained in a sufficiently large closed box, hence has a convergent subsequence $y^{l_k} \rightarrow y^\infty$. $y^\infty \in V$ due to closedness. We claim 
$$
y^{\infty} \leq' x^\infty. 
$$
Indeed, $\pi_i(x^{l_k}) \geq \pi_i(y^{l_k}), i = 1,...,n$. The LHS converges to $\pi_i(x^\infty)$; the RHS to $\pi_i(y^\infty)$. The claim follows. Since $y^\infty \in V$, by \eqref{Vtildedef} $x^\infty \in \widetilde{V}$; 
$\widetilde{V}$ is closed. 

We will now prove $\partial_{SW} V \subset \partial_{SW} \widetilde{V}$. Let $x \in \partial_{SW} V \subset V \subset \widetilde{V}$. If by contradiction $x \notin \partial_{SW} \widetilde{V}$ then $\exists x' \in \widetilde{V}$ with $x' <' x$. Then $\exists x'' \in V$ with $x'' \leq' x'$ by 
\eqref{Vtildedef}. Then $x'' <' x$, which contradicts the minimality of $x$. So $x \in \partial_{SW} \widetilde{V}$. 

To prove the other direction, let $x \in \partial_{SW} \widetilde{V}$. If $x' \in V, \ x' \leq' x$, because $\nexists y \in \widetilde{V}$ with $y <' x$, and $ x' \in V \subset \widetilde{V}$, it follows $x = x'$. Specifically, $\nexists y \in \widetilde{V} \supset V$ with $y<'x$. Furthermore at least one such $x'$ exists by \eqref{Vtildedef}, so $x = x' \in V$. By definition $x \in \partial_{SW} V$. So $\partial_{SW} \widetilde{V} \subset \partial_{SW} V$. 
\end{proof} 

\begin{lemma}\label{massretractlemmafk} 
Let $V \subset [0,\infty)^n$ be closed. Denote
\begin{equation}\label{massretracthdef} 
h(x) = \inf \big\{y \in V \ | \ (x,y) \in V \big\}, \ \forall x \in \mathbb{R}^{n-1}
\end{equation} 
which is $+\infty$ if the set is empty. Define $\forall x \in \mathbb{R}^{n-1}, \forall y \in \mathbb{R}$, 
\begin{equation}\label{deffmassretract} 
f(x,y) = 
\begin{cases} 
(x,h(x)) & \text{if } y \geq h(x) \\ 
(x,y) & \text{otherwise} 
\end{cases}
\end{equation}
Then $f$ is Borel measurable. 
\end{lemma} 
\begin{proof} 
Let 
\begin{equation}\label{defVlm}
V_{l,m} = \pi\Big(V \cap \mathbb{R}^{n-1} \times \Big[\frac{l}{2^m}, \frac{l+1}{2^m}\Big]\Big), \ \forall l,m \geq 0.
\end{equation} 
where $\pi:(x_1,...,x_n) \mapsto (x_1,...,x_{n-1}), \ \forall (x_1,...,x_n) \in \mathbb{R}^n$. It is well known the projection of a finite dimensional real closed set is Borel measurable (due to being a countable union of compact sets, compactness
being preserved under projection). So $V_{l,m}$ are measurable. Let 
\begin{equation}\label{defWlm}
W_{l,m} = V_{l,m} \backslash \bigcup_{0 \leq l' < l} V_{l',m}, \ \forall l,m \geq 0.
\end{equation}
Given $l,m$, we claim 
\begin{equation}\label{Wlmclaim} 
\forall x \in W_{l,m}, \ h(x) \in \Big[\frac{l}{2^m}, \frac{l+1}{2^m}\Big].
\end{equation}
Indeed, $x \in W_{l,m} \implies x \in V_{l,m} \implies \exists y \in \big[l/2^m, (l+1)/2^m \big]$ with $(x,y) \in V$. 
By \eqref{massretracthdef}, $h(x) \leq y \leq (l+1)/2^m$. Furthermore, if $h(x) < l/2^m$, by \eqref{massretracthdef} $\exists (x,y) \in V$ with $y$ in some 
$\big[l'/2^m, (l'+1)/2^m\big], \ l'<l$. But then $x \in V_{l',m}$. A contradiction to \eqref{defWlm}. Therefore $l/2^m \leq h(x)$, which proves the claim.

$\forall m \in \mathbb{N}$, define $\forall x \in \mathbb{R}^{n-1}, \forall y \in \mathbb{R}$,
\begin{equation}\label{massretractfmdef}  
f_m(x,y) = 
\begin{cases} 
(x,l/2^m) & \text{if } \exists l \geq 0, \ (x,y) \in W_{l,m} \times \big[l/2^m,\infty\big) \\ 
(x,y) & \text{otherwise}
\end{cases} 
\end{equation}
Note, by \eqref{defWlm}, if $l \neq l'$ then $W_{l,m} \cap W_{l',m} = \emptyset$. So each $f_m$ is well-defined. 

We claim each $f_m$ is Borel measurable. Indeed, for $U \subset \mathbb{R}^n$ measurable, notice
\begin{align*} 
f_m^{-1}(U) &=  \Big(U \backslash \bigcup_{l \geq 0} W_{l,m} \times \big[l/2^m, \infty\big) \Big)\\ 
               &\cup \bigcup_{l \geq 0} \pi\big(U \cap W_{l,m} \times \{l/2^m\} \big) \times \big[l/2^m, \infty\big).
\end{align*} 
If each $\pi\big(U \cap W_{l,m} \times \{l/2^m\} \big)$ is measurable, then the RHS is measurable. It is routine, in projective cases such as this, to show $\pi$ restricted to 
$\mathbb{R}^{n-1} \times \{l/2^m\}$ is a bijection between $\sigma(\mathbb{R}^{n-1})$ and the $\sigma$-algebra of $\mathbb{R}^{n-1} \times \{l/2^m\}$ given by intersecting the Borel measurable sets 
of $\mathbb{R}^n$ with $\mathbb{R}^{n-1} \times \{l/2^m\}$. It follows each $\pi\big(U \cap W_{l,m} \times \{l/2^m\} \big)$ is measurable. So $f_m$ is indeed measurable. 

Recall Borel measurability of real functions is preserved under pointwise convergence. Thus, to show $f$ is measurable, we may prove $f_m \rightarrow f$, which we do now. 

Fix $x \in \mathbb{R}^{n-1}, y \in \mathbb{R}$. 

Case $y \geq h(x)$: Since $h(x)$ \eqref{massretracthdef} is finite, it follows for each $m \geq 0, \ \exists l_m \geq 0$ such that $x \in W_{l_m,m}$ (otherwise, by \eqref{defWlm}, 
$x \notin \bigcup_{l \geq 0} V_{l,m} = \pi(V)$, which contradicts $h(x)$ being finite). By \eqref{Wlmclaim} $\forall m \in \mathbb{N}, \ h(x) \in [l_m/2^m, (l_m+1)/2^m]$. It follows
$$
f_m(x,y) \stackrel{\eqref{massretractfmdef}}{=} (x,l_m/2^m) \xrightarrow[m \rightarrow \infty]{} (x,h(x)) \stackrel{\eqref{deffmassretract}}{=} f(x,y). 
$$

Case $y < h(x) < + \infty$: Similarly to before, $\forall m \in \mathbb{N}, \ \exists l_m$ such that $X\in W_{l_m,m}$. Furthermore, $\forall m \in \mathbb{N}, \ h(x) \in [l_m/2^m, (l_m+1)/2^m]$. 
Thus for $m>>0$, $h(x) - 1/2^m > y$. That is, $l_m/2^m > y$. So $(x,y) \notin W_{l_m,m} \times \big[l_m/2^m, \infty \big)$. 

Because $x \in W_{l_m,m}$ and $W_{l,m} \cap W_{l',m} = \emptyset$ for $l \neq l'$, it follows by \eqref{massretractfmdef} for $m>>0$ that $f_m(x,y) = (x,y)$. But $(x,y) = f(x,y)$ by \eqref{deffmassretract}, so we are done.  

Case $h(x) = + \infty$: By definition of $h$, 
$$
x \notin \pi(V) \stackrel{\eqref{defVlm}}{=} \bigcup_{l \geq 0} V_{l,m} \stackrel{\eqref{defWlm}}{=} \bigcup_{l \geq 0} W_{l,m}, \ m = 1, 2,...
$$
Then $\forall m \in \mathbb{N}, \ f_m(x,y) \stackrel{\eqref{massretractfmdef}}{=} (x,y) \stackrel{\eqref{deffmassretract}}{=} f(x,y)$, so we are done.

This shows $f_m \rightarrow f$ pointwise, hence $f$ is measurable. 
\end{proof} 
\begin{definition}[directional mass retracts]\label{directionalmassretractsdef}  
Let $W \subset [0,\infty)^n$ be closed. For $k=1,...,n$, $\forall x = (x_1,...,x_n) \in \mathbb{R}^n$, denote
\begin{equation}\label{massretracthkdef}
h_k(x) = \inf \Big\{y \in \mathbb{R} \ \big| \ (x_1,...,x_{k-1},y,x_{k+1},...,x_n) \in W  \Big\}
\end{equation}
from which we define
\begin{equation}\label{defGammak} 
\Gamma_k(x) := 
\begin{cases} 
\big(x_1,..,x_{k-1},h_k(x),x_{k+1},...,x_n\big) & \text{if } x_k \geq h_k(x) \\ 
x & \text{otherwise} 
\end{cases} 
\end{equation} 
Lastly, define $\Gamma := \Gamma_n \circ \cdots \circ \Gamma_1$.
\end{definition}  

By lemma~\ref{massretractlemmafk} (because the $n$th component was arbitrary) each $\Gamma_k$ is measurable. Notice $\Gamma_k$ satisfies definition~\ref{massretractdef} and 
$\forall x \in V, \ \Gamma_k(x) \in W$. $\Gamma$ is a mass retract by lemma~\ref{massretractproperties} (iii). Clearly $\forall x \in W, \ \Gamma(x) \in W$.

\begin{lemma}\label{Gammalemma}
If $W = \widetilde{W}$ \eqref{Vtildedef} in definition~\ref{directionalmassretractsdef}, then $\Gamma \circ \Gamma = \Gamma$.  
\end{lemma} 
\begin{proof} 
If $x \notin \widetilde{W}$ notice the set at \eqref{massretracthkdef} is empty, hence $h_k = +\infty$, in particular $\Gamma_k(x) = x$ by \eqref{defGammak}. Because this holds for all $k$, $\Gamma(x) = x$ which implies $\Gamma^2(x) = \Gamma(x)$. This proves the case $x \notin W$. 

Now, suppose by contradiction $\exists x \in W$ with $\Gamma^2(x) \neq \Gamma(x)$. Necessarily $\Gamma_k(\Gamma(x)) \neq \Gamma(x)$ for some $k$, otherwise $\Gamma^2(x) = \Gamma(x)$. Denote
$$
j = \min \{ k \ | \ \Gamma_k(\Gamma(x)) \neq \Gamma(x) \}.
$$
Applying \eqref{massretracthkdef} and \eqref{defGammak}, closedness of $W$, and $\Gamma(W) \subset W$,
$$
\Gamma_j(\Gamma(x)) = \big(\pi_1(\Gamma(x)), ..., \pi_{j-1}(\Gamma(x)), h_j(\Gamma(x)), \pi_{j+1}(\Gamma(x)),..., \pi_n(\Gamma(x)) \big) \in W.
$$
Because $\Gamma_j(\Gamma(x)) \neq \Gamma(x)$, 
\begin{equation}\label{Gammalemmaeq1}
   \pi_j(\Gamma_j(\Gamma(x))) = h_j(\Gamma(x)) < \pi_j(\Gamma(x)).
\end{equation}   
Because $\Gamma_j(\Gamma(x)) \in W, \ \exists y^1 \in \partial_{SW} W, \ y^1 \leq' \Gamma_j(\Gamma(x))$. By \eqref{Gammalemmaeq1}
\begin{equation}\label{Gammalemmaeq2} 
\pi_j(y^1) < \pi_j(\Gamma(x)).
\end{equation} 
Write $y^2 = (\Gamma_{j-1} \circ \cdots \circ \Gamma_1)(x) \geq' (\Gamma_n \circ \cdots \circ \Gamma_1)(x) = \Gamma(x)$. Combining inequalities, 
\begin{equation}\label{Gammalemmaeq3}
y^2 \geq' \Gamma(x) >' \Gamma_j(\Gamma(x)) \geq' y^1.
\end{equation} 
Let $y^3$ be $y^2$ with its $j$th component replaced by the $j$th component of $y^1$. That is, 
\begin{equation}\label{y3eqdef}
y^3 = (y^2_1,...,y^2_{j-1},y^1_j,y^2_{j+1},...,y^2_n). 
\end{equation} 

By \eqref{Gammalemmaeq3}, $y^3 \geq' y^1$. Because $y^1 \in \partial_{SW} W \subset W$, by definition of $\widetilde{W}$, $y^3 \in \widetilde{W} = W$. In particular, 
by \eqref{massretracthkdef} and \eqref{y3eqdef}, $y^3 \in W$ implies $h_j(y^2) \leq \pi_j(y^3) = y^1_j$. i.e.\
\begin{equation}\label{Gammalemmaeq4}
h_j(y^2) \stackrel{\eqref{defGammak}}{=} \pi_j(\Gamma_j(y^2)) = \pi_j(\Gamma_j \circ \cdots \circ \Gamma_1 (x)) \leq y^1_j. 
\end{equation}
Notice $\Gamma(x) \leq' \Gamma_j \circ \cdots \circ \Gamma_1(x)$. Applying $\pi_j$ and by \eqref{Gammalemmaeq4},
$$
\pi_j(\Gamma(x)) \leq y^1_j \stackrel{\eqref{Gammalemmaeq2}}{<} \pi_j(\Gamma(x))
$$
which is a contradiction. So $x \in W \implies \Gamma^2(x) = \Gamma(x)$. 
\end{proof} 
\begin{theorem} 
Let $V \subset [0,\infty)^n$ be closed. Define $\Gamma$ by definition~\ref{directionalmassretractsdef} with $W = \widetilde{V}$ where $\widetilde{V}$ is given by 
\eqref{Vtildedef}. Define
\begin{equation}\label{SWmassretractdef}
\forall x \in \mathbb{R}^n, \ \gamma_{SW}(x) := 
\begin{cases} 
\Gamma(x) & \text{if } x \in V \\
x & \text{ otherwise} 
\end{cases} 
\end{equation}
Then $\gamma_{SW}$ is a mass retract of $V$ with $\gamma_{SW}(V) \subset \partial_{SW} V$.
\end{theorem} 
\begin{proof} 
Recall $\Gamma$ is a mass retract. Notice $\gamma_{SW}$ is measurable and a mass retract. Below we prove $\gamma_{SW}(V) \subset \partial_{SW} V$. 

Firstly, if $x \in \widetilde{W}$ then $\exists x^1 \in W$ with $x^1 \leq' x$ by \eqref{Vtildedef}. Because $W = \widetilde{V}, \exists x^2 \in V$ with $x^2 \leq' x^1$. So $x \in \widetilde{V} = W$. Therefore  $\widetilde{W} \subset W$. 
Because $W \subset \widetilde{W}$ trivially, equality holds. By lemma~\ref{Gammalemma}, $\Gamma^2 = \Gamma$. It immediately follows $\gamma_{SW} \circ \gamma_{SW} = \gamma_{SW}$. 

Fix $x \in V$. Suppose by contradiction $\gamma_{SW}(x) \notin \partial_{SW} V$. Write $y = \gamma_{SW}(x) \stackrel{\eqref{SWmassretractdef}}{=} \Gamma(x)$. Then $\exists y^1 \in \partial_{SW} V$ with $y^1 <' y$. So $y^1_k \leq y_k$ for all $k$ with inequality for some 
$k$, say $y^1_j < y_j$. Now 
$$
y^1 \leq' (y_1,...,y_{j-1},y_j^1,y_{j+1},...,y_n).
$$
Because $y^1 \in \partial_{SW}V \subset V$, the RHS is an element of $\widetilde{V} = W$. By \eqref{massretracthkdef}, $h_j(y) \leq y_j^1$. Therefore $\pi_j(\Gamma_j(y)) = h_j(y) \leq y_j^1 < y_j$. So $\Gamma_j(y) \neq y$. Let
$$
j_0 = \min \big\{k \ | \ \Gamma_k(y) \neq y \big\}. 
$$
Recall $\Gamma_k(y)$ leaves all components of $y$ fixed except possibly the $k$th \eqref{defGammak}, which is sometimes made less. One has  
$$
\Gamma(y) \leq' \Gamma_{j_0} \circ \cdots \circ \Gamma_1(y) = \Gamma_{j_0}(y) <' y = \Gamma(x).
$$
That is, $\Gamma(\Gamma(x)) <' \Gamma(x)$. This contradicts $\Gamma^2 = \Gamma$. So $\gamma_{SW}(x) \in \partial_{SW}V$. 
\end{proof} 

\begin{theorem} 
Let $f \in \mathcal{B}^n$. Let $Q:(x_1,...,x_n) \mapsto (|x_1|,...,|x_n|), \ \forall x \in \mathbb{R}^n$. Let $V \subset \mathbb{R}^n$ be closed. Then 
\begin{equation}\label{maineqdeptails} 
\sup P(X \in V ) = \sup P\big(X \in \partial_{SW} Q(V)\big)
\end{equation} 
where the supremums are taken over\footnote{$X_k$ are not required to be independent.} $X = (X_1,...,X_n) : \Omega \rightarrow \mathbb{R}^n$ such that $P(|X_k| \geq t) \leq f_k(t), \ \forall t \geq 0, \ k=1,...,n$.
\end{theorem} 
\begin{proof} 
As described below \eqref{depsupcircus}, the existence of $\gamma_{SW}$ implies \eqref{depsupcircus}. Now, suppose $V \subset \mathbb{R}^n$ is closed. $Q(V) \subset [0,\infty)^n$ is closed (lemma~\ref{Qlemma}), 
therefore one can apply \eqref{depsupcircus} to \eqref{Q(U)WLOG}, resulting in \eqref{SWQ(V)WLOG}. \eqref{SWQ(V)WLOG} can be phrased in terms of r.v.'s with 
\eqref{rephrasementwithmeasureseq}. The result of these combinations is \eqref{maineqdeptails}.
\end{proof} 

I solved \eqref{maineqdeptails} exactly when $n=2$ and $f_1$ is continuous, but the proof is not included because it is long. Specifically, for $\emptyset \neq V \subset \mathbb{R}^2 \backslash \{0\}$ closed, write $W = \pi_1(\partial_{SW} Q(V))$ and let
$$
\kappa = \inf_{t \in W \backslash \sup W} f_1(j(t)) + f_2(h_2(t)) 
$$
where $j(t) = \inf (t,\infty) \cap W$ and $h_2(t) = \inf \{y \geq 0 \ | \ (t,y) \in V\}$.  I computed \eqref{maineqdeptails} equals
\begin{equation}\label{SWexact}
\sup P((X_1,X_2) \in V) = 
\begin{cases} 
\min \big\{ f_1(\inf W), \kappa \big\} & \text{if } \sup W \notin W \\ 
\min \big\{f_1(\inf W), \kappa, f_2(h_2(\sup W)) \big\} & \text{if } \sup W \in W
\end{cases} 
\end{equation}
Naturally this results in an explicit formula for the sharpest tail bound for the case $g: \mathbb{R}^2 \rightarrow \mathbb{R}$ is continuous and 
$X_1,X_2$, are absolute tail bounded by $f_1,f_2$ respectively ($f_1$ is continuous), but are not necessarily independent. In particular, 
$\forall t \in \mathbb{R}$ the lowest upper bound on $P(g(X_1,X_2) \geq t)$ is given by setting 
$W = \pi_1\big(\partial_{SW} Q(g^{-1}([t,\infty))\big)$ in the equation above.

\chapter{Shift Operators for mutually ind.\ real r.v.'s}
The goal of this chapter is to develop techniques to find the lowest possible upper bound on $P((X_1,...,X_n) \in V)$ when $V \subset \mathbb{R}^n$ is closed, $X_1,...,X_n$ are real r.v.'s and tail bounded, and are known to be mutually independent.
\section{Subscript notation}
This chapter is the continuation of chapter $2$. The results are, roughly, ``common sense," 
once shift operators are grasped, but the details are much more work. To start, many operations on subsets of $\mathbb{R}$
are applied, often in succession. This results in many bulky equations. Therefore, the following abbreviations are used: $\forall V \subset \mathbb{R}$,    
\begin{enumerate}\label{abbreviationschapter4}\label{Abr.S.O.'s}
   \item [1.] $V_{>x} = V \cap (x, \infty), \ V_{<x} = V \cap (-\infty, x), \ V_{\geq x} = V \cap [x,\infty), \ V_{\leq x} = V \cap (-\infty, x]$
   \item [2.] $V_- = \{ -x \ \big| \ x \in V\}$
   \item [3.] $V_{|\cdot |} = \{ |x| \ \big| \ x \in V\}$
   \item [4.] $X = (X_1,...,X_n)$ and  (recall lemma~\ref{Xtilde}) $$\widetilde{X}_k = \sup f_k^{-1}[1-s,1], \ \forall s \in (0,1)$$
   \item [5.] $G$ stands for a set of the form $G^1 \times \cdots \times G^n \subset \mathbb{R}^n$
\end{enumerate}
It is elementary to show $\overline{(V_{|\cdot|})} = \left(\overline{V} \right)_{|\cdot|}$, so $\overline{V}_{|\cdot|}$ is unambiguous.

All r.v.'s in this chapter have mutually ind.\ components---this is the situation shift operators were designed for. Recall $\mathcal{N}^n = \{(X_1,...,X_n) \ |  \ X_1,...,X_n \in \mathcal{N}\}$, where each $(X_1,...,X_n): (0,1)^n \rightarrow \mathbb{R}^n$ is the product r.v. of mutually independent $X_1,...,X_n$. Recall the following are subsets of $\mathcal{N}^n$: $\mathcal{T}(f)$ \eqref{T(f)} for $f \in \mathcal{B}^n$ \eqref{absolutetails}, $\mathcal{T}_2(f)$ for $f \in \mathcal{B}_2^n$ \eqref{B2T2}, $\mathcal{T}_R(f)$ for $f \in \mathcal{B}_R^n$, and $\mathcal{T}_L(f)$ for $f \in \mathcal{B}_L^n$. Therefore these sets also consists of r.v.'s whose elements have mutually ind.\ components. 

\section{Overview}
Recall the situations (dilemas) described with example $1$ and example $2$ (figure~\ref{eg.1fig}, \eqref{eg2eq}). Example $2$ was resolved by simplifying the domains of the r.v.'s with 
neat r.v.'s, specifically at \eqref{1.3again} and \eqref{whoknows}. Example $1$ will be resolved by introducing shift operators.

$\mathscr{S}_R(f,G)$ is the right shift operator, which will be \textit{constructed} to return, in the next  few pages, for $\emptyset \neq G = G^1 \times \cdots \times G^n \subset \mathbb{R}^n$ closed, a r.v.\ 
$(X_1,...,X_n) \in \mathcal{N}^n$ such that 
\begin{enumerate}[label=(\roman*)]
\item $\forall s_k \in (0,1), \ X_k(s_k) \in G^k\cup(-\infty, \inf G^k]$
\item $P(X_k \geq t) \leq f_k(t), \ \forall t \in \mathbb{R}$, with equality $\forall t \in G^k \cup (-\infty, \inf G^k]$
\end{enumerate}
$\mathscr{S}_R(f,G)$ is interpretable as the $n$-tuple of neat r.v.'s $(X_1,...,X_n)$ with each $X_k$ ``shifted rightwards," without leaving $G^k \cup (-\infty, \inf G^k]$, as far as $P(X_k \geq t) \leq f_k(t)$ 
allows. 

Our main goal is to define $\mathscr{S}_R, \mathscr{S}_L, \mathscr{S}_2, \mathscr{S}$, and then prove the following: let $V \subset \mathbb{R}^n$ be closed. Then,
\begin{align}
\text{For } f \in \mathcal{B}^n, \sup_{X \in \mathcal{T}(f)} P(X \in V) &= P(\mathscr{S}(f,G) \in V), \text{ some } G\ni 0. \label{Ssup2} \\
\text{For } f \in \mathcal{B}_2^n, \sup_{X \in \mathcal{T}_2(f)} P(X \in V) &= P(\mathscr{S}_2(f,G,c) \in V), \text{ some } G \ni 0, \label{Ssup4}
\end{align} 
some $c \in [0,1]^n$ (recall $G = G^1 \times \cdots \times G^n \subset \mathbb{R}^n$). These equations are rooted in the following informal argument. Here is the figure.

\begin{figure}[h!]
   \includegraphics[scale = .22, center]{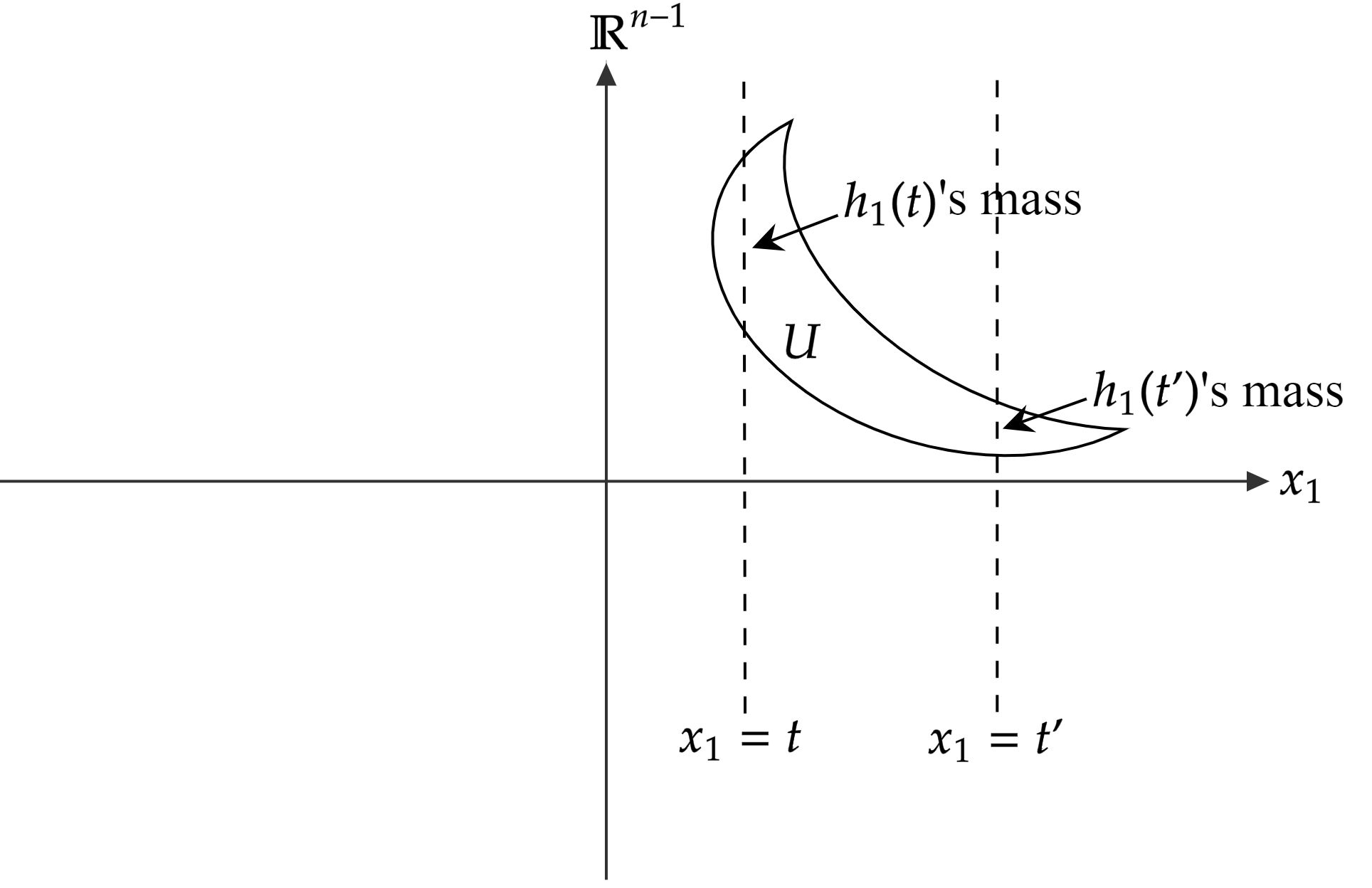}\label{testfig2}
   \caption{Informal argument for shift operators}\label{fig.1}
\end{figure}
Suppose $(X_1,...,X_n)$ maximizes the mass/measure of $U \subset \mathbb{R}^n$ over $X_k$ mutually ind.\ and tail bounded.
Denote 
\begin{align*}\label{h1informal}
h_1(t) &= P\big((t, X_2...,X_n) \in U\big) \\ 
       &= P\big((t, X_2...,X_n) \in U \cap \{x \in \mathbb{R}^n \ \big| \ \pi_1(x) = t\}\big)
\end{align*}
Each slice $U \cap \{x \in \mathbb{R}^n \ \big| \ \pi_1(x) = t\}$ of $U$, given by fixing the $1$st component at $t$, has (infinitesimal) mass ``$P(X_1 = t) h_1(t).$" Comparing two such slices at $t$ and $t'$, $0 \leq t < t'$, if $h_1(t) \geq h_1(t')$, then the overall measure of $U$ can be increased by moving $X_1$'s mass at $t'$ to $t$. That is, 
$$
P(X_1 = t) h_1(t) + P(X_1 = t') h_1(t') \leq (P(X_1 = t) + P(X_1 = t'))h_1(t).
$$
Vice versa, if $h_1(t') > h_1(t)$, then one can increase the overall measure of $U$ by moving $X_1's$ mass at $t$ to $t'$. There is one caveat: one can always move mass inwards from $t'$ to $t$, 
but the tail bound on $X_1$ may prevent moving mass outwards from $t$ to $t'$.

Because $(X_1,...,X_n)$ maximizes the measure of $U$ by hypothesis, by the previous informal argument we conclude if $X_1$ has mass at some $t'$, then $\forall t \in [0,t'), h_1(t) < h_1(t')$ (otherwise the measure of $U$ could be increased by moving $X_1$'s mass at $t'$ backwards to $t$). Continuing this line of reasoning, because we have concluded $\forall t \in [0,t'), \ h_1(t) < h_1(t')$, the measure of $U$ can be increased by moving $X_1$'s mass from all $t \in [0,t')$ to $t'$ until the tail bound on $X_1$ (which may be a right, absolute, or $2$ tail bound) prevents adding more mass to $t'$.

Because $t'$ was arbitrary, these observations describe the form of $(X_1,...,X_n)$. In conclusion, if $(X_1,...,X_n)$ maximizes the mass of $U$, without loss of generality $X_1$ is ``pushed outwards against its tail bound on some $G^1 \subset \mathbb{R}."$ Furthermore, 
$\forall t, t' \in G^1, \ 0 \leq t<t'$, one has $h_1(t) < h_1(t')$. Lastly, $X_1$ was arbitrary, so each $X_k$ has the form of being ``pushed outwards on some $G^k$ against its tail" as well.

Notice this informal argument appears to apply no matter what type of tail bound is placed on each $X_k$ (section~\ref{tailtypes}). Unfortunately I am still looking for a general framework which treats all tail types collectively, rather 
than on a case by case basis. 
\index{Right shift operator $\mathscr{S}_R$} 
\begin{definition}[Right shift operator]\label{SRdef}
The right shift operator $\mathscr{S}_R$ is defined as
$$
\mathscr{S}_R: (f, G) \mapsto (X_1,...,X_n)
$$
where $f=(f_1,...,f_n) \in \mathcal{B}_R^n, \ G$ is non-empty and closed, and (recall section~\ref{abbreviationschapter4} notations) 
\begin{equation}\label{SRXkeq}
   X_k(s) = \sup \big(G^k \cup (-\infty, \inf G^k]\big)_{\leq \widetilde{X}_k(s)}, \ \forall s \in (0,1)
\end{equation}
where
\begin{equation}\label{Xtilde(2)}
\widetilde{X}_k(s) := \sup f_k^{-1}[1-s,1], \ \forall s \in (0,1)
\end{equation}
is from lemma~\ref{Xtilde}. 
\end{definition}
The full expression is, $\forall s = (s_1,...,s_n) \in (0,1)^n$,
\begin{align*}
\mathscr{S}_R((f_1,...,f_n), G)(s) =& \bigg( \sup \Big(\big(G^1\cup(-\infty,\inf G^1]\big) \cap \big(-\infty, \sup f_1^{-1}[1-s_1,1] \big]\Big), \\
&..., \sup \Big(\big(G^n \cup (-\infty, \inf G^n]\big) \cap (-\infty, \sup f_n^{-1}[1-s_n,1] \big]\Big)\bigg).
\end{align*}
Notice $\mathscr{S}_R(f_k,G^k)$, the right shift operator for $n=1$, equals $X_k$ in \eqref{SRXkeq}. So, to condense this formula, we may write\footnote{Similar condensations exist and will be used for the other shift operators.}
\begin{align}\label{condenseSR}
\mathscr{S}_R(f,G) &= (\mathscr{S}_R(f_1,G^1),...,\mathscr{S}_R(f_n,G^n)), \text{ or } \\
&= \Big(\sup \big(G^1 \cup (-\infty, \inf G^1]\big)_{\leq \widetilde{X}_1},..., \sup \big(G^n \cup (-\infty, \inf G^n]\big)_{\leq \widetilde{X}_n}\Big).
\end{align}
Explanation of \eqref{SRXkeq}: lemma~\ref{Xtilde} said $\widetilde{X}_k$ is the ``largest," or ``rightmost,'' neat r.v.\ which satisfied 
$$
P(\widetilde{X}_k \geq t) \leq f_k(t), \ \forall t \in \mathbb{R}.
$$ 
More specifically, (ii) said
$$
P(\widetilde{X}_k \geq t) = f_k(t), \ \forall t \in \mathbb{R},
$$
and (iii) said
\begin{equation}\label{Xtilde(3)}
f_k \in \mathcal{B}_R, \ X_k \in \mathcal{T}_R(f_k) \implies \forall s \in (0,1), \ X_k(s) \leq \widetilde{X}_k(s).
\end{equation}
In other words, there does not exist a neat r.v.\ which sends $s \in (0,1)$ to a value greater than $\widetilde{X}_k(s)$ while maintaining the right tail bound $f_k$.

$\mathscr{S}_R(f,G) = (X_1,...,X_n)$ was supposed to formalize ``shift each $X_k$ outwards on $G^k$ as far as possible against $f_k$." By lemma~\ref{existenceneat}, WLOG $X_k$ are neat, 
which is advantageous. Given these desires for $X_k$, why \eqref{SRXkeq}?

A good try is $X_k$ mapping each $s \in (0,1)$ to the largest value of $G^k$ it can while maintaining $X_k \nearrow$ and $P(X_k \geq t) \leq f_k(t), \ \forall t \in \mathbb{R}$. In terms
of formulas, these restrictions read $X_k(s) \in G^k$ and $X_k(s) \leq \widetilde{X}_k(s)$ (by \eqref{Xtilde(3)}). The largest value of $\mathbb{R}$ allowed by the restrictions is 
$\sup \big(G^k \cap (-\infty, \widetilde{X}_k(s)]\big) = \sup G^k_{\leq \widetilde{X}_k(s)}$. 

There is one problem: this supremum is $-\infty$ when $\widetilde{X}_k(s) < \inf G^k$, because the least upper bound of the empty set is $- \infty$. Adjoining $(-\infty, \inf G^k]$ to $G^k$ solves this problem, 
which is \eqref{SRXkeq}.

\begin{theorem}\label{SRTR(f)}
For all $f \in \mathcal{B}_R^n$ and $G$ non-empty and closed,\footnote{Recall $G$ is of the form $G=G^1 \times \cdots \times G^n \subset \mathbb{R}^n$ (section~\ref{abbreviationschapter4}).}
$$
\mathscr{S}_R(f,G) \in \mathcal{T}_R(f).
$$
\end{theorem}
\begin{proof}
To show $(\mathscr{S}_R(f_1,G^1),...,\mathscr{S}_R(f_n,G^n))$ satisfies \eqref{TR}, the definition of $\mathcal{T}_R(f)$, we must use \eqref{SRXkeq} to prove $\mathscr{S}_R(f_k,G^k)$ are neat and $P(\mathscr{S}_R(f_k,G^k) \geq t) \leq f_k(t), \forall t \in \mathbb{R}$. 

$\mathscr{S}_R(f_k,G^k) \in \mathcal{N}$: 

By lemma~\ref{Xtilde} (i), $\widetilde{X}_k$ is neat, i.e.\ monotonically increasing on $(0,1)$. The subset $\big(G^k \cap (-\infty, \inf G^k] \big)_{\leq \widetilde{X}_k(s)}$ in \eqref{SRXkeq} 
can only become larger if $\widetilde{X}_k(s)$ increases. It follows $\big(G^k \cap (-\infty, \inf G^k] \big)_{\leq \widetilde{X}_k(s)}$ is monotonically increasing on $(0,1)$ as desired.

$P(\mathscr{S}_R(f_k,G^k) \geq t)$ satisfies the right tail $f_k$: 

It is sufficient to show 
\begin{equation}\label{K1}
\big\{s \in (0,1) \ \big| \ \mathscr{S}_R(f_k,G^k)(s) \geq t \big\} \subset \big\{s \in (0,1) \ \big| \ \widetilde{X}_k(s) \geq t \big\},
\end{equation}
because by lemma~\ref{Xtilde} (ii), the measure of the RHS is $f_k(t)$. \eqref{K1} is true because by \eqref{SRXkeq}, $\mathscr{S}_R(f_k,G^k)(s) \leq \widetilde{X}_k(s), \forall s \in (0,1)$. Since $k$ was arbitrary we are done.
\end{proof}

Notice if $G = G^1 \times \cdots \times G^n = \mathbb{R}^n$, then 
$$
\mathscr{S}_R(f_k,\mathbb{R})(s) \stackrel{\eqref{SRXkeq}}{=} \sup \mathbb{R}_{\leq \widetilde{X}_k(s)} = \widetilde{X}_k(s).
$$
All together, $\mathscr{S}_R(f,\mathbb{R}^n) = (\widetilde{X}_1,...,\widetilde{X}_n)$. 
\begin{equation*}
\sup_{X \in \mathcal{T}_R(f)} P(X \in g^{-1}[t,\infty)) = P\big(\mathscr{S}_R(f,\mathbb{R}^n) \in g^{-1}[t,\infty)\big), \ \ \forall t \in \mathbb{R}.
\end{equation*}

\begin{theorem}
The reversed CDF (section~\ref{preliminaries}) of $\mathscr{S}_R(f_k,G^k)$ is 
\begin{equation}\label{SRRCDF}
P(\mathscr{S}_R(f_k,G^k) \geq t) = f_k\Big(\inf \big(G^k\cup (-\infty, \inf G^k]\big)_{\geq t}\Big), \ \forall t \in \mathbb{R}
\end{equation}
with the conventions $\inf \emptyset = + \infty$ and $f_k(+\infty) = 0$. 
\begin{proof}
For convenience, we may replace $G^k$ with $G^k \cup (-\infty, \inf G^k]$ WLOG. \eqref{SRXkeq} becomes
$$
\mathscr{S}_R(f_k,G^k)(s) = \sup G^k_{\leq \widetilde{X}_k(s)}, \ \forall s \in (0,1).
$$
One finds
$$
\mathscr{S}_R(f_k,G^k)(s) \geq t \iff \exists t' \in G^k \cap [t, \widetilde{X}_k(s)].
$$
The RHS holds iff $\widetilde{X}_k(s) \geq \inf G^k \cap [t,\infty) = \inf G^k_{\geq t}, \ G^k_{\geq t} \neq \emptyset$. Thus, 
\begin{equation}\label{gamma1}
\mathscr{S}_R(f_k,G^k)(s) \geq t \iff \widetilde{X}_k(s) \geq \inf G^k_{\geq t}, \ G^k_{\geq t} \neq \emptyset. 
\end{equation}
So $P(\mathscr{S}_R(f_k,G^k) \geq t) = 0$ if $G^k_{\geq t} = \emptyset$ (i.e.\ if $\inf G^k_{\geq t} = +\infty$, because the greatest lower bound of the empty set is $+\infty$). For $G^k_{\geq t} \neq \emptyset$, define
$$
I_t := \big\{ s \in (0,1) \ \big| \ \mathscr{S}_R(f_k,G^k)(s) \geq t \big\} \stackrel{\eqref{gamma1}}{=} \big\{ s \in (0,1) \ \big| \ \widetilde{X}_k(s) \geq \inf G^k_{\geq t} \big\}.
$$
This gives
\begin{align*}
P(\mathscr{S}_k(f_k,G^k) \geq t) = m(I_t) &= P(\widetilde{X}_k \geq \inf G^k_{\geq t}) \\ 
&= f_k(\inf G^k_{\geq t})
\end{align*}
where the last step is because $P(\widetilde{X}_k \geq t) = f_k(t), \ \forall t \in \mathbb{R}$ \ (lemma~\ref{Xtilde} (ii)).
\end{proof} 
\end{theorem}
Recall lemma~\ref{uniquenessneat} states there is a unique right/left continuous neat r.v.\ with a given CDF (equivalently, a given reversed CDF). In the next theorem we will prove 
$\sup \ (G^k \cup (-\infty, \inf G^k])_{\leq \widetilde{X}_k(s)}$ is right continuous. It follows $\mathscr{S}_R(f_k,G^k)$ is the unique neat right continuous r.v.\ which satisfies \eqref{SRRCDF}.

\begin{theorem}
Given $f_k \in \mathcal{B}_R$ and $G^k$,
\begin{equation}\label{gamma2}
   \mathscr{S}_R(f_k,G^k) (s) = \sup (G^k \cup (-\infty, G^k])_{\leq \widetilde{X}_k(s)}, \ \forall s \in (0,1)
\end{equation} 
is right continuous.
\end{theorem}
\begin{proof}
There are two cases to consider.

$
(1) \ \ \ \ \ \ \ \ \text{Case } \big(G^k \cup (-\infty, \inf G^k]) \cap \big[\widetilde{X}_k(s), \widetilde{X}_k(s) + \epsilon \big] = \emptyset, \ \text{some } \epsilon >0:
$

Since $\widetilde{X}_k$ is right continuous and $\widetilde{X}_k \nearrow$ (lemma~\ref{Xtilde}), necessarily $\exists \epsilon_s>0$ such that $\widetilde{X}_k\big([s,s+\epsilon_s)\big) \subset 
[ \widetilde{X}_k(s), \widetilde{X}_k(s)+\epsilon].$ By this and the case hypothesis, 
$$
\widetilde{X}_k \big([s,s+\epsilon_s)\big) \cap \big(G^k \cup (-\infty, \inf G^k ]\big) = \emptyset.
$$
Combining this with \eqref{SRXkeq},
$$
\mathscr{S}_R(f_k,G^k)(s) = \mathscr{S}_R(f_k,G^k)(s'), \ \forall s' \in [s,s+\epsilon_s).
$$
This implies right continuity as desired. 
$$
\text{Case } \big(G^k \cup (-\infty, \inf G^k] \big) \cap \big[ \widetilde{X}_k(s), \widetilde{X}_k(s) + \epsilon \big] \neq \emptyset, \ \forall \epsilon >0:
$$
One can construct a sequence in $G^k \cup (-\infty, \inf G^k]$ which converges to $\widetilde{X}_k(s)$. $G^k$ is (always) closed (section~\ref{Abr.S.O.'s}). So $\widetilde{X}_k(s) \in G^k \cup (-\infty, \inf G^k]$. It follows 
\begin{equation*} 
   \mathscr{S}_R(f_k,G^k)(s) \stackrel{\eqref{SRXkeq}} = \widetilde{X}_k(s).
\end{equation*}
On the other hand, $\mathscr{S}_R(f_k,G^k)(s+\epsilon) \geq \mathscr{S}_R(f_k,G^k)(s)$ ($\mathscr{S}_R$ is neat by theorem~\ref{SRTR(f)}). Furthermore, by \eqref{SRXkeq}, $\mathscr{S}_R(f_k,G^k)(s+\epsilon) \leq \widetilde{X}_k(s+\epsilon)$. Combining these three,
$$
\widetilde{X}_k(s) = \mathscr{S}_R(f_k,G^k)(s)\leq \mathscr{S}_R(f_k,G^k)(s+\epsilon) \leq \widetilde{X}_k(s+\epsilon)
$$
Because $\widetilde{X}_k$ is right continuous (lemma~\ref{Xtilde} (ii)), this implies $\mathscr{S}_R(f_k, G^k)(s+\epsilon) \rightarrow \mathscr{S}_R(f_k,G^k)(s)$ as $\epsilon \searrow 0$.
\end{proof}
It is convenient to collect the properties of $\mathscr{S}_R$ we have learned in one location, which are mostly the previous results restated. 
\begin{theorem}[$\mathscr{S}_R$ properties]\label{SRproperties}
Let $f \in \mathcal{B}_R^n$ \eqref{BRTR}. Let $\emptyset \neq G = G^1 \times \cdots \times G^n \subset \mathbb{R}^n$, with all $G^k$ closed. Then 
\begin{align}
   \mathscr{S}_R(f,G) &= (\mathscr{S}_R(f_1,G^1),...,\mathscr{S}_R(f_n,G^n)) \\
   &= \Big(\sup \big(G^1 \cup (-\infty, \inf G^1]\big)_{\leq \widetilde{X}_1},..., \sup \big(G^n \cup (-\infty, \inf G^n]\big)_{\leq \widetilde{X}_n}\Big) \label{SRXkeq2}
   \end{align}
has the following properties. 
\begin{enumerate}[label = (\roman*)]
\item $\mathscr{S}_R(f,G) \in \mathcal{T}_R(f)$
\item $\forall k$, $P(\mathscr{S}_R(f_k,G^k) \geq t) = f_k \big( \inf \big(G^k \cup (-\infty, \inf G^k]\big)_{\geq t} \big), \ \forall t \in \mathbb{R}$ (employing the convention\footnote{This is natural because $f_k \searrow 0$ by definition of $\mathcal{B}_R$.} $f_k(+\infty) = 0$)
\item $\mathscr{S}_R(f_k,G^k)$ are right continuous
\item $\mathscr{S}_R(f,G)$ is the unique element of $\mathcal{N}^n$ which satisfies (ii) and (iii)
\item $\forall s_k \in (0,1), \ \mathscr{S}_R(f_k,G^k)(s_k) \in G^k \cup (-\infty, \inf G^k]$
\item $\forall s_k \in(0,1), \ \mathscr{S}_R(f_k,G^k)(s_k) \leq \widetilde{X}_k(s_k), \ k=1,...,n$
\end{enumerate}
\begin{proof} 
(iv) is an application of lemma~\ref{uniquenessneat}. (v) and (vi) follow immediately by \eqref{SRXkeq} and since $G^k$ are closed.
\end{proof} 
\end{theorem}

\section{Applying the right shift operator to $g: \mathbb{R} \rightarrow \mathbb{R}$}
For $g: \mathbb{R} \rightarrow \mathbb{R}$ measurable, one way to visualize the sharpest right tail of $g(X)$, for $X$ right tail bounded by $f \in \mathcal{B}_R$, is to graph $g$. 

\begin{figure}[h!]
\includegraphics[scale = .35, center]{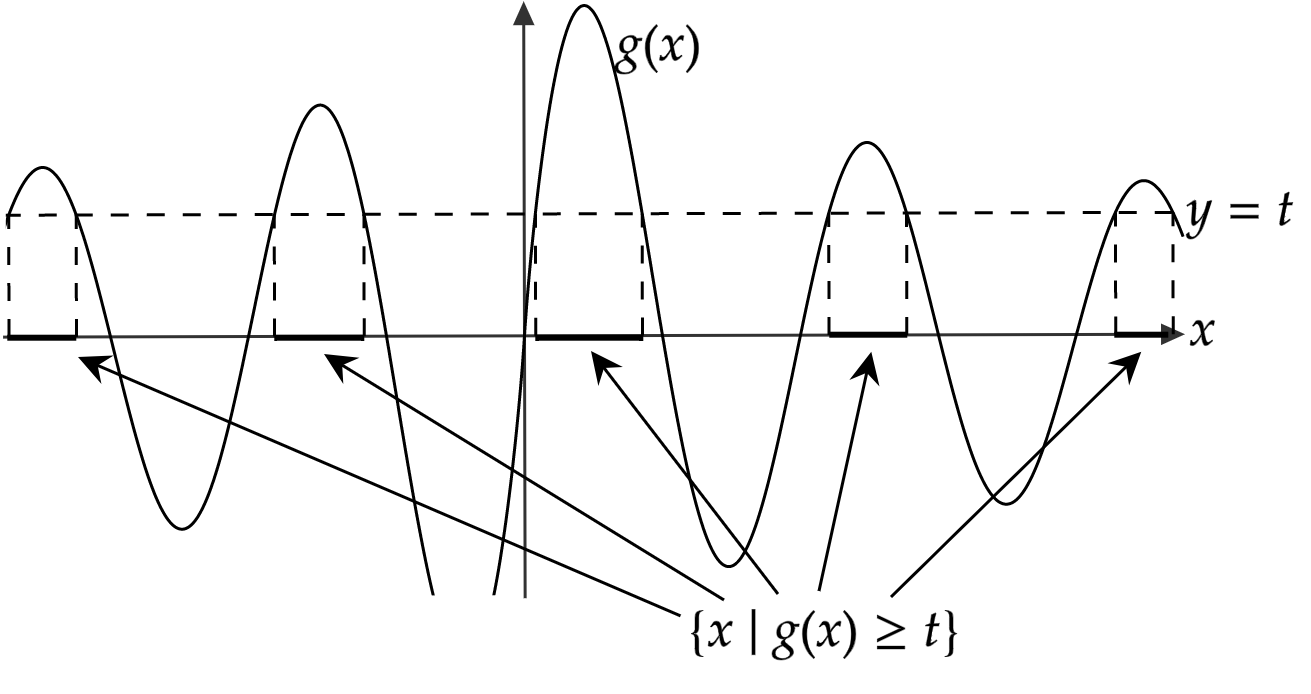}\label{testfig}
\caption{Visualizing the sharpest tail bound}\label{fig.2}
\end{figure}
Then, maximizing the measure of $g^{-1}[t,\infty)$ amounts to getting $X$'s mass as much in $g^{-1}[t,\infty)$ as the right tail allows. Potentially, $g^{-1}[t,\infty), \ g^{-1}[t',\infty), \ t \neq t'$ have different $X$ which maximize their mass/measure.  

It is better to have $X$'s mass at points where $g$ is highest, since that mass will be in $g^{-1}[t,\infty)$ for many values of $t$. This way, the same $X$ can be used to maximize the mass of $g^{-1}[t,\infty)$ for many $t$.

This discussion also applies to higher dimensions when $g: \mathbb{R}^n \rightarrow \mathbb{R}$, trying to maximize the mass of $g^{-1}[t,\infty)$ over $(X_1,...,X_n) \in \mathcal{T}_R(f_1,...,f_n)$; one can simply suppose the horizontal axis is $\mathbb{R}^{n-1}$. Theorem~\ref{thm.1}, theorem~\ref{thm.2}, and example $2$ were 
such cases where a single choice of $(X_1,...,X_n)$ maximized the measure of $g^{-1}[t,\infty)$ for all $t$. In particular, $G = \mathbb{R}^n$, which returns $\mathscr{S}_R(f,\mathbb{R}^n) = (\widetilde{X}_1,...,\widetilde{X}_n)$ \eqref{SRXkeq}.

\begin{theorem}
Let $g: \mathbb{R} \rightarrow \mathbb{R}$ be continuous and $f \in \mathcal{B}_R$. If 
\begin{equation}\label{gcontG^1def}
G:= \big\{ t \in \mathbb{R} \ \big| \ \forall t' < t, g(t') \leq g(t) \big\}
\end{equation}
is not bounded below,\footnote{The necessity for $G^1$ to not be bounded below can be understood by figure~\ref{fig.2}; if $G$ were bounded below, then no single $X \in \mathcal{T}_R(f_1)$ 
would maximize the measure of $g^{-1}[t,\infty)$ for all $t$. This theorem is difficult to generalize to higher dimensions, so this is not essential to understand.} then 
\begin{equation}\label{gconteq}
\sup_{X \in \mathcal{T}_R(f)}  P(g(X) \geq t) = P\big(g(\mathscr{S}_R(f,G) \geq t)\big), \ \forall t \in \mathbb{R}.
\end{equation}
\end{theorem}
\begin{proof}
To show $G$ is closed: Suppose $(t_k)_k \subset G$ such that $t_k \rightarrow t$. We will show $t$ satisfies the definition of $G$. Suppose $t' <t$. By \eqref{gcontG^1def}, 
$g(t') \leq g(t_k)$ for all $k>>0$. Taking $k \rightarrow \infty$ and using the continuity of $g$,
$$
g(t') \leq \lim_{k \rightarrow \infty} g(t_k) = g(t)
$$
as desired. So $G$ is closed. We will now show\footnote{This is hopefully visually intuitive (figure~\ref{fig.2}), but is too lengthy to explain to be worthwhile if it is not.}
\begin{equation}\label{gcontclaim}
g\big(\mathscr{S}_R(f,G)(s)\big) = \sup_{t \in (-\infty, \widetilde{X}(s)]} g(t), \ \forall s\in (0,1).
\end{equation}
LHS $\leq$ RHS: $\mathscr{S}_R(f,G)(s) \in (-\infty, \widetilde{X}(s)]$ (theorem~\ref{SRproperties} (vi)), so we are done.

LHS $\geq$ RHS: there are two cases for $t \leq \widetilde{X}(s)$.

Case $t \in (-\infty, \mathscr{S}_R(f,G)(s))$: 

Because $\mathscr{S}_R(f,G) \in G$ (theorem~\ref{SRproperties} (v)), by the definition of $G, \\ g(t) \leq g(\mathscr{S}_R(f,G))$. 

Case $t \in [\mathscr{S}_R(f,G)(s), \widetilde{X}(s)]$: 

Because $g$ is continuous, $\exists t_{max} \in [\mathscr{S}_R(f,G)(s), \widetilde{X}(s)]$ such that $g(t_{max})$ achieves the max of $g$ on 
$[\mathscr{S}_R(f,G), \widetilde{X}(s)]$. We will show 
$$
t_{max} = \mathscr{S}_R(f,G)(s).
$$
This will be done by showing $t_{max} \in G$, since then, because $t_{max} \leq \widetilde{X}(s)$,
$$
\mathscr{S}_R(f,G)(s) = \sup G_{\leq \widetilde{X}(s)} \geq t_{max} \geq \mathscr{S}_R(f,G)(s).
$$
$t_{max} \in G$: 

By \eqref{gcontG^1def}, definition of $G$, we must show $\forall t' \in (-\infty, t_{max}], \ g(t') \leq g(t_{max})$. This certainly holds for $t' \in [\mathscr{S}_R(f,G)(s), t_{max}]$, by 
definition of $t_{max}$. If instead $t' < \mathscr{S}_R(f,G)(s)$, then
\begin{align*}
\ \ \ \ \ \ \ \ \ \ \ \ \ \ \ \ \ \ \ \ \ \ \ \ \ \ \ \ \ g(t') &\leq g(\mathscr{S}_R(f,G)(s)) \ \ \ \ \eqref{gcontG^1def} \text{ and since } \mathscr{S}_R(f,G) \in G\\ 
\ \ \ \ \ \ \ \ \ \ \ \ \ \ \ \ \ \ \ \ \ \ \ \ \ \ \ \ \ &\leq g(t_{max}). \ \ \ \ \ \ \ \ \ \ \ \ \ \ t_{max} \text{ definition}
\end{align*}
 So $t_{max} \in G$ as desired. Thus LHS $\geq$ RHS. Thus \eqref{gcontclaim} holds.

Now, suppose $X \in \mathcal{T}_R(f), \ t \in \mathbb{R}$. Firstly $X(s) \leq \widetilde{X}(s)$ (lemma~\ref{Xtilde}). So by \eqref{gcontclaim}, if $s \in (0,1)$ such that $g(X(s)) \geq t$, then
$g(\mathscr{S}_R(f,G)(s)) \geq t$ too. Therefore
$$
\{s \in (0,1) \ \big| g(X(s)) \geq t\} \subset \{ s \in (0,1) \ \big| \ g(\mathscr{S}_R(f,G)(s)) \geq t \}.
$$
Applying the measure $m$, and then the supremum over $X \in \mathcal{T}_R(f)$,
$$
\sup_{X \in \mathcal{T}_R(f)} P\big(g(X)\geq t\big) \leq P\big(g(\mathscr{S}_R(f,G)) \geq t\big).
$$
Equality holds since $\mathscr{S}_R(f,G) \in \mathcal{T}_R(f)$ (theorem~\ref{SRproperties} (i)).
\end{proof} 
\section{Left and $2$ shift operators}

We have only rigorously defined $\mathscr{S}_R$. We will now define $\mathscr{S}_L$ and $\mathscr{S}_2$, corresponding to left and $2$ tail bounds (\eqref{B2T2} and section~\ref{tailtypes}), 
respectively.

\index{Left shift operator $\mathscr{S}_L$} 
$\mathscr{S}_L$ is symmetric to $\mathscr{S}_R$ (it is a reflection) shifting the $X_k$ subject to $P(X_k \leq t) \leq f_k(t), \ \forall t \in \mathbb{R}$ leftwards on $G^k$ rather than rightwards. Specifically, 
\begin{definition}
The left shift operator $\mathscr{S}_L$ is defined as
\begin{equation*}
\mathscr{S}_L: (f,G) \mapsto (X_1,...,X_n)
\end{equation*}
where $f=(f_1,...,f_n) \in \mathcal{B}_L, \ \emptyset \neq G = G^1 \times \cdots \times G^n$ is closed, and \footnote{This is an analog of \eqref{SRXkeq}, obtained by reflecting about $0$. $f_k^{-1}[s,1], \forall s \in (0,1)$ 
is used instead of $f_k^{-1}[1-s,1]$ because we want $\mathscr{S}_L(f_k,G^k) \in \mathcal{N}$, namely to be monotonically increasing.}
\begin{equation*}
   X_k(s) := \inf \big(G^k\cup [\sup G^k, \infty) \big)_{\geq \inf f_k^{-1}[s,1]}, \ \forall s \in(0,1). 
\end{equation*}
\end{definition}
It would be very tedious and not very helpful to reprove reflected versions of lemma~\ref{Xtilde} and theorem~\eqref{SRproperties} for $\mathscr{S}_L$, when it is clear they should hold due to 
$\mathscr{S}_L$ being totally symmetric to $\mathscr{S}_R$, so we will skip these proofs. The reflected version of theorem~\ref{SRproperties} is 
\begin{theorem}[$\mathscr{S}_L$ properties]\label{SLproperties}
   Let $f \in \mathcal{B}_L^n$. Let $\emptyset \neq G = G^1 \times \cdots \times G^n \subset \mathbb{R}^n$, with all $G^k$ closed. Then 
   \begin{align*}
      \mathscr{S}_L(f,G) = &(\mathscr{S}_L(f_1,G^1),...,\mathscr{S}_L(f_n,G^n)) \\
      = &\Big(\inf \big(G^1 \cup [\sup G^1, \infty)\big)_{\geq \inf f_1^{-1}[s,1]},\\ &...,\inf \big(G^n \cup [\sup G^n, \infty)\big)_{\geq \inf f_n^{-1}[s,1]}\Big)
      \end{align*}
   has the following properties.
\begin{enumerate}[label = (\roman*)]
   \item $\mathscr{S}_L(f,G) \in \mathcal{T}_L(f)$
   \item $\forall k$, $P(\mathscr{S}_L(f_k,G^k) \leq t) = f_k \big( \sup \big(G^k \cup [\sup G^k, \infty)\big)_{\leq t} \big), \ \forall t \in \mathbb{R}$ (employing the convention $f_k(-\infty) = 0$)
   \item $\mathscr{S}_L(f_k,G^k)$ are left continuous
   \item $\mathscr{S}_L(f,G) :(0,1)^n \rightarrow \mathbb{R}^n$ is the unique function which satisfies (i)-(iii)
   \item $\forall s_k \in (0,1)^n, \ \mathscr{S}_L(f_k,G^k)(s_k) \in G^k \cup [\sup G^k, +\infty)$
   \item $\forall s_k\in(0,1), \ \mathscr{S}_L(f_k,G^k)(s_k) \geq \inf f_k^{-1}[s_k,1], \ k=1,...,n$
   \end{enumerate}
   \end{theorem}
The two shift operator and absolute shift operator have the same goal $\mathscr{S}_L$ and $\mathscr{S}_R$ do: shifting each $X_k$ outwards as much as possible against its tails on $G^k$. But, there are 
new nuances. Recall for $(X_1,...,X_n)$  to satisfy an n-tuple $((f_1^-,f_1^+,...,f_n^-,f_n^+)) \in \mathcal{B}_2^n$ \eqref{B2T2} of $2$ tail bounds means 
\begin{align*}
P(X_k \geq t) &\leq f_k^+(t), \ \forall t \geq 0 \\
P(X_k \leq t) &\leq f_k^-(t), \ \forall t \leq 0\\
\text{ with } f_k^-(0) &= f_k^+(0) =1
\end{align*}
for $k=1,...,n$. To shift $X_k$ ``outwards against its tails on $G^k \subset \mathbb{R}$," no longer means exclusively right or leftwards, but rightwards on $G^k$ for mass on the positive axis, and leftwards on $G^k$ for mass on the negative axis. 

\index{$2$ shift operator $\mathscr{S}_2$} Because of this, one must also specify the amount of mass, $c_k \in [0,1]$, that will be in $[0,\infty)$. Lastly, lemma~\ref{Xtilde} tells us the output can, in principle, be neat. The definition which meets these wants is
\begin{definition}[$2$ shift operator]\label{S2def}
The $2$ shift operator $\mathscr{S}_2$ is defined as
\begin{align*}
   \mathscr{S}_2: (f,G,c) \mapsto (X_1,...,X_n)
\end{align*} 
where $f = \big((f_1^-,f_1^+),...,(f_n^-,f_n^+)\big) \in \mathcal{B}_2^n, \ G \ni 0$ and is closed, $c=(c_1,...,c_n) \in [0,\infty]^n$, and for $k=1,...,n$, 
$X_k:(0,1) \rightarrow \mathbb{R}$ such that
\begin{equation}\label{S2Xkeq}
X_k(s) :=
\begin{cases}
   \mathscr{S}_L(f_k^-, G^k)(s) & 0 < s < c_k \\ 
   \mathscr{S}_R(f_k^+, G^k)(s) & c_k \leq s < 1
\end{cases}
\end{equation}
\end{definition}
By definition $6$ and definition $7$, \eqref{S2Xkeq} expands to\footnote{Because $0 < s < 1$, if $c_k = 0$ or $1$, it may be no $s$ satisfies $0 < s < c_k$ or $c_k \leq s <1$.}
$$
X_k(s) = 
\begin{cases}
   \inf \big(G^k \cup [\sup G^k, \infty)\big)_{\geq \inf {f_k^-}^{-1}[s,1]} & 0<s<c_k \\
   \sup \big(G^k \cup (-\infty, \inf G^k] \big)_{\leq \sup {f_k^+}^{-1}[1-s,1]} & c_k \leq s < 1
\end{cases}
$$
This can be simplified slightly: $f_k^+(0) = 1$, by definition of $\mathcal{B}_2$ \eqref{B2T2}, implies $ \sup {f_k^+}^{-1}[1-s,1] \geq 0, \ \forall s \in(0,1)$. Since $0 \in G^k$ by hypothesis, it follows
\begin{equation*}
\sup \big(G^k \cup (-\infty, \inf G^k] \big)_{\leq \sup {f_k^+}^{-1}[1-s,1]}  \geq 0.
\end{equation*}
Since $(-\infty, \inf G^k] \subset (-\infty, 0]$, this implies
\begin{equation}\label{SR2geq0}
0 \leq \sup \big(G^k \cup (-\infty, \inf G^k] \big)_{\leq \sup {f_k^+}^{-1}[1-s,1]} = \sup G^k_{\leq \sup {f_k^+}^{-1}[1-s,1]}, \ \forall s \in (0,1).
\end{equation}
By a symmetrical argument $0 \geq \inf \big(G^k \cup [\sup G^k, \infty)\big)_{\geq \inf {f_k^-}^{-1}[s,1]} = \\ \inf G^k_{\geq \inf {f_k^-}^{-1}[s,1]}$. Together,
\begin{equation}\label{condenseS2Xk}
\forall s \in(0,1), \ X_k(s) = 
\begin{cases}
   \inf G^k_{\geq \inf {f_k^-}^{-1}[s,1]} & 0<s<c_k \\
   \sup G^k_{\leq \sup {f_k^+}^{-1}[1-s,1]} & c_k \leq s < 1
\end{cases}
\end{equation}
Recall \eqref{condenseSR}. Similar to how $\mathscr{S}_R$ for $n=1$ allowed us to condense the full expression of $\mathscr{S}_R(f,G)$, $\mathscr{S}_2$ for $n=1$ allows us to write
\begin{equation}\label{S2condensed}
\mathscr{S}_2(f,G) = \Big(\mathscr{S}_2((f_1^-,f_1^+),G^1,c_1),...,\mathscr{S}_2((f_n^-,f_n^+),G^n,c_n)\Big)
\end{equation}
where $\mathscr{S}_2((f_k^-,f_k^+),G^k,c_k)$ equals \eqref{S2Xkeq} (equivalently \eqref{condenseS2Xk}), $k=1,...,n$. More useful, perhaps, is by \eqref{SR2geq0}, $\forall s \in (0,1)$,
\begin{align} 
\mathscr{S}_2((f_k^-,f_k^+),G^k,c_k) &= \mathscr{S}_R(f_k^+,G^k)(s) = \sup G^k_{\leq \sup {f_k^+}^{-1}[1-s,1]} \geq 0, s \in [c_k,1) \label{S2useful+}\\
\mathscr{S}_2((f_k^-,f_k^+),G^k,c_k) &= \mathscr{S}_L(f_k^-,G^k)(s) = \inf G^k_{\geq \inf {f_k^-}^{-1}[s,1]} \ \ \ \  \leq 0, s \in(0,c_k) \label{S2useful-}
\end{align}
As one can see, there is a lot of machinery in $\mathscr{S}_2$, too much to unpack, reliant on understanding the properties (not necessarily the formulas) of $\mathcal{N}$ (lemma~\ref{neatrvs}, 
lemma~\ref{Xtilde}), $\mathscr{S}_R$ (theorem~\ref{SRproperties}), and $\mathscr{S}_L$ (theorem~\ref{SLproperties}). 

Here is a picture.

\begin{figure}[H]
   \includegraphics[scale = .29, center]{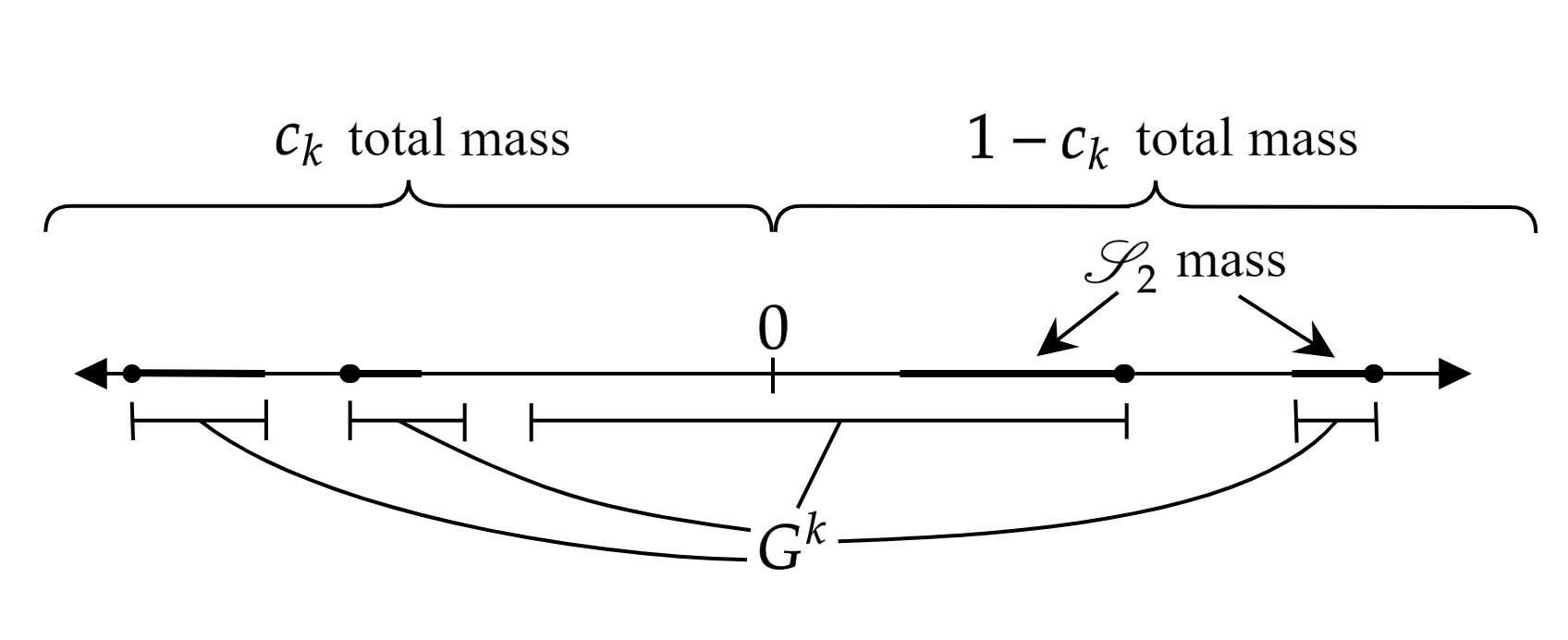}\label{randomcrap}
   \caption{Distribution of $\mathscr{S}_2((f_k^-,f_k^+),G^k,c_k)$}\label{fig.3}
\end{figure}
   
The bold dots represent a large accumulation of mass at a point. 

\begin{theorem}[$\mathscr{S}_2$ properties]\label{S2properties} Let $f \in \mathcal{B}_2^n, \ 0 \in G = G^1 \times \cdots \times G^n \subset \mathbb{R}^n$ be closed, and $c=(c_1,...,c_n) \in [0,1]^n$. Then
$\mathscr{S}_2(f,G,c)$ satisfies the following.
\begin{enumerate}[label = (\roman*)]
   \item $\mathscr{S}_2(f,G,c) \in \mathcal{T}_2(f)$
   \item $\forall s \in [c_k,1)$, $$\mathscr{S}_2\big((f_k^-,f_k^+),G^k,c_k\big)(s) = \mathscr{S}_R(f_k^+,G^k)(s) = \sup G^k_{\leq \sup {f_k^+}^{-1}[1-s,1]} \in G^k_{\geq 0} $$
   \item $\forall s \in (0,c_k)$, $$\mathscr{S}_2\big((f_k^-,f_k^+),G^k,c_k\big)(s) = \mathscr{S}_L(f_k^-,G^k)(s) = \inf G^k_{\geq \inf {f_k^-}^{-1}[s,1]} \in G^k_{\leq 0}$$
   \item $\mathscr{S}_2((f_k^-,f_k^+),G^k,c_k)(s) \in G^k, \ \forall s\in(0,1),$ for $k=1,...,n$
\end{enumerate}
\end{theorem}
\begin{proof}
(i): We must show for $k=1,...,n$, 
\begin{align*}
   1. \ & \mathscr{S}_2(f_k,G^k,c_k) \nearrow \text{ on } (0,1) \\
   2. \ & P\big(\mathscr{S}_2((f_k^-,f_k^+),G^k,c_k\big) \geq t) \leq f_k^+(t), \ \forall t \geq 0 \\ 
   3. \ &P\big(\mathscr{S}_2((f_k^-,f_k^+),G^k,c_k\big) \leq t) \leq f_k^-(t), \ \forall t \leq 0
\end{align*}
1. By \eqref{S2useful-}, $\forall s \in (0,c_k)$, $\mathscr{S}_2(f_k,G^k,c_k)(s) = \mathscr{S}_L(f_k,G^k)(s) \leq 0$,
which is monotonically increasing on $(0,c_k)$ by theorem~\ref{SLproperties} (i). By \eqref{S2useful+}, $\forall s \in [c_k,1)$, $\mathscr{S}_2(f_k,G^k,c_k)(s) \\ = \mathscr{S}_R(f_k,G^k)(s) \geq 0$,
which is monotonically increasing on $[c_k,1)$ by theorem~\ref{SRproperties} (i). 

2. and 3. follows easily from \eqref{S2useful+}, \eqref{S2useful-}, $P(\mathscr{S}_R(f_k^+,G^k) \geq t) \leq f_k^+(t), \ \forall t \in \mathbb{R}$ and $P(\mathscr{S}_L(f_k^-,G^k) \leq t) \leq f_k^-(t), \ \forall t\in \mathbb{R}$
(theorem~\ref{SRproperties} (i), theorem~\ref{SLproperties} (i)). 

(ii): By \eqref{S2useful+} we know $\mathscr{S}_R(f_k^+,G^k)(s) \geq 0, \ k =1,...,n$. This becomes \\ $\mathscr{S}_R(f_k,G^k)(s) \in G^k \cap [0,\infty)$ due to 
theorem~\ref{SRproperties} (v) and $0 \in G^k$.  

(iii): Similar to (ii), using \eqref{S2useful-} and theorem~\ref{SLproperties}. 

(iv) follows by (ii) and (iii).
\end{proof}
One can certainly pass down other properties of $\mathscr{S}_R$ and $\mathscr{S}_L$ to $\mathscr{S}_2$, using \eqref{S2useful+} and \eqref{S2useful-}, but it becomes a technical chore 
and would not help us. 

\section{Main result for the $2$ shift operator}\label{sectionS2mainresult}
In this section, let $f = ((f_1^-,f_1^+),...,(f_n^-,f_n^+)) \in \mathcal{B}_2^n$ and let $V\subset \mathbb{R}^n$ be closed. In this section we will prove \eqref{Ssup4}, which was
\begin{equation}\label{Ssup4again}
\sup_{X \in \mathcal{T}_2(f)} P(X \in V) = P(\mathscr{S}_2(f,G,c) \in V)
\end{equation}
for some $G$ (section~\ref{abbreviationschapter4}) and $c \in [0,1]^n$. This will be done in two steps: 

Step $1$: Consider a $(X^m)_m \subset \mathcal{T}_2(f)$ such that 
$$
P(X^m \in V)  \xrightarrow[m \rightarrow \infty]{} \sup_{X \in \mathcal{T}_2(f)} P(X \in V).
$$
By corollary~\ref{S2subseqconva.e.}, the $1$st component of $(X^m)_m$, $(\pi_1 \circ X^m)_m \subset \mathcal{T}_2(f_1^-,f_1^+)$, has a subsequence which converges a.e.\ to some $X_1 \in \mathcal{T}_2(f_1^-,f_1^+)$. 
Successively repeating this for the other components leads to a standard diagonalization, from which one obtains a subsequence $(X^{m_l})_l \subset (X^m)_m$, such that

\begin{equation}\label{xsitilde1}
X^{m_l} \xrightarrow[]{a.e.} (X_1,...,X_n) =: X_0 \in \mathcal{T}_2(f).
\end{equation}
Notice $P(X^{m_l} \in V) \xrightarrow[l \rightarrow \infty]{} \sup_{X \in \mathcal{T}_2(f)} P(X \in V)$. From this, the a.e.\ convergence, and $V$ being closed, we can show (see below) that 
$P(X_0 \in V) = \sup_{X \in \mathcal{T}_2(f)} P(X \in V)$.

Step $2$: Recall the informal argument at figure~\ref{fig.1}. We will use shift operators and the simplified domains of neat r.v.'s to make it rigorous., i.e.\ 
$$
\sup_{X \in \mathcal{T}_2(f)} P( X \in V) = P(X_0 \in V) \leq P(\mathscr{S}_2(f,G_{X_0},c_{X_0}) \in V) \text{ for some } G_{X_0},c_{X_0}. 
$$
Equality holds because $\mathscr{S}_2(f,G_{X_0},c_{X_0}) \in \mathcal{T}_2(f)$, which is \eqref{Ssup4again}.

Implementation of step $1$: 
\begin{lemma}\label{X0lemma}
$\exists X_0 \in \mathcal{T}_2(f)$ such that
$$
\sup_{X \in \mathcal{T}_2(f)} P(X \in V) = P(X_0 \in V).
$$
\end{lemma}
\begin{proof}
As described at \eqref{xsitilde1} $\exists (X^m)_{m \in \mathbb{N}} \subset \mathcal{T}_2(f), X_0 \in \mathcal{T}_2(f)$ such that\footnote{We have dropped the $l$ by reenumeration.}
\begin{equation}\label{twoassumptions}
   X^m \xrightarrow[]{a.e.} X_0 \in \mathcal{T}_2(f) \text{ and } 
   P(X^m \in V) \rightarrow \sup_{X \in \mathcal{T}_2(f)} P(X \in V).
\end{equation}
Denote $\mu$ the Lebesgue measure of $(0,1)^n$ (with the Borel $\sigma$-algebra). Define $W := \{w \in (0,1)^n \ \big| \ X^m(w) \xrightarrow[m \rightarrow \infty]{} X_0(w)\}$. By \eqref{twoassumptions}, $\mu(W^c) = 0$.\footnote{One might recall, the specific construction of the 
convergent a.e.\ subsequence in lemma~\ref{neatnessconvergence} converged for all $s \in(0,1)$ except countably many points. One could drag that thread all the way up here to get $W^c$ is countable if need be.} Consider 
\begin{equation}\label{X0lemmaWtildedef}
   \widetilde{W} := X_0^{-1}(V^c)\cap W
\end{equation}
which is measurable. Now, for $w \in \widetilde{W} \subset W$, $X^m(w)$ converges to $X_0(w) \in V^c$. Because $V^c$ is open, $\exists N \in \mathbb{N}$ such that 
$m \geq N \implies X^m(w) \notin V$. Define $N(w)$ the smallest of all such $N$; that is, 
$$
N(w) := \min \big\{N \in \mathbb{N} \ | \ \forall m \geq N, X^m(w) \notin V \big\}, \ \forall w \in \widetilde{W}.
$$
This allows classifying the points in $\widetilde{W}$: define 
\begin{equation}\label{widetildeWk}
   \widetilde{W}_k := \{w \in \widetilde{W} \ \big| \ N(w) = k\}, \ k = 1,2,...
\end{equation}
Notice each $w \in \widetilde{W}$ is an element of precisely one $\widetilde{W}_k$. In particular
\begin{equation}\label{widetildeWpartition}
   \widetilde{W} = \widetilde{W_1} \cup \widetilde{W}_2 \cup \cdots, \text{ and } \widetilde{W}_k \text{ are pairwise disjoint}.
\end{equation}
To show $\widetilde{W}_k$ are measurable: 

For $k=1,2,...$, define $W'_k := \bigcup_{m \geq k} {(X^m)}^{-1}(V)$, which are measurable. We claim 
$$
\forall k \in \mathbb{N}, \ \widetilde{W}_k = \big(W'_{k-1} \backslash W'_k\big) \cap \widetilde{W}
$$
where $W_0':= \widetilde{W}$. Proving this claim will show $\widetilde{W}_k$ are measurable. LHS, RHS will refer to this equation. For $k=1$,
\begin{align*} 
   w \in \text{ LHS} &\iff w \in \widetilde{W} \text{ and } \forall m \geq 1, X^m(w) \notin V \\ 
   &\iff w \in \text{ RHS}.
\end{align*} 
For $k \geq 2$, 
\begin{align*} 
w \in \text{ LHS } &\iff w \in \widetilde{W} \text{ and } N(w) = k \\ 
&\iff w \in \widetilde{W}, \ X^{k-1}(w) \in V, \text{ and } \forall m \geq k, X^m(w) \notin V \\ 
&\iff w \in \text{ RHS}. 
\end{align*} 
Therefore $\widetilde{W}_k$ are measurable. Now, by $\sigma$-additivity of $\mu$ applied repeatedly, 
\begin{align}
P(X^m \in V) &= \mu \big((X^m)^{-1}(V)\big) \nonumber \\
&=\mu\big((X^m)^{-1}(V) \cap W^c\big) + \mu\big({(X^m)}^{-1}(V) \cap W\big) \nonumber \\ 
&\stackrel{\mu(W^c)=0}{=}\mu\big((X^m)^{-1}(V)\cap W\big) \nonumber \\ 
&=\mu\big((X^m)^{-1}(V)\cap ((W\backslash{\widetilde{W}})\cup \widetilde{W})\big) \nonumber \\
&=\mu\big((X^m)^{-1}(V) \cap (W \backslash \widetilde{W})\big) + \mu\big((X^m)^{-1}(V)\cap \widetilde{W}\big) \nonumber \\
&\leq \mu\big(W \backslash \widetilde{W}\big) + \mu\big((X^m)^{-1}(V) \cap \widetilde{W}\big) \label{eqP(XinV)} 
\end{align}
The left term is $\mu(W \backslash \widetilde{W}) \stackrel{\eqref{X0lemmaWtildedef}}{=} \mu(W \backslash  (X_0^{-1}(V^c)\cap W)) 
= \mu(W \backslash X_0^{-1}(V^c)) = \mu(W \cap X_0^{-1}(V)) = \mu(X_0^{-1}(V))$ because $\mu(W^c) = 0$. 

For the right term, by definition of $N(w)$, notice for $w \in (X^m)^{-1}(V) \cap \widetilde{W}$ that $N(w) > m$. In particular, $w \notin \widetilde{W}_1, ... ,\widetilde{W}_m$. However, 
$\widetilde{W} \stackrel{\eqref{widetildeWpartition}}{=} \widetilde{W}_1 \cup \widetilde{W}_2 \cup \cdots$. It follows $(X^m)^{-1}(V) \cap \widetilde{W} = \bigcup_{k > m} \widetilde{W}_k$. Since $\widetilde{W}_k$ are disjoint,
$$
\mu\big((X^m)^{-1}(V) \cap \widetilde{W}\big) = \sum_{k > m} \mu(\widetilde{W}_k).
$$
Applying these simplifications of the left and right term in \eqref{eqP(XinV)},
$$
P(X^m \in V) = \mu(X_0^{-1}(V)) + \sum_{k > m} \mu(\widetilde{W}_k).
$$
The sum is finite, being less than $\mu(\widetilde{W})$. Taking $m \rightarrow \infty$ gives 
$$
\sup_{X \in \mathcal{T}_2(f)} P(X \in V) = P(X_0 \in V).
$$
\end{proof} 
Implementation of Step 2: 

Recall that the informal argument (figure~\ref{fig.1}) shifted the $X_k$ outwards on subsets against the tails by considering the (infinitesimal) masses of slices of $V \subset \mathbb{R}^n$. 
This can be made rigorous as follows.

The following abbreviation is useful for considering slices of $V$. 
\begin{equation}\label{notationpi}
\{x \in \mathbb{R}^n\}_{x_k = t} = \{(x_1,...,x_n) \in \mathbb{R}^n \ \big| \ x_k = t\}.
\end{equation}
\begin{definition}\label{hkdefinition}
For $k=1,...,n$, define
\begin{align}\label{hkdef}
\forall t \in \mathbb{R}, \ h_k(t) &:= P\big((X_1,...,X_{k-1},t,X_{k+1},...,X_n) \in V \big) \\ 
       &= P\big((X_1,...,X_{k-1},t,X_{k+1},...,X_n) \in V \cap \{x \in \mathbb{R}^n\}_{x_k = t}   \big) \label{hkdefeq2}
\end{align}
\end{definition}
Because $X_k$ are mutually ind., the, possibly infinitesimal, mass of the slice $V \cap \{x \in \mathbb{R}^n\}_{x_k = t}$ is, informally,
$$
``P\big((X_1,...,X_n) \in V \cap \{x \in \mathbb{R}^n\}_{x_k = t}\big) = h_k(t) P(X_k = t)".
$$
We need the following lemma, which is almost certainly an exercise somewhere or has been proven somewhere before.
\begin{lemma}\label{hklemma}
Let $(t_l)_{l \in \mathbb{N}} \subset \mathbb{R}, \ \lim_{l \rightarrow \infty} t_l = t_{\infty} \in \mathbb{R}$, and $h_k(t_l)$ converge as $l \rightarrow \infty$. Then
\begin{equation}
h_k(t_\infty) \geq \lim_{l \rightarrow \infty} h_k(t_l).
\end{equation}
\end{lemma}
\begin{proof}
WLOG $k=1$. Consider the slices of $V$, 
$$
S_l := V \cap \{x \in \mathbb{R}^n\}_{x_1 = t_l}, \ l=1,2,..., \text{ and } l=\infty.
$$
Consider the projections of the slices given by $(x_1,...,x_n) \mapsto (x_2,...,x_n)$. That is,
$$
S_l' := \{ x \in \mathbb{R}^{n-1} \ \big| \ (t_l,x) \in S_l\}, \ l=1,2,..., \text{ and } l=\infty.
$$
For $l = 1,2,....,$ and $l = \infty$, one has
\begin{align}
   h_1(t_l) &\stackrel{\eqref{hkdef}}{=} P\big((t_l,X_2,...,X_n) \in S_l) \big) \\ 
   & \ = P\big((X_2,...,X_n) \in S_l' \big). \label{hkxsi}
\end{align} 
$S_l$ is the intersection of two closed sets, thus closed. Furthermore, $S_l' \subset \mathbb{R}^{n-1}$ are closed: given $(x^m)_{m \in \mathbb{N}} \subset S_l', \ x^m \rightarrow x'$, one has 
$\big((t_l,x^m)\big)_m \subset S_l$ with $(t_l,x^m) \xrightarrow[m \rightarrow \infty]{} (t_l,x') \in S_l$, i.e.\ $x' \in S_l'$. 

Define
\begin{equation}\label{hklemmaQdef} 
Q:= \big\{x \in \mathbb{R}^{n-1} \ \big| \ x \in S_l' \text{ for infinitely many } l \in\mathbb{N} \big\}.
\end{equation} 

To show $Q$ is measurable: $Q^c$ are the points of $\mathbb{R}^{n-1}$ which are only in finitely many $S_l'$. In particular,
$$
\forall m \in \mathbb{N}, \ \bigcap_{l \geq m} S_l'^c \subset Q^c. 
$$
Conversely, if $q \in Q^c$, then 
$$
q \in \bigcap_{l \geq m} S_l'^c \text{ for } m>>0.
$$
Thus, $Q^c = \bigcup_{m \geq 1} \big(\bigcap_{l \geq m} S_l'^c \big) \implies Q$ is measurable. 

To show $Q \subset S_\infty'$: $\forall p \in Q, \ p$ is in infinitely many $S_l'$ by definition \eqref{hklemmaQdef}. Notice in such cases $(t_l,p) \in S_l$. This gives a sequence in $V$ converging to $(t_\infty,p) \in \overline{V} = V$. So $p \in S_\infty'$. 

By $\sigma$-additivity and monotonicity, 
\begin{align} 
   h_1(t_l) &\stackrel{\eqref{hkxsi}}{=} P \big((X_2,...,X_n) \in S_l' \big) \nonumber \\
   &= P\big((X_2,...,X_n) \in S_l' \cap Q\big) +P\big((X_2,...,X_n) \in S_l' \cap Q^c\big) \nonumber \\ 
   &\leq P\big((X_2,...,X_n) \in Q \big) +  P\big((X_2,...,X_n) \in S_l' \cap Q^c\big) \nonumber \\ 
   &\stackrel{Q \subset S_\infty'}{\leq} P\big((X_2,...,X_n) \in S_\infty' \big) + P\big((X_2,...,X_n) \in S_l' \cap Q^c\big) \nonumber \\ 
   &\stackrel{\eqref{hkxsi}}{=}h_1(t_\infty) + P\big((X_2,...,X_n) \in S_l' \cap Q^c\big). \label{hklemmashiteq}
\end{align} 
We will show second term goes to $0$ as $l \rightarrow \infty$. Define 
$$
\widetilde{Q}_l := \left\{q \in Q^c \ \big| \ q \in S_l' \text{ and } q \notin \bigcup_{m > l} S_m' \right\}, \ l = 1,2,...
$$
Define $\widetilde{Q}_0$ as those points in $Q$ which are not in any $S_l'$. $Q^c$ consists of the points of $\mathbb{R}^{n-1}$ which are in finitely many $S_l'$ \eqref{hklemmaQdef}. Notice then
$$
Q^c = \widetilde{Q}_0 \cup \widetilde{Q}_1 \cup \cdots
$$ 
Furthermore, $\widetilde{Q}_l$ are disjoint. Analagous to $Q, \ \widetilde{Q}_l$ can be shown to be measurable. This gives
\begin{align*}
   P\big((X_2,...,X_n) \in S_l' \cap Q^c\big) = m(S_l' \cap Q^c) &= m\big(S_l' \cap \big(\bigcup_{m \geq l} \widetilde{Q}_m\big)\big) \\  
&= \sum_{m \geq l} m(S_l' \cap \widetilde{Q}_m) \\
&\leq \sum_{m \geq l} m(\widetilde{Q}_m)
\end{align*}
Because $m(Q^c) = \sum_{m \geq 0} m(\widetilde{Q}_m)$ is finite, $\sum_{m \geq l} m(\widetilde{Q}_m)$ goes to $0$ as $l \rightarrow \infty$. Thus, $\lim_{l \rightarrow \infty} h_1(t_l) \leq h_1(t_\infty)$ by \eqref{hklemmashiteq}.

\end{proof}

\begin{theorem}[$2$ shift operator main result] \label{S2maintheorem}
Let $V \subset \mathbb{R}^n$ be closed. Then there exists $G \ni 0$ and $c = (c_1,...,c_n) \in [0,1]^n$ such that\footnote{$\mathscr{S}_2$ is at definition~\ref{S2def}.}
\begin{equation}\label{S2maintheoremeq}
\sup_{X \in \mathcal{T}_2(f)} P(X \in V) = P(\mathscr{S}_2(f,G,c) \in V)
\end{equation}
\end{theorem}

\begin{proof}
By lemma~\ref{X0lemma}, $\exists X_0 \in \mathcal{T}_2(f)$ such that
\begin{equation}\label{S2eq1}
\sup_{X \in \mathcal{T}_2(f)} P(X \in V) = P(X_0 \in V).
\end{equation}
We will show $\forall X \in \mathcal{T}_2(f)$, there exists a $0 \in G = G^1 \times \cdots \times G^n \subset \mathbb{R}^n$ closed, and $c \in [0,1]^n$ such that 
\begin{equation}\label{S2eq2}
P(X \in V) \leq P(\mathscr{S}_2(f,G,c) \in V). 
\end{equation}
After proving this, we may set $X=X_0$ in \eqref{S2eq2} to get 
$$
\sup_{X \in \mathcal{T}_2(f)} P(X \in V) \leq P(\mathscr{S}_2(f,G,c) \in V), \ \text{ some } G,c.
$$
Because $\mathscr{S}_2(f,G,c) \in \mathcal{T}_2(f)$ (theorem~\ref{S2properties}) equality holds, which is the theorem. Below we prove \eqref{S2eq2}. 

Suppose $X \in \mathcal{T}_2(f)$. Define ($h_k$ is from definition~\ref{hkdefinition})
\begin{align*}
H^1 &:= \big\{ t \in [0,\infty)\  \big| \ \forall t' \in [0,t), h_1(t') < h_1(t)\big \} \\ & \ \ \cup \big\{ t \in (-\infty,0] \ \big| \ \forall t' \in (t,0], h_1(t')<h_1(t) \big\} \\
G^1 &:= \overline{H^1} \cup \{0\}
\end{align*}

The proof will be sequential simplifications; there will be a sufficient condition for \eqref{S2eq2}, then there will be a simpler condition to show that condition, and so on, 
until we reach a condition simple enough to prove.

To prove \eqref{S2eq2}, it is sufficient to show 
\begin{equation}\label{S2eq4}
P(X \in V) \leq P\big((\mathscr{S}_2(f_1,G^1,c_1),X_2,...,X_n) \in V\big), \text{ where } c_1 = P(X_1<0).
\end{equation}
This is because $(\mathscr{S}_2(f_1,G^1,c_1),X_2,...,X_n) \in \mathcal{T}_2(f)$ ($\mathscr{S}_2(f_1,G^1,c_1) \in \mathcal{T}_2(f_1)$, theorem~\ref{S2properties} (i), and $X_k \in \mathcal{T}_2(f_k)$ by definition), 
hence one may induct \eqref{S2eq4} over the remaining components to get\footnote{This induction strategy will be used many more times in various proofs.}
\begin{align*}
&P\big((\mathscr{S}_2(f_1,G^1,c_1),X_2,...,X_n) \in V \big) \\ 
&\leq P\big(\mathscr{S}_2(f_1,G^1,c_1),\mathscr{S}_2(f_2,G^2,c_2),X_3,...,X_n) \in V \big) \\
&\leq \cdots \\
&\leq P\big((\mathscr{S}_2(f_1,G^1,c_1),...,\mathscr{S}_2(f_n,G^n,c_n)) \in V \big) \\ 
&= P(\mathscr{S}_2(f,G,c) \in V). 
\end{align*} 
where $c_k = P(X_k < 0), \forall k$. 
Denote $m$ the Lebesgue measure on $(0,1)$ (with the Borel $\sigma$-algebra) and ${X_1}_\#(m)$ its pushforward \eqref{pushforwarddef} to $\mathbb{R}$ via $X_1$. 
Because $h_1(t) = P((t,X_2,...,X_n) \in V), \ \forall t \in \mathbb{R}$, and $X_k$ are mutually ind., one can apply Fubini's theorem (theorem~\ref{Fubini'stheorem}) to indicator function of $V$, and integrate over the first component to get
$$
P(X \in V) = \int_{\mathbb{R}} h_1 \ d{X_1}_\#(m).
$$
Similarly, 
$$
P\big((\mathscr{S}_2(f_1,G^1,c_1),X_2,...,X_n) \in V\big) = \int_{\mathbb{R}} \ h_1 \ d\mathscr{S}_2(f_1, G^1,c_1)_\#(m).
$$
By change of variables for Lebesgue integration,\footnote{The informal argument (figure~\ref{fig.1}) interprets as the $1$st integral is $\leq$ the $2$nd. But, how to actually compare these quantities is nebulous at first. It is the same situation as example $2$'s equation \eqref{eg2eq},
which was resolved by making the domains of the $X_k$ as simple as possible. Neat r.v.'s will work here too.}
\begin{align} 
   \int_\mathbb{R} h_1 \ d{X_1}_\# (m) &= \int_{(0,1)} h_1 \circ X_1 \ dm \label{xsiint1} \\
   \int_\mathbb{R} h_1 \ d\mathscr{S}_2(f_1, G^1,c_1)_\#(m) &= \int_{(0,1)} h_1 \circ \mathscr{S}_2(f_1,G^1,c_1) \ dm. \label{xsiint2}
\end{align} 
To compare these integrals, it is sufficient to show (by standard properties)
\begin{equation}\label{S2eq6}
\forall s \in(0,c_1) \cup (c_1,1), \ h_1(X_1(s)) \leq h_1 \big(\mathscr{S}_2(f_1, G^1,c_1)(s) \big).
\end{equation}
There are two cases: $0<s<c_1$ and $c_1 < s < 1$, corresponding to the negative and positive axis respectively, by \eqref{S2useful+} and \eqref{S2useful-}. We will only prove 
\eqref{S2eq6} for $c_1 < s < 1$ (we assume $c_1 \neq 1$), because the case $0<s<c_1$ is done similarly, being symmetric. Write 
$$
\widetilde{X}_1^+(s) = \sup {f_1^+}^{-1}[1-s,1], \ \forall s \in(0,1).
$$
One has $\widetilde{X}_1^+ \geq 0$, because $f_1^+(0) = 1$. By \eqref{S2useful+}, $\forall s \in(c_1,1)$, 
\begin{equation}\label{S2maintheoremeqskripka1}
\mathscr{S}_2(f_1,G^1,c_1)(s) = \mathscr{S}_R(f_1^+,G^1)(s) = \sup G^1_{\leq \widetilde{X}_1^+(s)}.
\end{equation}
Notice $\forall s \in (c_1, 1), \ X_1(s) \geq 0$. Indeed, if this were not the case, because $X \nearrow$ on $(0,1)$, one quickly finds $P(X_1 < 0) > c_1$, but $c_1 = P(X_1 < 0)$ by definition.

Furthermore, because $X_1 \in \mathcal{T}_2(f_1^-,f_1^+)$, $X_1$ is right tail bounded by $f_1^+$. 
In particular, lemma~\ref{Xtilde} (iii) says
$$
X_1(s) \leq \widetilde{X}_1^+(s).
$$
Thus $X_1(s) \in [0, \widetilde{X}_1^+(s)], \ \forall s \in(c_1,1)$. Thus, to show \eqref{S2eq6} for $s \in (c_1,1)$, it is sufficient to show (by \eqref{S2maintheoremeqskripka1}) that 
\begin{equation}\label{S2claimeq}
\sup_{t \in [0,\widetilde{X}_1^+(s)]} h_1(t) = h_1\big(\sup G^1_{\leq \widetilde{X}_1^+(s)}\big).
\end{equation}
This is confined within the next lemma.

\begin{lemma}\label{S2claimeqlemma}
Let $(f_1^-,f_1^+) \in \mathcal{B}_2$ and $c_1 \in (0,1)$. Define\footnote{$h_k$ is at definition~\ref{hkdefinition}.}
\begin{align*} 
H^1 &:= \big\{ t \in [0,\infty)\  \big| \ \forall t' \in [0,t), h_1(t') < h_1(t)\big \} \\ 
    & \ \ \cup \big\{ t \in (-\infty,0] \ \big| \ \forall t' \in (t,0], h_1(t')<h_1(t) \big\} \\
G^1 &:= \overline{H^1} \cup \{0\}.
\end{align*}
Write $\widetilde{X}_1^+(s) = \sup {f_1^+}^{-1}[1-s,1], \ \forall s \in(0,1)$. Then
\begin{equation}\label{S2claimeq'}
   \sup_{t \in [0,\widetilde{X}_1^+(s)]} h_1(t) = h_1\big(\sup G^1_{\leq \widetilde{X}_1^+(s)}\big), \ \forall s \in (c_k,1).
\end{equation}
\end{lemma}
\begin{proof}

Fix $s \in (c_k,1)$. By theorem~\ref{S2properties} (ii), 
$$
\mathscr{S}_2((f_1^-,f_1^+),G^1,c_1)(s) = \sup G^1_{\leq \widetilde{X}_1^+(s)} \in G^1 = \overline{H^1} \cup \{0\}.
$$
Thus $\exists (t_k')_k \subset H^1 \cup \{0\}$ with $t_k' \rightarrow \sup G^1_{\leq \widetilde{X}_1^+(s)}$. By discarding terms, WLOG 
$(t_k')_k$ is either monotonically increasing or monotonically decreasing. This way, the definition of $H^1$ implies $(h_1(t_k'))_k$ is monotonically
increasing or monotonically decreasing. In particular, it converges. By lemma~\ref{hklemma},
\begin{equation}\label{bytheaboveequation}
h_1 (\sup G^1_{\leq \widetilde{X}_1^+(s)}) \geq \lim_{k \rightarrow \infty} h_1(t_k'). 
\end{equation}
Note $\sup G^1_{\leq \widetilde{X}_1^+(s)} \leq \widetilde{X}^+(s)$. The following claim will be used to prove \eqref{S2claimeq'} (it is not sufficient).

Claim: 
\begin{equation}\label{h1itsownsup}
h_1 (\sup G^1_{\leq \widetilde{X}_1^+(s)}) \geq h_1(t), \ \forall t \in [0,\sup G^1_{\leq \widetilde{X}_1^+(s)}].
\end{equation}
Indeed, suppose $t \in [0, \sup G^1_{\leq \widetilde{X}_1^+(s)}), \ \sup G^1_{\leq \widetilde{X}_1^+(s)} > 0$. Then $\forall k>>0, \ t_k' > t$. Applying the definition of $H^1$, 
$\forall k >> 0, \ h_1(t_k') > h_1(t)$. The claim follows by \eqref{bytheaboveequation}.

Now, let $(t_l)_{l \in \mathbb{N}} \subset [0,\widetilde{X}_1^+(s)]$ such that 
$$
h_1(t_l) \xrightarrow[l \rightarrow \infty]{} \sup_{t \in [0,\widetilde{X}_1^+(s)]} h_1(t).
$$  
By compactness, $(t_l)_l$ has a convergent subsequence $(t_{l_k})_k$ with $t_{l_k} \rightarrow t_{\infty}$, some $t_\infty \in [0,\widetilde{X}^+_1(s)]$. 
By lemma~\ref{hklemma}, 
\begin{equation}\label{hk(tinfty)}
h_1(t_\infty) \geq \lim_{k \rightarrow \infty} h_1(t_{l_k}) = \sup_{t \in[0,\widetilde{X}_1^+(s)]} h_1(t) \geq h_1(t_\infty),
\end{equation}
hence \eqref{hk(tinfty)} holds with equalities in plcae of inequalities. 

Suppose, by contradiction, that $t_\infty > \sup G^1_{\leq \widetilde{X}_1^+(s)}$. So $t_\infty \notin G^1 \supset H^1$. Then by definition of $H^1$, 
there must exist $t' \in [0,t_\infty)$ such that $h_1(t') > h_1(t_\infty)$. A contradiction by \eqref{hk(tinfty)}.

Therefore $0 \leq t_\infty \leq \sup G^1_{\leq \widetilde{X}_1^+(s)}$. Setting $t = t_\infty$ in \eqref{h1itsownsup} gives \eqref{S2claimeq'} by \eqref{hk(tinfty)}.
\end{proof}
Lemma~\ref{S2claimeqlemma} was the proof of \eqref{S2claimeq} which, as described there, finishes the proof of theorem~\ref{S2maintheorem}. 
\end{proof}

For $A \subset \mathbb{R}^n$, define $\#_0(A)$ as the smallest product subset containing $A \cup \{0\}$. \index{$\#_0$} That is, 
\begin{equation}\label{defgridfunction}
\#_0(A) := \pi_1(A) \cup \{0\} \times \cdots \times \pi_n(A) \cup \{0\}, \ \forall A \subset \mathbb{R}^n.
\end{equation}
\begin{lemma}\label{gridfnproperties}
Let $G \ni 0$ and $V \subset \mathbb{R}^n$. $\#_0$ satisfies the following. 
\begin{enumerate}[label=(\roman*)]
\item $\forall A \subset B \subset \mathbb{R}^n$, $A \subset \#_0(A) \subset \#_0(B)$ 
\item $\#_0(G) = G$ 
\item $\#_0(G \cap V) \cap V = G \cap V$
\item $\forall A \subset \mathbb{R}^n, \ \overline{\#_0\big( \overline{A}\big) } = \overline{\#_0(A)}$
\end{enumerate} 
\end{lemma} 
\begin{proof} 
(i) and (ii) are immediate by \eqref{defgridfunction}. (iii) is since by (i) $\#_0(G \cap V) \cap V \supset G \cap V$ whereas 
$$
\#_0(G \cap V) \subset \#_0(G) \stackrel{(ii)}{=} G \implies \#_0(G \cap V) \cap V \subset G \cap V.
$$
(iv) is by \eqref{defgridfunction} and standard closure results.
\end{proof} 

The next theorem (with the previous lemma) implies WLOG the $G$ in theorem~\ref{S2maintheorem} satisfies $G = \overline{\#_0(G \cap V)}$. 
This is useful when $V$ is finite for the following reason: typically one must check infinitely many possible $G$ to find the $G$ which satisfies theorem~\ref{S2maintheorem}. 
However if we know $G=\overline{\#_0(G \cap V)}$ this task is significantly simplified; when $V$ is finite there are only finite possible sets 
$\overline{\#_0(G \cap V)}$ can be over all possible inputs of $G$. In other words, as a function of $G$ it has an infinite domain, but finite range. Thus 
combined with $G = \overline{\#_0(G \cap V)}$ there are now only finitely many $G$ to check for theorem~\ref{S2maintheorem} as opposed to infinitely many $G$.

\begin{theorem}\label{S2finiteV}
Let $f \in \mathcal{B}_2^n, \ V \subset \mathbb{R}^n$ be closed, $G \ni 0$ closed, and $c \in [0,1]^n$. Then
\begin{equation}\label{S2finiteVequation} 
 P(\mathscr{S}_2(f,G,c) \in V) \leq P\big(\mathscr{S}_2\big(f,\overline{\#_0(G \cap V)},c\big) \in V\big).
\end{equation} 
\end{theorem}
\begin{proof}
Write $f = (f_1,...,f_n)$ where $f_k = (f_k^-,f_k^+) \in \mathcal{B}_2$. Also write
$$
\overline{\#_0(G \cap V)} = G' = {G^1}'\times \cdots \times{G^n}'
$$
where ${G^k}' = \overline{\pi_k(G \cap V)} \cup \{0\}$. We claim for $k=2,...,n$ that 
\begin{equation}\label{S2finiteVinductionclaim} 
\overline{\pi_k\big({G^1}' \times \cdots \times {G^{k-1}}' \times G^k \times \cdots \times G^n \cap V\big )} \cup \{0\} = {G^k}'.
\end{equation} 
LHS, RHS will refer to this equation, which we prove now. Notice $G' \subset G$ since $G = G^1 \times \cdots \times G^n \ni 0$ is closed. In particular, ${G^j}' \subset G^j, \ j = 1,...,n$. Thus  
$\text{LHS} \subset \overline{\pi_k(G \cap V)} \cup \{0\}= \text{RHS}$. 

Conversely, suppose $x_k \in \text{ RHS}$. If $x_k = 0$ then $x_k \in \text{ LHS}$. If $x_k \neq 0$ then $\exists \{y_j\}_{j \in \mathbb{N}} \subset G \cap V, \ \pi_k(y_j)\xrightarrow[j \rightarrow \infty]{} x_k$. By lemma~\ref{gridfnproperties} (iii) and standard closure properties, 

\begin{align*} 
G \cap V = \overline{G \cap V} &= \overline{\#_0(G \cap V) \cap V} \\ 
                                                 &\subset \overline{\#_0(G \cap V)} \cap \overline{V} \\ 
                                                 &= G' \cap V. 
\end{align*} 
It follows $\{y_j\}_j \subset G' \cap V$. Since ${G^j}' \subset G^j, \ j = 1,...,n$, it follows $x_k \in \text{LHS}$. Thus LHS $\supset$ RHS. This shows \eqref{S2finiteVinductionclaim}. 

Next we claim
\begin{equation}\label{S2finiteVeq}
P\big(\mathscr{S}_2(f,G,c) \in V \big) \leq P\big((\mathscr{S}_2(f,{G^1}' \times G^2 \times \cdots \times G^n,c) \in V \big).  
\end{equation}
Having shown this, since the first component and $G$ was arbitrary, one may reapply it to ${G^1}' \times G^2 \times \cdots \times G^n$ to get
$$
P\big((\mathscr{S}_2(f,{G^1}' \times G^2 \times \cdots \times G^n,c) \in V \big) \leq P\big((\mathscr{S}_2(f,{G^1}' \times {G^2}'' \times G^3 \cdots \times G^n,c) \in V \big)
$$
where ${G^2}'' = \overline{\pi_2({G^1}' \times G^2 \times \cdots \times G^n)} \cup \{0\}$. By \eqref{S2finiteVinductionclaim}, ${G^2}'' = {G^2}'$. Thus 
$$
P\big((\mathscr{S}_2(f,{G^1}' \times G^2 \times \cdots \times G^n,c) \in V \big) \leq P\big((\mathscr{S}_2(f,{G^1}' \times {G^2}' \times G^3 \cdots \times G^n,c) \in V \big).
$$
Repeating this $n-2$ more times, one gets \eqref{S2finiteVequation}.    

LHS and RHS will refer to \eqref{S2finiteVeq}, which we prove now. Applying theorem~\ref{Fubini'stheorem} to the indicator function of $V$ and  Lebesgue integrating over the $1$st component
\begin{align*} 
\text{LHS} &= \int_\mathbb{R} h_1 \ d \mathscr{S}_2(f_1,G^1,c_1)_\#(m)  \\ 
\text{RHS} &= \int_\mathbb{R} h_1 \ d \mathscr{S}_2(f_1,{G^1}',c_1)_\#(m).
\end{align*} 
where $m$ is the Lebesgue measure on $(0,1)$ with the Borel $\sigma$-algebra and $h_1$ is defined at definition~\ref{hkdefinition}. By change of variables (recall $\mathscr{S}(f,G,c) \in \mathcal{N}^n$, which has mutually ind.\ components), 
\begin{align*} 
\text{LHS} &= \int_{(0,1)} h_1 \circ \mathscr{S}_2(f_1,G^1,c_1) \ dm \\ 
\text{RHS} &= \int_{(0,1)} h_1 \circ \mathscr{S}_2(f_1,{G^1}',c_1) \ dm
\end{align*}

To show LHS $\leq$ RHS, by monotonicity, we may show 
\begin{equation}\label{S2mu2} 
h_1\big(\mathscr{S}_2(f_1,G^1,c_1)(s) \big) \leq h_1 \big(\mathscr{S}_2(f_1,{G^1}',c_1)(s) \big), \ \forall s \in (0,1).
\end{equation}
Restating \eqref{S2useful+} and \eqref{S2useful-}, 
\begin{equation}\label{S2mu3}
\mathscr{S}_2(f_1,G^1,c_1)(s) =
\begin{cases}
   \mathscr{S}_L(f_1^-,G^1)(s) & s \in (0,c_1)\\ 
   \mathscr{S}_R(f_1^+,G^1)(s) & s \in [c_1,1) 
\end{cases}
\end{equation} 
with $\forall s \in(0,c_1), \ \mathscr{S}_L(f_1^-,G^1)(s) \leq 0$ and $\forall s \in [c_1,1), \ \mathscr{S}_R(f_1^+,G^1)(s) \geq 0$. Similarly for $\mathscr{S}_2(f_1,{G^1}',c_1)$. 

There are two cases in \eqref{S2mu2}.

Case $\mathscr{S}_2(f_1,G^1,c_1)(s) \notin \overline{\pi_1(G \cap V)} \cup \{0\}$: 

Denoting $\mathscr{S}_2(f_1,G^1,c_1)(s) = a$, 
\begin{equation}\label{S2maincase1} 
   h_1(a) \stackrel{\eqref{hkdef}}{=}  P\big(\big(a, \mathscr{S}_2(f_2,G^2,c_2),...,\mathscr{S}_2(f_n,G^n,c_n)\big) \in V\big).
\end{equation}
For \eqref{S2mu2}, because $h_1 \geq 0$, it suffices to show $h_1(a) = 0$. Suppose by contradiction $\exists s_2,...,s_n \in (0,1)$ such that 
$$
\big(a, \mathscr{S}_2(f_2,G^2,c_2)(s_2),...,\mathscr{S}_n(f_n,G^n,c_n)(s_n)\big) \in V.
$$
By theorem~\ref{S2properties} (iv), this is also an element of $G$. Thus $a \in \pi_1(G \cap V)$. A contradiction; no such $s_2,...,s_n$ exist. Therefore \eqref{S2maincase1} is $0$.  

Case $\mathscr{S}_2(f_1,G^1,c_1)(s) \in \overline{\pi_1(G \cap V)} \cup \{0\}$: 

We claim 
\begin{equation}\label{S2mu4}
   \mathscr{S}_2(f_1,{G^1}',c_1)(s) = \mathscr{S}_2(f_1,G^1,c_1)(s)
\end{equation} 
from which \eqref{S2mu2} follows. Either $s \in (0,c_1)$ or $s \in [c_1,1)$. Say $s \geq c_1$. Then,
\begin{align} 
\overline{\pi_1(G \cap V)} \cup \{0\} \ni \mathscr{S}_2(f_1,G^1,c_1)(s) & \stackrel{\eqref{S2mu3}}{=} \mathscr{S}_R(f_1^+,G^1)(s) \\ 
   &\ = \sup G^1_{\leq \sup {f_1^+}^{-1}[1-s,1]} \geq 0 \label{S2mu5}
\end{align}
where the last step is by theorem~\ref{S2properties} (ii). That is, $\mathscr{S}_2(f_1,G^1,c_1)(s) \geq 0$ is an element of $\overline{\pi_1(G \cap V)} \cup \{0\}$ which achieves the supremum of 
$G^1_{\leq \sup {f_1^+}^{-1}[1-s,1]}$, therefore it achieves the supremum of $\big(\overline{\pi_1(G \cap V)} \cup \{0\}\big) \cap (-\infty, {f_1^+}^{-1}[1-s,1]]$, which is the definition of $\mathscr{S}_2(f_1,{G^1}',c_1)(s)$ \eqref{S2mu3}. 
Thus \eqref{S2mu4} holds for $s \in [c_1,1)$.

The case $s \in (0,c_1)$ is similar, because $\mathscr{S}_L$ is the reflection of $\mathscr{S}_R$. So \eqref{S2mu2} holds as desired.
\end{proof}

\section{Main result for the absolute shift operator}\label{sectionabsoluteshiftoperatormainresult}

\begin{definition}[Radially neat r.v.'s]\label{radiallyneatdef} 
A real r.v.\ $X$ is radially neat iff $X:(0,1) \rightarrow \mathbb{R}$ and $|X| \nearrow$. (In particular, $|X| \in \mathcal{N})$. The set of radially neat r.v.'s is $\mathcal{N}_{rad}$. 
\end{definition} 

Whereas $\mathscr{S}_R, \mathscr{S}_L,$ and $\mathscr{S}_2$ returned a $n$-tuple of neat r.v.'s, $\mathscr{S}$ will return a $n$-tuple of radially neat r.v.'s. 
\begin{definition}\label{Sdef}
The absolute shift operator $\mathscr{S}$ maps
$$
\mathscr{S}: (f,G) \mapsto (X_1,...,X_n)
$$
where $f=(f_1,...,f_n) \in \mathcal{B}^n, \ 0 \in G = G^1 \times \cdots \times G^n \subset \mathbb{R}^n$, $G^k$ Borel measurable such that $G^k_{|\cdot|}$ is closed for $k=1,...,n$, $(X_1,...,X_n)$ is the mutually ind.\ product of 
the r.v.'s $X_1,...,X_n$, and for $k=1,...,n$,
\begin{equation}\label{SXkeq} 
   \forall s \in (0,1), \ X_k(s) := 
   \begin{cases} 
      \mathscr{S}_R(f_k,G^k_{|\cdot|})(s) & \text{if } \mathscr{S}_R(f_k,G^k_{|\cdot|})(s) \in G^k \cap [0,\infty)  \\
      -\mathscr{S}_R(f_k,G^k_{|\cdot|})(s) & \textnormal{otherwise}
   \end{cases} 
\end{equation}
\end{definition} 
\index{Absolute shift operator $\mathscr{S}$} $V_{|\cdot|}$ is defined in section~\ref{abbreviationschapter4} along with the other subscript notations. $\mathscr{S}_R$ and $\widetilde{X}_k$ are at definition~\ref{SRdef} and \eqref{condenseSR}. 
Because $f_k(0) = 1$ \eqref{absolutetails}, by \eqref{Xtildeabbreviations} we have
$$
\forall s \in (0,1), \ 0 \in f_k^{-1} [1-s,1] \implies \widetilde{X}_k \geq 0.
$$
Because $0 \in G^k$ by definition of $\mathscr{S}$, trivially $0 \in G^k_{|\cdot|}$. Combining this with $\widetilde{X}_k(s) \geq 0$,
\begin{align} 
\mathscr{S}_R(f_k,G^k_{|\cdot|})(s) &\stackrel{\eqref{SRXkeq}}{=} \sup \big(G^k_{|\cdot|}\cup(-\infty, \inf G^k_{|\cdot|}] \big)_{\leq \widetilde{X}_k(s)} \nonumber \\ 
&\ = \sup \big(G^k_{|\cdot|}\big)_{\leq \widetilde{X}_k(s)} \geq 0 \label{SSRgeq0}
\end{align}
is an expansion $\mathscr{S}_R$ in \eqref{SXkeq}. \eqref{SSRgeq0} is an important fact to remember: by \eqref{SXkeq} it follows if $\mathscr{S}_R(f_k,G_{|\cdot|}^k) \in G^k_{\geq 0}$ then $\mathscr{S}(f_k,G^k)(s) \geq 0$, and if $\mathscr{S}_R(f_k,G_{|\cdot|}^k)(s) \notin G^k_{\geq 0}$ 
then $\mathscr{S}(f_k,G^k)(s) \leq 0$. We will use these simple, but ever present facts continuously. For example, 
\begin{equation}\label{|S|Xkeq} 
\forall s \in (0,1), \ |\mathscr{S}(f_k,G^k)(s)| \stackrel{\eqref{SXkeq}}{=} |\mathscr{S}_R(f_k,G^k_{|\cdot|})(s)| = \mathscr{S}_R(f_k,G^k_{|\cdot|})(s).
\end{equation} 

Similar to the other shift operators (see \eqref{condenseSR}), the definition of $\mathscr{S}$ for $n=1$ allows the shorter expression
$$
\mathscr{S}(f,G) = \big(\mathscr{S}(f_1,G^1),...,\mathscr{S}(f_n,G^n)\big).
$$

\begin{theorem}[$\mathscr{S}$ properties]\label{Sproperties} 
$\mathscr{S}(f,G)$ (defined at definition~\ref{Sdef}) satisfies
   \begin{enumerate}[label=(\roman*)]
      \item $\mathscr{S}(f,G)(s) \in G, \ \forall s \in(0,1)^n$
      \item $\mathscr{S}(f,G)$ is Borel measurable
      \item For $k=1,...,n$, $\mathscr{S}(f_k,G^k)$ is a radially neat r.v. (definition~\ref{radiallyneatdef})
      \item For $k=1,...,n$, $P\big(|\mathscr{S}(f_k,G^k)| \geq t \big) \leq f_k(t), \forall t \geq 0$ with equality $\forall t \in G^k_{|\cdot|}$
   \end{enumerate} 
\end{theorem}
   \begin{proof} 
   (i): Fix $k$, $s_k \in (0,1)$. It is sufficient to show $\mathscr{S}(f_k,G^k)(s_k) \in G^k$. This is an application of \eqref{SXkeq} and \eqref{SSRgeq0}: if $\mathscr{S}_R(f_k,G^k_{|\cdot|})(s) \in G^k \cap [0,\infty)$ there is nothing to show. 
        Therefore, suppose this is not the case. By theorem~\ref{SRproperties} (v), 
\begin{equation}\label{SSRinG}
\mathscr{S}_R(f_k,G^k_{|\cdot|})(s) \in G^k_{|\cdot|} \cup (-\infty, \inf G^k_{|\cdot|}], \ \forall s \in (0,1).
\end{equation} 
Recall 
        $$ 
        G^k_{|\cdot|} = \{|t| \ \big| \ t \in G^k\} \ \ \ \ \ \ (\text{section~\ref{abbreviationschapter4}}).
        $$
In particular because $0 \in G^k$ by hypothesis, $0 \in G^k_{|\cdot|} \subset [0,\infty)$. By \eqref{SSRgeq0} and \eqref{SSRinG}, 
\begin{equation}\label{SSRinG2} 
\mathscr{S}_R(f_k,G^k_{|\cdot|})(s) \in G^k_{|\cdot|}, \ \forall s \in (0,1). 
\end{equation}
Because $\mathscr{S}_R(f_k,G^k_{|\cdot|})(s) \notin G^k \cap [0,\infty)$ by assumption, \eqref{SSRinG2} implies $-\mathscr{S}_R(f_k,G^k_{|\cdot|}) \in G^k \cap (-\infty, 0]$. 
Applying \eqref{SXkeq} and \eqref{SSRgeq0} trivially implies $\mathscr{S}(f_k,G^k)(s) \in G^k$. 

   (ii): It is sufficient to show each component of $\mathscr{S}(f,G)$ is measurable. Fix $k$. Let $U \subset \mathbb{R}$ be measurable. We will show $\mathscr{S}(f_k,G^k)^{-1}(U)$ is measurable. 
   Firstly,
   $$
   \mathscr{S}(f_k,G^k)^{-1}(U) \stackrel{\text{(i)}}{=} \mathscr{S}(f_k,G^k)^{-1}(U \cap G^k). 
   $$
   We claim (section~\ref{abbreviationschapter4} notations)
   \begin{align}\label{(*)claim}  
      \mathscr{S}(f_k,G^k)^{-1}(U \cap G^k) &= \mathscr{S}_R(f_k,G^k)^{-1}(U \cap G^k_{\geq 0}) \nonumber \\ 
                                                                           &\cup \mathscr{S}_R(f_k,G^k)^{-1}\big((U \cap G^k_{<0})_- \backslash G^k_{\geq0}\big).
   \end{align} 
   Then, because $U \cap G^k_{\geq 0}$ and $(U \cap G^k_{<0})_- \backslash G^k_{\geq0}$ are measurable, being a finite combination of unions, intersections, reflections, and complements of measurable 
   sets, and because $\mathscr{S}_R(f_k,G^k_{|\cdot|})$ is a neat r.v., it will follow $\mathscr{S}(f_k,G^k)^{-1}(U)$ is measurable.
   \begin{align*}
   &s \in \mathscr{S}(f_k,G^k)^{-1}(U \cap G^k) \\ &\iff \mathscr{S}(f_k,G^k)(s) \in U \cap G^k_{\geq 0} \text{ or } \mathscr{S}(f_k,G^k)(s) \in U \cap G^k_{<0}.
   \end{align*}
   Using \eqref{SXkeq} and $\mathscr{S}_R(f_k,G^k_{|\cdot|}) \geq 0$ by \eqref{SSRgeq0}, this is iff
   \begin{align*} 
   &\bigg(\mathscr{S}_R(f_k,G^k_{|\cdot|})(s) \in U \cap G^k_{\geq 0} \text{ or } \Big(- \mathscr{S}_R(f_k,G^k_{|\cdot|})(s) \in U \cap G^k_{<0} \\ & \text{and } \mathscr{S}_R(f_k,G^k_{|\cdot|})(s) \notin G^k_{\geq 0}\Big)\bigg) \\ &\iff \Big(s \in \mathscr{S}(f_k,G^k)^{-1} (U \cap G^k_{\geq 0}) \text{ or } s \in \mathscr{S}(f_k,G^k)^{-1} \big((U \cap G^k_{>0})_{-} \backslash G^k_{\geq 0}\big)\Big) 
   \end{align*} 
Thus \eqref{(*)claim} is proved.

   (iii): By definition~\ref{radiallyneatdef}, we must show $|\mathscr{S}(f_k,G^k)| \in \mathcal{N}$ and $\mathscr{S}(f_k,G^k)$ is measurable. The latter property is (ii). The former follows since 
   $|\mathscr{S}(f_k,G^k)| \stackrel{\eqref{|S|Xkeq}}{=} \mathscr{S}_R(f_k,G^k_{|\cdot|})$  and the function $\mathscr{S}_R$ always returns a neat r.v.\ for any input by theorem~\ref{SRproperties} (iv). 

   (iv): Because $|\mathscr{S}(f_k,G^k)| = \mathscr{S}_R(f_k,G^k_{|\cdot|})$, this is follows by theorem~\ref{SRproperties} (i) and (ii).
   
\end{proof} 

We may attempt to find the CDF of $\mathscr{S}(f_k,G^k)$: \eqref{|S|Xkeq} allows applying \eqref{SRRCDF}, which provides the reversed CDF of $\mathscr{S}_R(f_k,G^k_{|\cdot|})(s)$, to get
\begin{equation}\label{SRCDF} 
   P\big(|\mathscr{S}(f_k,G^k)| \geq t\big) =  f_k\big(\inf \big(G^k_{|\cdot|} \cup (-\infty, \inf G^k_{|\cdot|}]\big)_{\geq t}\big), \ \forall t \geq 0
\end{equation} 
employing the conventions $\inf \ \emptyset = +\infty$ and $f_k(+\infty) = 0$. Because $ 0 \in G^k$ notice $0 \in G^k_{|\cdot|}$. In particular, this simplifies to 
\begin{equation}\label{SRCDF2}
P\big(|\mathscr{S}(f_k,G^k)| \geq t\big) =  f_k\big(\inf \big(G^k_{|\cdot|}\big)_{\geq t}\big), \ \forall t \geq 0.
\end{equation}
Recall the statement above \eqref{|S|Xkeq}. Expressing the reversed CDF slightly more specifically, for $t \geq 0$, 
\begin{align}
   P\big(\mathscr{S}(f_k,G^k) \geq t \big) &\stackrel{\eqref{SXkeq}}{=} P\big(\mathscr{S}_R(f_k,G^k_{|\cdot|}) \geq t \text{ and } \mathscr{S}_R(f_k,G^k_{|\cdot|}) \in G^k_{\geq 0}  \big) \nonumber \\
   &\ = P\big( \mathscr{S}_R(f_k,G^k_{|\cdot|}) \in G^k_{\geq t}\big). \label{Sxsi1}
\end{align} 
Similarly, for $t<0$, 
\begin{align} 
   P\big( \mathscr{S}(f_k,G^k) \leq t\big) & \stackrel{\eqref{SXkeq}}{=} P\big(-\mathscr{S}_R(f_k,G^k_{|\cdot|}) \leq t \text{ and } \mathscr{S}_R(f_k,G^k_{|\cdot|}) \notin G^k_{\geq 0} \big) \nonumber \\ 
   &=  P\big( \mathscr{S}_R(f_k,G^k_{|\cdot|}) \in [-t,\infty) \backslash G^k_{\geq0} \big). \label{Sxsi2}
\end{align} 
Our main goal in this section is to prove the following. 
\begin{theorem}\label{Smainresult}
Let $f \in \mathcal{B}^n$ \eqref{absolutetails}. Let $V \subset \mathbb{R}^n$ be closed. Then there exists $0 \in G = G^1 \times \cdots \times G^n \subset \mathbb{R}^n$ 
Borel measurable, $G^k_{|\cdot|}$ closed, $k=1,...,n$, such that\footnote{$\mathscr{S}$ is at definition~\ref{Sdef} and $\mathcal{T}(f)$ at \eqref{T(f)}.}
\begin{equation}\label{Smainresulteq1}
\sup_{X \in \mathcal{T}(f)} P(X \in V) = P(\mathscr{S}(f,G) \in V).
\end{equation}
\end{theorem} 
Notice this is applicable to sharpest tail bounds (see section~\ref{tailboundintroduction}) since for $g: \mathbb{R}^n \rightarrow \mathbb{R}$ continuous, $g^{-1}[t,\infty)$ is closed $\forall t >0$. In particular, setting $V_t = g^{-1}[t,\infty)$ in \eqref{Smainresulteq1} yields
$$
\forall t \in \mathbb{R}, \sup_{X \in \mathcal{T}(f)} P(g(X) \geq t) = P(g(\mathscr{S}(f, G_t)) \geq t), \ \text{some } \{G_t\}_{t \in \mathbb{R}}
$$ 

To prove this theorem we will repeat the two step process used to prove theorem~\ref{S2maintheorem} (for $\mathscr{S}_2$) for $\mathscr{S}$. An outline of the two steps are below. 

Step $1$: let $(X^m)_{m \in \mathbb{N}} \subset \mathcal{T}(f)$ with
$$
P(X^m \in V) \xrightarrow[m\rightarrow \infty]{} \sup_{X \in \mathcal{T}(f)} P(X \in V).
$$
Extract from $(X^m)_m$ a subsequence $(X^{m_k})_k$ and $X_0 \in \mathcal{T}(f)$ such that $X^m \xrightarrow[]{a.e.} X_0$. Then show 
$$
P(X_0 \in V) = \lim_{k \rightarrow \infty} P(X^{m_k} \in V). 
$$
This shows there exists an $X_0 \in \mathcal{T}(f)$ which achieves the supremum. 

Step $2$: show $\forall X \in \mathcal{T}(f), \ \exists G_X$ satisfying definition~\ref{Sdef}, such that 
\begin{equation}\label{Sstep2eq}
P(X \in V) \leq P\big(\mathscr{S}(f,G_X) \in V \big). 
\end{equation}
In particular, 
$$
\sup_{X \in \mathcal{T}(f)} P(X \in V) = P(X_0 \in V) \leq P( \mathscr{S}(f,G_{X_0}) \in V). 
$$
For $\mathscr{S}_2$ in place of $\mathscr{S}$, at this point we were done, since ``$\mathscr{S}_2(f,G,c)$" was an element of $\mathcal{T}_2(f)$, thus the reverse inequality held as well. However, because 
$\mathscr{S}(f,G) \in \mathcal{N}_{rad} \supsetneq \mathcal{N}$, $\mathscr{S}(f,G)$ is not always an element of $\mathcal{T}(f)$. Fortunately this is not a problem, because to show 
$P\big(\mathscr{S}(f,G_{X_0}) \in V \big) \leq \sup_{X \in \mathcal{T}(f)} P(X \in V)$ we may find $X' \in \mathcal{T}(f)$ such that
$$
P(\mathscr{S}(f,G_{X_0}) \in V) = P(X' \in V).
$$
The existence of $X'$ is immediate, because $\mathscr{S}(f,G)$ satisfies the absolute tail bounds $f=(f_1,...,f_n)$ (theorem~\ref{Sproperties} (iv)), and has mutually ind.\ components (definition~\ref{Sdef}), 
so $\exists X' \in \mathcal{N}^n$ with the same distribution as $\mathscr{S}(f,G)$ by corollary~\ref{c2.1}.

Implementation of step $1$: 
\begin{lemma}\label{SX0lemma}
$\exists X_0 \in \mathcal{T}(f)$ such that 
$$
\sup_{X \in \mathcal{T}(f)} P(X \in V) =P(X_0 \in V).
$$
\end{lemma} 
\begin{proof} 
Let $(X^m)_{m \in \mathbb{N}} \subset \mathcal{T}(f)$ be such that 
$$
\lim_{m \rightarrow \infty} (X^m \in V) \xrightarrow[m \rightarrow \infty]{} \sup_{X \in \mathcal{T}(f)} P(X \in V).
$$
By lemma~\ref{conva.e.subseq}, the $1$st component of $(X^m)_m$, $(\pi_1 \circ X^m)_m \subset \mathcal{T}(f_1)$, has a subsequence which converges a.e.\ to some $X_1 \in \mathcal{T}(f_1)$. 
Successively repeating this for the other components leads to a standard diagonalization, from which one obtains a subsequence $(X^{m_l})_l \subset (X^m)_m$, such that
\begin{equation}\label{xsitilde}
X^{m_l} \xrightarrow[]{a.e.} (X_1,...,X_n) =: X_0 \in \mathcal{T}(f).
\end{equation}
The remainder of the proof, that $\lim_{k\rightarrow \infty} P(X^{m_k} \in V) = P(X_0 \in V)$, is precisely lemma~\ref{X0lemma}'s proof (see section~\ref{sectionS2mainresult}).
\end{proof} 

Implementation of step $2$: 
\begin{theorem}\label{SWLOG}
$\forall X \in \mathcal{T}(f), \ \exists G \ni 0$ Borel measurable with $G^k_{|\cdot|}$ closed for $k=1,...,n$, such that 
$$
P(X \in V) \leq P(\mathscr{S}(f,G)\in V).
$$
\end{theorem} 
\begin{proof} 
All r.v.'s being considered in this theorem are mutually ind.\ by construction (see section~\ref{preliminaries}). Introduce 
\begin{equation}\label{G1definitioneq}  
G^1 := \big\{t \in \mathbb{R} \ \big| \ \forall t' \in \big[-|t|,|t|\big], \ h_1(t') \leq h_1(t) \big\}
\end{equation} 
where $h_1$ is from definition~\ref{hkdef}. Note $0 \in G^1$. 

\textit{Claim}: $G^1$ is Borel measurable and $G^1_{|\cdot|}$ is closed. 

We will show measurability first. Define 
$$
H_1(t) = \sup_{t' \in [-|t|, |t|]} h_1(t'), \ \forall t \in \mathbb{R}.
$$ 
Note since $h_1 \geq 0$ that $H^1 \geq 0$. $H_1 \nearrow$ on $[0,\infty)$ and $H_1 \searrow$ on $(-\infty,0]$. Because monotone functions are Borel measurable, $H_1 \mathbbm{1}_{[0,\infty)}$ and 
$H_1 \mathbbm{1}_{(-\infty,0)}$ are measurable, where $\mathbbm{1}$ denotes the indicator function. Adding these functions gives $H_1$, implying $H_1$ is measurable. Clearly 
$$
t \in G^1 \iff h_1(t) = H_1(t).
$$
It follows $G^1 = (H_1-h_1)^{-1}(\{0\})$. So $G^1$ is measurable.

Now we address closedness. Suppose $\{t_j\}_{j \in \mathbb{N}} \subset G^1_{|\cdot|}$ such that $t_j \xrightarrow[j \rightarrow \infty]{} t_\infty$. We must show $t_\infty \in G^1_{|\cdot|}$. Recall 
$$
G^1_{|\cdot|} = \{|t| \ \big| \ t \in G^1\}.
$$
Considering the preimages of the $t_j$ in $G^1$, notice there exists a monotone sequence $\{t_{j_k}'\}_{k \in \mathbb{N}} \subset G^1$, $t_{j_k}' \rightarrow t_\infty' \in G^1$, such that  
$\forall k, \ |t'_{j_k}| = t_{j_k}$ and $|t_\infty'| = t_\infty$.

If $t_\infty = 0$ we are done, since $t_\infty \in G^1_{|\cdot|}$ trivially. Otherwise suppose $t \in (-|t'_\infty|, |t'_\infty|)$. Since $t_{j_k}$ converges to a value of larger magnitude than $t$, by definition of $G^1$,
$$
h_1(t) \leq h_1(t_{j_k}), \ \forall k >> 0
$$
By lemma~\ref{hklemma} (recall $t_{j_k}$ is monotone), $\lim_{k \rightarrow \infty} h_1(t_{j_k}) \leq h_1(t_\infty)$. It follows $h_1(t) \leq h_1(t_\infty')$. Since $t$ was arbitrary, $\forall t \in (-|t_\infty'|,|t_\infty'|), \ h_1(t) \leq h_1(t_\infty')$. 
Now, there are two possibilities: either $h_1(t_\infty') \geq h_1(-t_\infty')$ in which case $t_\infty' \in G^1$, or 
$$
h_1(-t_\infty) \geq h_1(t_\infty) \geq h_1(t), \ \forall t \in (-|t_\infty'|, |t_\infty'|)
$$
in which case $-t_\infty' \in G^1$. In either case $t_\infty = |t_{\infty}'| \in G^1_{|\cdot|}$. Therefore $G^1_{\cdot}$ is closed. 

The proof of theorem~\ref{SWLOG} centers around the claim
\begin{equation}\label{Smu1} 
P(X \in V) \leq P\big((\mathscr{S}(f_1,G^1),X_2,...,X_n) \in V \big). 
\end{equation} 
This is more difficult than \eqref{S2eq4}, because this time $\mathscr{S}(f_1,G^1)$ and $X_1$ have different types of domain organization: $X_1 \nearrow$ on $(0,1)$ whereas 
$|\mathscr{S}(f_1,G^1)| \nearrow $ on $(0,1)$ (theorem~\ref{Sproperties} (iii)). 

The workaround is to introduce $0 \in {G^1}' \subset \mathbb{R}$ measurable with ${G^1}'_{|\cdot|}$ closed, $f_1' \in \mathcal{B}$, and $\phi: \mathbb{R} \rightarrow \mathbb{R}$, such that the intermediate steps 
\begin{align*} 
P(X \in V) &\stackrel{\text{step A}}{\leq} P\big((\phi \circ X_1,X_2,...,X_n) \in V \big) \\ 
&\stackrel{\text{step B}}{=} P\big((\mathscr{S}(f_1',{G^1}'),X_2,...,X_n) \in V \big) \\ 
&\stackrel{\text{step C}}{\leq} P\big((\mathscr{S}(f_1,G^1),X_2,...,X_n) \in V \big)
\end{align*} 
are manageable. Informally, step A combines, $\forall t \geq 0, \ X_1$'s collective mass at $t$ and  $-t$, into either $t$ or $-t$, depending on which will lead to greater mass in $V$. 

Step B organizes the domain of $\phi \circ X_1$ by finding $f_1'$ and ${G^1}'$ such that $\phi \circ X_1$ and $\mathscr{S}(f_1',{G^1}')$ have the same distribution. 

Step C considers the integral expressions of the probabilities with Fubini's theorem and applies change of variables for Lebesgue integration to make use of the radially neat domains. 

These steps will be put in their own lemmas. With \eqref{Smu1} shown, we can inductively reapply \eqref{Smu1} the same way \eqref{S2eq4} was inducted. That is, 
\begin{align*}
   &P\big((\mathscr{S}(f_1,G^1),X_2,...,X_n) \in V \big) \\ 
   &\leq P\big(\mathscr{S}(f_1,G^1),\mathscr{S}(f_2,G^2),X_3,...,X_n) \in V \big) \\
   &\leq \cdots \\
   &\leq P\big((\mathscr{S}(f_1,G^1),...,\mathscr{S}_2(f_n,G^n)) \in V \big) \\ 
   &= P(\mathscr{S}(f,G) \in V). 
\end{align*} 
There is one nuance: $X_k \in \mathcal{T}(f_k), \ k =1,...,n$, but $\mathscr{S}(f_k,G^k) \notin \mathcal{T}(f_k)$, because $\mathscr{S}(f_k,G^k)$ only satifies the absolute tail $f_k$ (theorem~\ref{Sproperties} (iv)), 
not the neat domain condition for $\mathcal{T}(f)$ \eqref{T(f)}.

Fortunately we are saved, because, in the future, the proof of \eqref{Smu1} will never invoke that the r.v.'s $X_2,...,X_n$ being held fixed (while $X_1$ is worked on) are neat; only that $X_2,...,X_n$ are Borel 
measurable and mutually ind., which $\mathscr{S}(f_k,G^k)$ are as well (definition~\ref{Sdef} and theorem~\ref{Sproperties}). So this induction is valid. 

\begin{lemma}[Step A in theorem~\ref{SWLOG}]\label{SStepA}
This lemma proceeds from the assumptions in theorem~\ref{Smainresult}. Define
\begin{equation}\label{G1'def} 
   {G^1}' := \Big\{t \in \mathbb{R} \ \big| \ h_1(t) \geq h_1(-t) \Big\}\\
\end{equation}
where, as usual, $h_1(t) = P\{(t,X_2,...,X_n) \in V\}, \ \forall t \in \mathbb{R}$. Define
\begin{equation}\label{phi1} 
\phi_1(t) := 
\begin{cases} 
|t| & \text{if }|t| \in {G^1}' \\ 
-|t| & \text{otherwise} 
\end{cases} 
\end{equation} 
Then 
$$
P(X \in V) \leq P\big((\phi \circ X_1,X_2,...,X_n) \in V \big)
$$
\begin{proof} 
Notice, $\forall t \in \mathbb{R}$, $t$ or $-t \in {G^1}'$. Thus
\begin{equation}\label{G1'=[0,infty)}
({G^1}')_{|\cdot|} = [0,\infty).
\end{equation} 
${G^1}'$ is measurable because ${G^1}' = H_1^{-1}([0,\infty))$, where $H^1(t) = h_1(t) - h_1(-t), \ \forall t \in \mathbb{R}$, which is measurable. 

It is elementary to show $\phi_1$ is Borel measurable. Indeed, $\forall U \subset \mathbb{R}$ measurable, one may verify
\begin{align*} 
\phi_1^{-1}(U) =& \Big[\big( U_{\geq 0} \cap {G^1}'\big) \cup \Big(\big\{-x \ \big| \ x \in U \cap (-\infty,0)   \big\} \backslash {G^1}'\Big) \Big] \\
               &\cup \Big[\big(U_{\geq 0} \cap {G^1}'\big) \cup \Big(\big\{-x \ \big| \ x \in U \cap (-\infty,0) \big\}\backslash {G^1}'\Big) \Big]_{-}
\end{align*}
which is a finite combination of unions, intersections, compliments, and reflections of measurable sets (sections~\ref{abbreviationschapter4} notations)). By theorem~\ref{Fubini'stheorem} applied to the indicator function of $V$, integrating $P(X \in V)$ over the $1$st component, 
$$
P(X \in V) = \int_{\mathbb{R}} h_1 {dX_1}_\#(m) 
$$
where $X_1\#(m)$ is the pushforward \eqref{pushforwarddef} of the Lebesgue measure on $(0,1)$. Similarly, 
$$
P\big((\phi_1 \circ X_1,X_2,...,X_n) \in V \big) = \int_\mathbb{R} h_1 \ d(\phi_1 \circ X_1)_\#(m).
$$
By change of variables for Lebesgue integration, 
\begin{align*} 
&P(X \in V) = \int_{(0,1)} h_1 \circ X_1 \ dm\\
P\big((\phi_1 \circ X_1,&X_2,...,X_n) \in V \big) = \int_{(0,1)} h_1 \circ \phi_1 \circ X_1 \ dm.
\end{align*} 
By definition of $\phi_1$ and ${G^1}'$ \eqref{G1'def}, notice $\forall t \in \mathbb{R}$, $\phi_1$ sends $t$ to $t$ or $-t$, depending which makes $h_1$ larger (with the nuance if $h_1(t) = h_1(-t)$, then $\phi$ sends 
$t$ to $|t|$). That is, $\forall t \in \mathbb{R}, h_1(t) \leq h_1(\phi_1(t))$. Since $X_1:(0,1) \rightarrow \mathbb{R}$ it follows $\forall s \in (0,1), \ h_1(X_1(s)) \leq h_1 \circ \phi(X_1(s))$. The lemma follows 
by monotonicity of Lebesgue integration.  
\end{proof} 
\end{lemma} 
\begin{lemma}[Step B in theorem~\ref{SWLOG}] 
This lemma proceeds from the assumptions in theorem~\ref{SWLOG}. Define $\phi_1$ and ${G^1}'$ as in lemma~\ref{SStepA}. Define  
\begin{equation}\label{Sf1'} 
   f_1'(t) := P(|X_1| \geq t), \ \forall t \geq 0
\end{equation} 
Then
\begin{equation}\label{SstepBlemmaeq}
P\big((\phi_1 \circ X_1,X_2,...,X_n) \in V \big) = P\big((\mathscr{S}(f_1',{G^1}'),X_2,...,X_n) \in V \big).
\end{equation}
\end{lemma} 
\begin{proof}
We will show $\mathscr{S}(f_1',{G^1}')$ and $\phi_1 \circ X_1$ have the same distribution. Then \eqref{SstepBlemmaeq} follows with straightforward, albeit tedious rigor.
   
We will show 
\begin{enumerate}[label=(\roman*)]
   \item $ P\big(\phi \circ X_1 \geq t \big) = P\big(\mathscr{S}(f_1',{G^1}') \geq t \big), \ \forall t \geq 0 $
   \item $P\big(\phi \circ X_1 \leq t \big) = P\big(\mathscr{S}(f_1',{G^1}') \leq t \big), \ \forall t \leq 0 $
\end{enumerate} 
Recall the $\lambda$-$\pi$ theorem (section~\ref{preliminaries}). Sets of the form $(-\infty, -t_1], \ [t_2,\infty), \ \forall t_1,t_2 \geq 0$, form a $\pi$-system and generate the Borel $\sigma$-algebra of $\mathbb{R}$. Therefore, 
by the $\lambda$-$\pi$ theorem, proving (i) and (ii) will show $\phi \circ X_1$ and $\mathscr{S}(f_1,',{G^1}')$ have the same distribution.

(i): Following \eqref{Xtildedefn} and lemma~\ref{Xtilde}, let
$$
\widetilde{X}_1(s) = \sup {f_1'}^{-1}[1-s,1], \ \forall s \in(0,1).
$$
Recall $\mathscr{S}(f_1',{G^1}')(s) \geq 0 \iff \mathscr{S}_R(f_1,({G^1}')_{|\cdot|})(s) \in {G^1}'\cap [0,\infty)$ (\eqref{SSRgeq0} and \eqref{SXkeq}). It follows for $t \geq 0$, 
\begin{align} 
P\big(\mathscr{S}(f_1',{G^1}') \geq t \big) &\stackrel{\eqref{SXkeq}}{=} P\big(\mathscr{S}_R\big(f_1',({G^1}')_{|\cdot|}\big) \in ({G^1}')_{\geq t} \big) \nonumber \\ 
&\stackrel{\eqref{SSRgeq0}}{=} P\big(\sup \ \big(({G^1}')_{|\cdot|}\big)_{\leq \widetilde{X}_1(s)} \in ({G^1}')_{\geq t} \big) \nonumber \\  
&\stackrel{\eqref{G1'=[0,infty)}}{=} P\big(\sup \ [0,\infty)_{\leq \widetilde{X}_1(s)} \in ({G^1}')_{\geq t} \big) \nonumber \\
&= P(\widetilde{X}_1 \in ({G^1}')_{\geq t}). \label{Sgamma1} 
\end{align} 
On the other hand, for $t \geq 0$, 
\begin{align} 
P(\phi_1 \circ X_1 \geq t) &= P( X_1 \in \phi_1^{-1}[t,\infty)) \nonumber \\ 
&\stackrel{\eqref{phi1}}{=} P\Big(X_1 \in ({G^1}')_{\geq t} \cup \big(({G^1}')_{\geq t}\big)_-\Big) \nonumber \\ 
&\ = P\big(|X_1| \in ({G^1}')_{\geq t}\big). \label{Sgamma2}
\end{align} 
But $|X_1|$ and $\widetilde{X}_1$ both satisfy the absolute tail $f_1'$: 
\begin{align*} 
P(|X_1| \geq t) &\ \stackrel{\eqref{Sf1'}}{=} f_1'(t), \ \forall t \geq 0 \\  
P(\widetilde{X}_1 \geq t)&\stackrel{\text{lemma~\ref{Xtilde}}}{=}f_1'(t), \ \forall t \geq 0.
\end{align*} 
Notice $f_1'(0) = 1$, so $\widetilde{X}_1 \geq 0$. Therefore the equality $P(|X_1| \geq t) = P(\widetilde{X}_1 \geq t)$ not only holds $\forall t \geq 0$, but also $\forall t < 0$, both sides being $1$. That is, $|X_1|$ and $\widetilde{X}_1$ have the same reversed CDFs, hence the same distribution (section~\ref{preliminaries}). 
Therefore \eqref{Sgamma1} equals \eqref{Sgamma2}, which shows (i).

(ii): Similarly, for $t<0$, \eqref{Sxsi2} says 
\begin{align} 
P\big(\mathscr{S}(f_1',{G^1}')\leq t) &= P \Big(\mathscr{S}_R\big(f_1',({G^1}')_{|\cdot|}\big) \in [-t,\infty) \backslash G^k_{\geq 0} \Big) \nonumber \\ 
&\stackrel{\eqref{G1'=[0,infty)}}{=} P \Big(\mathscr{S}_R\big(f_1',[0,\infty)]\big) \in [-t,\infty) \backslash G^k_{\geq 0} \Big) \nonumber \\
&= P\Big(\widetilde{X}_1 \in [-t,\infty) \backslash ({G^1}')_{\geq 0} \Big). \label{Saleph1}
\end{align} 
where the last step is the same calculation as \eqref{Sgamma1}. On the other hand, suppose $t<0, \ s \in (0,1)$ and $\phi_1(X_1(s)) \leq t$. By \eqref{phi1} this happens iff $|X_1(s)| \notin {G^1}'$ and $|X_1(s)| \geq |t| = -t$. 
This is iff $|X_1(s)| \in [-t,\infty)\backslash {G^1}'$. Thus for $t<0$, 
\begin{equation}
P(\phi \circ X_1 \leq t) = P\Big(|X_1| \in [-t,\infty) \backslash {G^1}'_{\geq 0}\Big). \label{Saleph2}
\end{equation}

Recall in (i) we showed $|X_1|$ and $\widetilde{X}_1$ have the same distribution. So \eqref{Saleph1} equals \eqref{Saleph2}, which proves (ii) as desired. As mentioned, (i) and (ii) imply 
$\phi\circ X_1$ and $\mathscr{S}(f_1',{G^1}')$ have the same distribution.

Lastly it is a straightforward application of the $\lambda$-$\pi$ theorem (using the closed product sets in $\mathbb{R}^n$ as the $\pi$-system) and mutual ind.\ of the r.v.'s (see section~\ref{preliminaries}) to show $(\phi_1 \circ X_1,X_2,...,X_n)$ and $(\mathscr{S}(f_1',{G^1}'),X_2,...,X_n)$ have the same distribution.
\end{proof} 

\begin{lemma}[Step C in theorem~\ref{SWLOG}] 
This lemma proceeds from the assumptions in theorem~\ref{SWLOG}. Define ${G^1}'$ the same as \eqref{G1'def} and $f_1'$ the same as \eqref{Sf1'}. Define $G^1$ and $H^1$ the same as theorem~\ref{SWLOG} does. Specifically
\begin{equation}\label{SG1defagain}
   G^1 := \big\{t \in \mathbb{R} \ \big| \ \forall t' \in \big[-|t|,|t|\big], \ h_1(t') \leq h_1(t) \big\} 
\end{equation} 
Then
\begin{equation}\label{StepCagaineq}
   P\big((\mathscr{S}(f_1',{G^1}'),X_2,...,X_n) \in V \big) = P\big((\mathscr{S}(f_1,G^1),X_2,...,X_n) \in V \big).
\end{equation}
\end{lemma} 
\begin{proof} 
Recall $h_1(t) = P((t,X_2,...,X_n) \in V), \ \forall t \in \mathbb{R}$. By Fubini's theorem (theorem~\ref{Fubini'stheorem}, Lebesgue integrating over the $1$st component and using change of variables yields
\begin{align} 
   P\big((\mathscr{S}(f_1',G^1),X_2,...,X_n) \in V \big) &= \int_\mathbb{R} \ \ h_1 \ d \mathscr{S}(f_1',{G^1}')_\#(m) \\ 
                                                        &= \int_{(0,1)} h_1 \circ \mathscr{S}(f_1',{G^1}') \ dm \label{SK1} \\
   P\big((\mathscr{S}(f_1,G^1),X_2,...,X_n) \in V \big) &= \int_{\mathbb{R}} \ \ h_1 \ d\mathscr{S}(f_1,G^1)_\#(m) \\ 
                                                        &= \int_{(0,1)} h_1 \circ \mathscr{S}(f_1,G^1) \ dm \label{SK2}
\end{align}
where $m$ is the Lebesgue measure on $(0,1)$ with Borel $\sigma$-algebra. To show \eqref{SK1} $\leq$ \eqref{SK2}, by monotonicity of integration it suffices to show 
\begin{equation} 
h_1\big(\mathscr{S}(f_1',{G^1}')(s) \big) \leq h_1\big(\mathscr{S}(f_1,G^1)(s) \big), \ \forall s \in (0,1) \label{SK3}.
\end{equation} 
Notice $f_1'(t) \stackrel{\eqref{Sf1'}}{=} P(|X_1| \geq t) \stackrel{X_1 \in \mathcal{T}(f_1)}{\leq} f_1(t), \ \forall t \geq 0$ ($X_1 \in \mathcal{T}(f_1)$ by hypothesis at the start of theorem~\ref{SWLOG}). It is easy to check
$$
\sup {f_1'}^{-1} [1-s,1] \leq \sup f_1^{-1} [1-s, 1 ], \ \forall s\in(0,1),
$$
which is left to the reader. Now, 
$$
\big|\mathscr{S}(f_1',{G_1}')\big| \stackrel{\eqref{|S|Xkeq}}{=} \mathscr{S}_R \big(f_1',({G_1}')_{|\cdot|}\big) \stackrel{\text{theorem~\ref{SRproperties}}}{\in} \mathcal{T}_R(f_1') \subset \mathcal{N}.
$$
We can now apply lemma~\ref{Xtilde} (iii) to get 
$$
\big|\mathscr{S}(f_1',{G_1}')(s)\big| \leq \sup {f_1'}^{-1}[1-s,1] \leq \sup f_1^{-1} [1-s, 1 ], \ \forall s \in (0,1). 
$$
In particular, $\mathscr{S}(f_1',{G_1}')(s) \in \big[-\sup f_1^{-1} [1-s, 1 ], \ \sup f_1^{-1} [1-s, 1 ]\big]$. Therefore to show \eqref{SK3}, it is sufficient to show, for a given $s \in (0,1)$, that 
\begin{equation}\label{SK4}
\sup_{|t| \leq \sup f_1^{-1}[1-s,1]} h_1(t) = h_1 \big(\mathscr{S}(f_1,G^1)(s) \big).
\end{equation} 
Before we prove \eqref{SK4}, we will need the following. 

Claim: $\forall t \geq 0, \ \exists t_0 \in [-t,t]$ such that 
$$
h_1(t_0) = \sup_{t' \in [-t,t]} h_1(t').
$$
To show this, suppose $(t_l)_{l \in \mathbb{N}} \subset [-t,t]$ with $h_1(t_l) \rightarrow \sup_{t' \in [-t,t]} h_1(t')$. By compactness $(t_l)$ has a convergent subsequence $(t_{l_k})$ with $t_{l_k} \rightarrow t_0$, some $t_0 \in [-t,t]$. 
We now get 
$$
h_1(t_0) \stackrel{\text{lemma~\ref{hklemma}}}{\geq} \lim_{k \rightarrow \infty} h_1(t_{l_k}) = \sup_{t' \in [-t,t]} h_1(t') \geq h_1(t_0)
$$
as desired. 

\eqref{SK4} is proved below.

Fix $s \in (0,1)$. By the claim, $\exists t_0 \in [-\sup f_1^{-1}[1-s,1], \ \sup f_1^{-1}[1-s,1]]$ such that 
\begin{equation}\label{St0eq}
h_1(t_0) = \sup_{|t| \leq \sup f_1^{-1}[1-s,1]} h_1(t).
\end{equation} 
Notice that
$$
|\mathscr{S}(f_1,G^1)(s)| \stackrel{\eqref{|S|Xkeq}}{=} \mathscr{S}_R(f_1,{G^1}_{|\cdot|})(s) \stackrel{\text{lemma~\ref{Xtilde}}}{\leq} \sup f^{-1}[1-s,1]
$$
since $\mathscr{S}_R(f_1,{G^1}_{|\cdot|}) \in \mathcal{T}_R(f_1)$ (theorem~\ref{SRproperties}). In particular, notice
$$
\mathscr{S}(f_1,G^1)(s) \in \big[ -\sup f^{-1}[1-s,1], \ \sup f^{-1}[1-s,1] \big].
$$ 
We will show that $h_1(t_0) = h_1(\mathscr{S}(f_1,G^1)(s))$ by considering two cases, proving \eqref{SK4}, thereby proving the lemma.

Case $|t_0| \leq |\mathscr{S}(f_1,G^1)(s)|$:
If $\mathscr{S}(f_1,G^1)(s) = 0$, there is nothing to show. Otherwise because $\mathscr{S}(f_1,G^1)(s) \in G^1$ (theorem~\ref{Sproperties} (i)), by definition of $G^1$ \eqref{SG1defagain}, one immediately has 
$h_1(t_0) \leq h_1\big(\mathscr{S}(f_1,G^1)(s)\big)$. Equality holds by \eqref{St0eq} and since $|\mathscr{S}(f_1,G^1)(s)| \leq \sup f_1^{-1}[1-s,1]$.

Case $|t_0| > |\mathscr{S}(f_1,G^1)(s)|$:

By \eqref{|S|Xkeq} and \eqref{SSRgeq0}, 
\begin{equation*}
|\mathscr{S}(f_1,G^1)(s)| = \sup \ \big(G^1_{|\cdot|}\big)_{\leq \sup f_1^{-1} \ [1-s,1]}.
\end{equation*}
In particular, this implies $|t_0| \notin \big(G^1_{|\cdot|}\big)_{\leq \sup f_1^{-1} \ [1-s,1]}$. But by construction \eqref{St0eq}, $t_0 \in [-\sup f_1^{-1}[1-s,1], \sup f_1^{-1}[1-s,1]]$, implying 
$|t_0| \notin G^1_{|\cdot|} = \big\{|t| \ \big| \ t \in G^1\big\}$. Thus $t_0 \notin G^1$. But by definition of $G^1$ \eqref{G1definitioneq}, this can only be if $t_0$ not a maximizing $t$ of $h_1(t)$ on $[-|t_0|,|t_0|]$. This contradicts \eqref{St0eq}. 

This proves \eqref{SK4}, so we are done.
\end{proof} 
This completes the proof of theorem~\ref{SWLOG}.
\end{proof} 
This completes step $1$, which was lemma~\ref{SX0lemma}, and step $2$ (recall \eqref{Sstep2eq}), which was theorem~\ref{SWLOG}, in proving theorem~\ref{Smainresult}. So we have shown theorem~\ref{Smainresult}.

For $\mathscr{S}_2$, because $\mathscr{S}_2$ always maps into $G$, we had theorem~\ref{S2finiteV}, which expressed WLOG $G$ is approximately in $V$. 
A similar result holds for $\mathscr{S}$; we will need the following.
\begin{lemma}\label{psilemma}
Let $V \subset \mathbb{R}^n$ be closed. For $G^k \subset \mathbb{R}$ denote $H^k = \big(G^k \cap \overline{\pi_k(V)} \big) \cup \{0\}$ and define 
\begin{equation}\label{defpsi} 
\psi(G^k) = H^k \cup \Big(\Big[\big(\overline{H^k_{|\cdot|}} \backslash H^k_{|\cdot|}\big) \cup \big(\overline{H^k_{|\cdot|}} \backslash H^k_{|\cdot|} \big)_- \Big] \cap \overline{\pi_k(V)}\Big).
\end{equation} 
Then $\psi(G^k)_{|\cdot|} = \overline{H^k_{|\cdot|}}$.
\end{lemma} 
\begin{proof} 
Recall for $A \subset \mathbb{R}, \ A_{|\cdot|} = \{|x| \ | \ x \in A \}$ and $A_- = \{-x \ | \ x \in A\}$. So $(A \cup B)_{|\cdot|} = A_{|\cdot|} \cup B_{|\cdot|}$. Also that closure and $|\cdot|$ commute, hence 
$\overline{A_{|\cdot|}}$ is unambiguous. By elementary operations
\begin{align} 
\psi(G^k)_{|\cdot|} &=H^k_{|\cdot|} \cup \Big[\big(\overline{ H^k_{|\cdot|} } \backslash H^k_{|\cdot|} \big)\cap \overline{\pi_k(V)} \Big]_{|\cdot|} \cup\ \Big[\big(\overline{H^k_{|\cdot|}} \backslash H^k_{|\cdot|}\big)_- \cap \overline{\pi_k(V)} \Big]_{|\cdot|} \nonumber \\ 
                    &= H^k_{|\cdot|} \cup \Big[\big(\overline{H^k_{|\cdot|}} \backslash H^k_{|\cdot|}\big) \cap \overline{\pi_k(V)} \Big] \cup \Big[\big(\overline{H^k_{|\cdot|}} \backslash H^k_{|\cdot|} \big)  \cap \overline{\pi_k(V)}_-\Big] \nonumber \\ 
                    &= H^k_{|\cdot|} \cup \Big[\big(\overline{H^k_{|\cdot|}} \backslash H^k_{|\cdot|} \big) \cap \big(\overline{\pi_k(V)} \cup \overline{\pi_k(V)}_- \big)\Big] \nonumber \\ 
                    &= \overline{H^k_{|\cdot|}} \cap \Big(H^k_{|\cdot|} \cup \big( \overline{\pi_k(V)} \cup \overline{\pi_k(V)}_-\big) \Big) \nonumber \\ 
                    &=  H^k_{|\cdot|} \cup \Big(\overline{H^k_{|\cdot|}} \cap \big( \overline{\pi_k(V)} \cup \overline{\pi_k(V)}_-\big)\Big). \label{terminatoreq} 
\end{align} 
Because $H^k \subset \overline{\pi_k(V)} \cup \{0\}$ it follows $H^k_{|\cdot|} \subset \overline{\pi_k(V)} \cup \overline{\pi_k(V)}_- \cup \{0\}$. In particular
$\overline{H^k_{|\cdot|}} \subset \overline{\pi_k(V)} \cup \overline{\pi_k(V)}_- \cup \{0\}$. Thus $\overline{H^k_{|\cdot|}} \cap \big( \overline{\pi_k(V)} \cup \overline{\pi_k(V)}_-\big)$ 
equals $\overline{H^k_{|\cdot|}}$ or $\overline{H^k_{|\cdot|}} \backslash \{0\}$. In either case, because $0 \in H^k_{|\cdot|}$, it immediately follows \eqref{terminatoreq} equals $\overline{H^k_{|\cdot|}}$. 
\end{proof} 

\begin{theorem}\label{SVfinite}
There exists $G \ni 0$ measurable, $G^k_{|\cdot|}$ closed and $G^k \subset \overline{\pi_k(V)} \cup \{0\}, \ k=1,...,n$, which satisfies theorem~\ref{Smainresult}.
\end{theorem} 
\begin{proof} 
Let $G \ni 0$ with $G^k$ Borel measurable and $G_{|\cdot|}^k$ closed for $k=1,...,n$. Consider
$$
G' = \psi(G^1) \times \cdots \times \psi(G^n)
$$
using \eqref{defpsi}. Because $G^k$ are Borel measurable, and the operations in \eqref{defpsi} preserve measurability, $G'$ is measurable. Furthermore $\psi(G^k) \ni 0$ 
because $0 \in H^k$, $\psi(G^k)_{|\cdot|}$ is closed by lemma~\ref{psilemma}, and $\psi(G^k) \subset \overline{\pi_k(V)} \cup \{0\}$. We claim  
\begin{equation}\label{SVfinitemaineq}
P(\mathscr{S}(f,G) \in V) \leq P(\mathscr{S}(f,G') \in V). 
\end{equation} 
We will finish the proof assuming this claim and then prove this claim. Suppose $G$ was a maximizing solution to \eqref{Smainresulteq1}. Then, by theorem~\ref{Sproperties} (iv) for $\mathscr{S}(f,G')$, LHS $\geq$ RHS in \eqref{SVfinitemaineq}. 
Thus $G'$ would also be a maximizing solution to theorem~\ref{Smainresult}, which would prove theorem~\ref{SVfinite}.  

Let us return to proving \eqref{SVfinitemaineq}. It is reducible to the following.
\begin{equation}\label{SVgamma2} 
   P\big(\mathscr{S}(f,G)\in V \big) \leq P\Big(\big(\mathscr{S}(f,\psi(G^1) \times G^2 \times \cdots \times G^n\big) \in V \Big).
\end{equation} 
Indeed, because the $1$st component and $G$ was arbitrary, reapplying \eqref{SVgamma2} gives
$$ 
P\Big(\mathscr{S}(f,\psi(G^1) \times G^2 \times \cdots \times G^n\big) \in V \Big) \leq P\Big(\mathscr{S}(f,\psi(G^1) \times \psi(G^2) \times G^3 \times \cdots \times G^n\big) \in V \Big) 
$$
Successively repeating this $n-2$ more times implies \eqref{SVfinitemaineq}. Below we prove \eqref{SVgamma2}.  

Still using $h_1(t) = P\big((t,X_2,...,X_n) \in V \big), \ \forall t \in \mathbb{R}$ (definition~\ref{hkdef}), because $X_k$ are mutually ind., 
using theorem~\ref{Fubini'stheorem} applied to the indicator function of $V$ and Lebesgue integrating over the $1$st component yields
\begin{align*} 
   P(\mathscr{S}(f,G) \in V) &= \int_{\mathbb{R}} h_1 \ d \mathscr{S}(f_1,G^1)_\# (m) \\ 
   P\Big(\big(\mathscr{S}(f,\psi(G^1) \times G^2 \times \cdots \times G^n\big) \in V \Big) &= \int_\mathbb{R} h_1 \ d \mathscr{S}(f_1,\psi(G^1))_\#(m)
\end{align*}
where $m$ is the Lebesgue measure on $(0,1)$. By change of variables, 
\begin{align*} 
   P(\mathscr{S}(f,G) \in V) &= \int_{(0,1)} h_1 \circ \mathscr{S}(f_1,G^1) \ dm \\
   P\Big(\big(\mathscr{S}(f,\psi(G^1) \times G^2 \times \cdots \times G^n\big) \in V \Big) &= \int_{(0,1)} h_1 \circ \mathscr{S}(f_1,\psi(G^1)) \ dm
\end{align*} 
By monotonicity it suffices to show 
\begin{equation} \label{SVgamma3} 
h\big(\mathscr{S}_1(f_1,G^1)(s) \big) \leq h_1 \big( \mathscr{S}_1(f_1,\psi(G^1))(s) \big), \ \forall s \in (0,1). 
\end{equation} 
Fix $s \in (0,1)$. 

Case $\mathscr{S}(f_1,G^1)(s) \notin \pi_1(V)$: 

It immediately follows
\begin{equation*} 
h_1(\mathscr{S}(f_1,G^1)(s)) = P((\mathscr{S}(f_1,G^1)(s), X_2,...,X_n) \in V) = 0.
\end{equation*} 
\eqref{SVgamma3} holds since $h_1 \geq 0$. 

Case $\mathscr{S}(f_1,G^1)(s) \in \pi_1(V)$: 

It is sufficient to prove 
\begin{equation}\label{SVgamma4} 
\mathscr{S}(f_1,G^1)(s) = \mathscr{S}(f_1,\psi(G^1))(s). 
\end{equation} 
First we will show 
\begin{equation}\label{SVgamma5} 
|\mathscr{S}(f_1,G^1)(s)| = | \mathscr{S}(f_1,\psi(G^1))(s)|. 
\end{equation} 
Recall $H^k = (G^k \cap \overline{\pi(V)}) \cup \{0\}$ in \eqref{defpsi}. By theorem~\ref{Sproperties} (i), 
\begin{equation}\label{SVepsilon1}
\mathscr{S}(f_1,G^1)(s) \in G^1 \cap \pi_1(V) \subset H^1.
\end{equation} 
By \eqref{SSRgeq0} and \eqref{|S|Xkeq}, $|\mathscr{S}(f_1,G^1)(s)| = \mathscr{S}_R(f_1,G^1_{|\cdot|})(s) = \sup G^1_{|\cdot|} \cap (-\infty, \widetilde{X}_1(s)]$ 
which is an element of $H^1_{|\cdot|}$ by the above, hence 
$$
|\mathscr{S}(f_1,G^1)(s)| = \sup \ G^1_{|\cdot|} \cap H^1_{|\cdot|} \cap (-\infty, \widetilde{X}_1(s)]. 
$$
Now $H^k \subset G^k \implies H^k_{|\cdot|} \subset G^k_{|\cdot|}$, thus 
\begin{align} 
|\mathscr{S}(f_1,G^1)(s)| = \mathscr{S}_R(f_1,G^1_{|\cdot|})(s) &= \sup H^1_{|\cdot|} \cap (-\infty, \widetilde{X}_1(s)] \nonumber \\ 
&= \sup \ \overline{H^1_{|\cdot|}} \cap (-\infty, \widetilde{X}_1(s)] \nonumber \\ 
&\stackrel{\text{lemma~\ref{psilemma}}}{=} \sup \ \psi(G^1)_{|\cdot|} \cap (-\infty, \widetilde{X}_1(s)] \nonumber \\ 
&\stackrel{\eqref{SSRgeq0}}{=} \mathscr{S}_R(f_1, \psi(G^1)_{|\cdot|})(s) \label{SVfinitepsieq} \\ & 
\stackrel{\eqref{|S|Xkeq}}{=} |\mathscr{S}(f_1,\psi(G^1))(s)|. \nonumber
\end{align} 
Therefore \eqref{SVgamma5} holds. We will now prove $\mathscr{S}(f_1,G^1)(s)$ and $\mathscr{S}(f_1,\psi(G^1))(s)$ share the same sign. 
By \eqref{SXkeq} and since $\mathscr{S}_R(f_1,G^1_{|\cdot|})(s) \geq 0$ \eqref{SSRgeq0}, 
\begin{equation}\label{SVepsilon5} 
\mathscr{S}(f_1,G^1)(s) \geq 0 \iff \mathscr{S}_R(f_1,G^1_{|\cdot|})(s) \in G^1.
\end{equation} 
Similarly 
\begin{equation}\label{SVepsilon4}
\mathscr{S}(f_1,\psi(G^1))(s) \geq 0 \iff \mathscr{S}_R(f_1,\psi(G^1)_{|\cdot|})(s) \in \psi(G^1). 
\end{equation} 
Now, $\mathscr{S}_R(f_1,G^1_{|\cdot|})(s) \in G^1 \implies \mathscr{S}_R(f_1,G^1_{|\cdot|})(s) \stackrel{\eqref{SXkeq}}{=} \mathscr{S}(f_1,G^1)(s) \in \pi_1(V)$ (case hypothesis), hence that $\mathscr{S}_R(f_1,G^1_{|\cdot|})(s) \in (G^1 \cap \overline{\pi_1(V)}) \cup \{0\} = H^1$. Conversely $\mathscr{S}_R(f_1,G^1_{|\cdot|})(s) \in H^1 \implies \mathscr{S}_R(f_1,G^1_{|\cdot|})(s) \in G^1$ since $0 \in G^1$ by hypothesis. \eqref{SVepsilon5} becomes
\begin{equation}\label{SVepsilon2} 
\mathscr{S}(f_1,G^1)(s) \geq 0 \iff \mathscr{S}_R(f_1,G^1_{|\cdot|})(s) \in H^1 \stackrel{\eqref{defpsi}}{\subset} \psi(G^1).
\end{equation} 
Now, $\mathscr{S}_R(f_1,G^1_{|\cdot|}) = |\mathscr{S}(f_1,G^1)(s)| \in H^1_{|\cdot|}$ by \eqref{SVepsilon1} which implies 
$$
\mathscr{S}_R(f_1,G^1_{|\cdot|}) \notin \overline{H^1_{|\cdot|}} \backslash H^1_{|\cdot|},
$$
whereas because $\big(\overline{H^1_{|\cdot|}}\backslash H^1_{|\cdot|} \big)_- \subset (-\infty, 0)$ and $\mathscr{S}_R(f_1,G^1_{|\cdot|})(s) \geq 0$, 
$$
\mathscr{S}_R(f_1,G^1_{|\cdot|})(s) \notin \big(\overline{H^1_{|\cdot|}}\backslash H^1_{|\cdot|} \big)_-.
$$
By \eqref{defpsi} we conclude if $\mathscr{S}_R(f_1,G^1_{|\cdot|})(s) \in \psi(G^1)$ then $\mathscr{S}_R(f_1,G^1_{|\cdot|})(s) \in H^1$. Thus by \eqref{SVepsilon2} 
\begin{equation}\label{SVepsilon3} 
\mathscr{S}(f_1,G^1)(s) \geq 0 \iff \mathscr{S}_R(f_1,G^1_{|\cdot|})(s) \in \psi(G^1). 
\end{equation} 
Since $\mathscr{S}_R(f_1,G^1_{|\cdot|})(s) = \mathscr{S}_R(f_1,\psi(G^1)_{|\cdot|})(s)$  by \eqref{SVfinitepsieq}, \eqref{SVepsilon3} and \eqref{SVepsilon4} imply \\ 
$\mathscr{S}(f_1,G^1)(s)$ and $\mathscr{S}(f_1, \psi(G^1))(s)$ have the same sign. This shows \eqref{SVgamma4} as desired.
\end{proof}
  
\begin{theorem}\label{improvedfiniteV}
Let $V \subset \mathbb{R}^n$ be finite. There exists a $G$ which satisfies theorem~\ref{Smainresult} such that
\begin{equation}\label{improvedfiniteVeq1}
G = \pi_1(G \cap V) \cup \{0\} \times \cdots \times \pi_n(G \cap V) \cup \{0\}.
\end{equation} 
\end{theorem} 
\begin{proof} 
Because the proof of this theorem is very similar to the proofs of theorem~\ref{SVfinite} and theorem~\ref{S2finiteV}, we will only give an outline. 

Let $G$ be a solution to theorem~\ref{SVfinite}. Then $G^k \subset \overline{\pi_k(V)} \cup \{0\} = \pi_k(V) \cup \{0\}, \ k=1,...,n$ ($V$ is finite). Consider (recall \eqref{defgridfunction})
$$
\#_0(G \cap V) = G' = {G^1}' \times \cdots \times {G^n}' 
$$
Notice $\#_0(G' \cap V) = G'$, i.e.\ $G'$ satisfies \eqref{improvedfiniteVeq1}. To show $G'$ is also a solution to theorem~\ref{Smainresult} (hence theorem~\ref{improvedfiniteV}), similar to theorem~\ref{SVfinite} it is sufficient to prove the claim
$$
P(\mathscr{S}(f,G) \in V) \leq P(\mathscr{S}(f,G') \in V).
$$
Notice because $G$ and $G \cap V$ are finite that $G \ni 0$ is closed and $G' = \overline{\#_0(G \cap V)}$. These were the assumptions for $G$ and $G'$ in theorem~\ref{S2finiteV}. The main subject of theorem~\ref{S2finiteV} was proving this same claim for $\mathscr{S}_2$. Therefore, we may defer to theorem~\ref{S2finiteV}, replacing any occurence of $\mathscr{S}_2$ with $\mathscr{S}$, until one reaches the point where by monotonicity of integration 
it suffices to show
\begin{equation}\label{h1lastequation} 
h_1(\mathscr{S}(f_1,G^1))(s) \leq h_1(\mathscr{S}(f_1,{G^1}')(s), \ \forall s \in (0,1).
\end{equation} 
There are two cases: $\mathscr{S}(f_1,G^1)(s) \notin \pi_1(G \cap V)$ and $\mathscr{S}(f_1,G^1)(s) \in \pi_1(G \cap V)$. The former case
has $h_1(\mathscr{S}(f_1,G^1)(s)) = 0$ similar to theorem~\ref{SVfinite}; \eqref{h1lastequation} holds. The latter case follows by proving 
$$
\mathscr{S}(f_1,G^1)(s) = \mathscr{S}(f_1,{G^1}')(s). 
$$
Similar to theorem~\ref{SVfinite}, this is done by showing their magnitudes are equal and then that their signs are as well. 
\end{proof} 

\section{Solving example $1$}\label{solvingeg.1section}
Either theorem~\ref{SVfinite} or theorem~\ref{improvedfiniteV} solves figure~\ref{eg.1fig}, but the latter has fewer $G$ to check. Recall figure~\ref{eg.1fig} had 
$V = \{(a_1,b_1),(a_2,b_2)\}, \ 0 \leq a_1 < a_2, \ 0 \leq b_2 < b_1$. There are only four $G$ which satisfy \eqref{improvedfiniteVeq1}. 
\begin{enumerate}[label=(\roman*)]
   \item $G = \{0\}$
   \item $G = \{0, a_1\} \times \{0, b_1\}$
   \item $G = \{0, a_2\} \times \{0, b_2\}$
   \item $G = \{0, a_1,a_2\} \times \{0, b_1,b_2\}$
\end{enumerate} 
We will compute $P(\mathscr{S}(f,G) \in V)$ for these $G$ and take the max. By theorem~\ref{improvedfiniteV}, this returns $\sup P(X \in V)$. 

For (i), $P(\mathscr{S}(f,G) \in V) = 0$ because $\mathscr{S}$ always maps into $G$ (theorem~\ref{Sproperties} (i)), but $0 \notin V$. For (ii), because $\mathscr{S}(f_1,\{0,a_1\})$ is shifted outwards on 
$\{0,a_1\}$ as far as $f_1$ allows (more precisely in the sense of \eqref{SRCDF2} and because $\mathscr{S}(f_1,\{0,a_1\})$ maps into $\{0,a_1\}$) it follows $P\big(\mathscr{S}(f_1,\{0,a_1\}) = a_1\big) = f_1(a_1)$. 

Similarly, $P\big(\mathscr{S}(f_2,\{0,b_1\}) = b_1\big) = f_2(b_1)$. Because the r.v.'s are ind., it follows their product has
$$
P\big(\mathscr{S}(f, \{0,a_1\} \times \{0,b_1\}) \in V\big) = f_1(a_1) f_2(b_1). 
$$
Similarly, for (iii) one has $P(\mathscr{S}(f, \{0,a_2\} \times \{0,b_2\}) \in V) = f_1(a_2)f_2(b_2)$. 

For (iv), by \eqref{SRCDF2} 
$$
P\big(|\mathscr{S}(f_1, \{0,a_1,a_2\})| \geq t\big) = f_1\big(\inf \ \{0,a_1,a_2\} \cap [t,\infty)\big), \ \forall t \geq 0.
$$
Because $\mathscr{S}(f_1,\{0,a_1,a_2\})(0,1) \subset \{0,a_1,a_2\} \subset [0, \infty)$ (theorem~\ref{Sproperties} (i)) and $f_1(0) = 1$, this implies 
$$
P\big(\mathscr{S}(f_1,\{0,a_1,a_2\}) \geq t\big) = 
\begin{cases} 
1 & \text{if } t = 0 \\ 
f_1(a_1) & \text{if } 0< t \leq a_1 \\ 
f_1(a_2) & \text{if } a_1< t \leq a_2 \\ 
0            &\text{if } a_2 < t
\end{cases} 
$$
which clearly implies 
\begin{align*} 
P\big(\mathscr{S}(f_1,\{0,a_1,a_2\}) = a_2\big) &= f_1(a_2) \\
P\big(\mathscr{S}(f_1,\{0,a_1,a_2\}) = a_1\big) &= f_1(a_1) - f_1(a_2).
\end{align*} 
Similarly $\mathscr{S}(f_2,\{0,b_1,b_2\})$ has $f_2(b_1)$ mass at $b_1$ (recall that $b_1 > b_2$) and $f_2(b_2)-f_2(b_1)$ at $b_2$. Because the r.v.'s are ind., their product
$\mathscr{S}\big(f, \{0,a_1,a_2\} \times \{0,b_1,b_2\}\big)$ has $(f_1(a_1)-f_1(a_2)) f_2(b_1)$ mass at $(a_1,b_1)$ and $f_1(a_2)(f_2(b_2) - f_2(b_1))$ mass at $(a_2,b_2)$. Adding, 
$$
P\big(\mathscr{S}(f,\{0,a_1,a_2\} \times \{0,b_1,b_2\}) \in V\big) = f_1(a_1)f_2(b_1) + f_1(a_2)f_2(b_2) - 2f_1(a_2)f_2(b_1).
$$
Taking the maximum of these four quantities yields
\begin{equation*}  
\max\Big\{f_1(a_1)f_2(b_1),\ f_1(a_2) f_2(b_2), \ f_1(a_1)f_2(b_1) + f_1(a_2)f_2(b_2)-2f_1(a_2)f_2(b_1) \Big\}.
\end{equation*}
as the solution to figure~\ref{eg.1fig}. This method readily extends to $V \subset \mathbb{R}^n$ finite (the computation is similar)
however the number of $G$ to check can be as large as the number of subsets of $V$.

\printindex
\end{document}